\documentclass[preprint]{elsarticle}

\usepackage[a4paper,top=2cm,bottom=2cm,left=3cm,right=3cm,marginparwidth=1.75cm]{geometry} 
\usepackage[OT1]{fontenc} 
\usepackage[english]{babel} 
\usepackage[utf8]{inputenc} 

\usepackage{amsmath,amssymb,amsfonts,mathrsfs,tikz-cd,bbold,leftindex,dsfont,enumerate, tocloft} 
\usepackage{hyperref}
\usepackage[amsmath,thmmarks,amsthm]{ntheorem} 
\usepackage{graphicx,xcolor} 
\usepackage{soul} 
\usepackage{pdfpages} 
\usepackage[all]{xy}
\usepackage{verbatim}
\usepackage{stmaryrd} 
\usepackage{sectsty}
\usepackage{titlesec}

\newcommand{\rom}[1]{\mathrm{#1}} % roman                                       
\newcommand{\ol}[1]{\overline{#1}}% underlined                               
\newcommand{\clg}[1]{\mathcal{#1}} % caligraphic                                
\newcommand{\scr}[1]{\mathscr{#1}} % script                                     
\newcommand{\frk}[1]{\mathfrak{#1}} % fraktur 

\newcommand{\A}{\mathbb{A}}

\newcommand{\C}{\mathds{C}}

\newcommand{\G}{\mathbb{G}}

\newcommand{\N}{\mathds{N}}
\newcommand{\Q}{\mathds{Q}}
\newcommand{\R}{\mathbb{R}}

\newcommand{\Z}{\mathds{Z}}

\newcommand{\PR}{\mathds{P}}
\newcommand{\UN}{\mathds{1}}

\newcommand{\FF}{\mathds{F}}

\newcommand{\grm}{\mathrm{gr}}

\newcommand{\pre}{\mathrm{pre}}

\newcommand{\cc}{\mathrm{cc}}

\newcommand{\pr}{\mathrm{pr}}

\newcommand{\ses}{\mathrm{ss}}

\newcommand{\cont}{\mathrm{cont}}

\newcommand{\stb}{\mathrm{st}}
\newcommand{\infi}{\mathrm{inf}}
\newcommand{\rig}{\mathrm{rig}}

\newcommand{\Zar}{\mathrm{Zar}}

\newcommand{\fppf}{\mathrm{fppf}}
\newcommand{\et}{\mathrm{\acute{e}t}}

\newcommand{\dr}{\mathrm{dR}}
\newcommand{\cris}{\mathrm{cris}}

\newcommand{\gp}{\mathrm{gp}}
\newcommand{\PSh}{\mathrm{PSh}}
\newcommand{\Sh}{\mathrm{Shv}}

\newcommand{\id}{\mathrm{id}}

\newcommand{\eff}{\mathrm{eff}}

\newcommand{\can}{\mathrm{can}}
\newcommand{\con}{\mathrm{conj}}

\newcommand{\Gammabf}{\mathbf{\Gamma}}
\newcommand{\RDA}{\mathbf{RigDA}}
\newcommand{\DA}{\mathbf{DA}}

\newcommand{\SH}{\mathbf{SH}}

\newcommand{\Sm}{\mathbf{Sm}}
\newcommand{\SmR}{\mathbf{SmRig}}

\newcommand{\Sch}{\mathbf{Sch}}

\newcommand{\Ht}{\mathbf{Ho}}
\newcommand{\Mod}{\mathbf{Mod}}
\newcommand{\De}{\mathbf{D}}
\newcommand{\Var}{\mathbf{Var}}
\newcommand{\Set}{\mathbf{Set}}

\newcommand{\LogAlg}{\mathbf{LogAlg}}

\newcommand{\RC}{\mathscr{R}}

\newcommand{\AC}{\mathscr{A}}

\newcommand{\XC}{\mathscr{X}}

\newcommand{\notes}[1]{}

\renewcommand{\epsilon}{\ensuremath\varepsilon}

\renewcommand{\setminus}{\smallsetminus}
\renewcommand{\hat}{\widehat}

\renewcommand{\phi}{\ensuremath{\varphi}}

\DeclareMathOperator{\coker}{coker}

\DeclareMathOperator{\Rig}{Rig}
\DeclareMathOperator{\Hom}{Hom}

\DeclareMathOperator{\Ext}{Ext}

\DeclareMathOperator{\Lie}{Lie}

\DeclareMathOperator{\RR}{R\!R}
\DeclareMathOperator{\RS}{R\Gamma}

\DeclareMathOperator{\Spec}{Spec}
\DeclareMathOperator{\Spf}{Spf}

\DeclareMathOperator{\Spa}{Spa}

\DeclareMathOperator{\Pic}{Pic}

\DeclareMathOperator{\Frob}{Frob}

\DeclareMathOperator{\Sym}{Sym}

\DeclareMathOperator{\dlog}{dlog}
\DeclareMathOperator{\gr}{gr}

\newcommand{\Hm}{\mathbf{H}}

\makeatletter
\newcommand*{\triplerightarrow}[1]{\mathrel{
  \settowidth{\@tempdima}{$\scriptstyle#1$}
  \mathop{\vcenter{
    \offinterlineskip\ialign{\hbox to\dimexpr\@tempdima+1em{##}\cr
    \rightarrowfill\cr\noalign{\kern.5ex}
    \rightarrowfill\cr\noalign{\kern.5ex}
    \rightarrowfill\cr}}}\limits^{\!#1}}}
\makeatother

\renewcommand{\epsilon}{\ensuremath\varepsilon}

\renewcommand{\phi}{\ensuremath{\varphi}}

\theoremstyle{plain}
\newtheorem{theorem}{Theorem}[section]
\newtheorem{cor}[theorem]{Corollary}
\newtheorem{lemma}[theorem]{Lemma}
\newtheorem{prop}[theorem]{Proposition}

\newtheorem*{satz}[theorem]{Theorem}

\theoremstyle{definition}
\newtheorem{defn}[theorem]{Definition}
\newtheorem*{defn2}[theorem]{Definition}

\theoremstyle{definition}
\newtheorem{remark}[theorem]{Remark}

\theoremstyle{definition}
\newtheorem{example}[theorem]{Example}

\newtheorem{construction}[theorem]{Construction}

\newtheorem*{propo*}[theorem]{Proposition}

\begin{document}

\begin{frontmatter}

\title{Motivic $p$-adic periods of $1$-motives}

\author{Felix Sefzig\corref{cor1}}
\cortext[cor1]{}
\ead{felix.sefzig@math.uzh.ch}

\address{Institute of Mathematics, University of Z\"urich}

\begin{abstract}
We give an explicit construction of the $p$-adic de Rham comparison isomorphism for $1$-motives. In particular, we prove that our construction recovers the classical de Rham comparison isomorphism and is functorial with respect to morphisms of rigid $1$-motives. In the second part, we construct a new ring of motivic $p$-adic periods using the formalism of rigid analytic motives. Furthermore, we explain the relation between these motivic periods and the periods of the classical $p$-adic de Rham comparison isomorphism. Finally, we apply our construction to explicitly compute the period pairing for several examples of $1$-motives and curves.
\end{abstract}


\end{frontmatter}

\setcounter{tocdepth}{1}
\tableofcontents

\section{Introduction}
The main result of this article is the explicit construction of the $p$-adic de Rham comparison isomorphism for $1$-motives, first constructed by Faltings in \cite{faltings_cris}.
This construction allows us to compute the periods of semistable elliptic and hyperelliptic curves (Sections \ref{sect_kummer_mot} \& \ref{sect_hyperell}).
Our approach is motivated by Fontaine's formula \cite{Fontaine81} for the Hodge-Tate comparison isomorphism for abelian varieties with good reduction, as well as by Beilinson's construction of the de Rham isomoprhism \cite{Beil_de_Rham}, which we briefly recall below, respectively, in Section \ref{sect-comp-iso}.

\begin{satz}[Theorem \ref{thm_1_mot_intro}]
    Let $M,M'$ be $1$-motives. The $p$-adic de Rham comparison isomorphism extends to $1$-motives and is functorial with respect to morphisms
    \begin{equation*}
        M^\rig \to M'^\rig
    \end{equation*}
    in the category $D^b_\fppf(\Spa(K,O_K),\Q)$.
\end{satz}

Before explaining our construction, we recall Fontaine's formula for the Hodge-Tate comparison isomorphism. Let $A$ be an abelian variety with good reduction. Since the higher cohomology groups of $A$ are wedge powers of its $H^1$, it suffices to prove the Hodge-Tate comparison theorem for $n=1$.
In this case, we have to construct a natural $ G_K$-equivariant isomorphism
\[
H^1(A,\clg{O}_A) \otimes \C_p \oplus H^0(A,\Omega^1_{A/K}) \otimes \C_p(-1) \cong H^1_\et(A_{\ol{K}},\Q_p) \otimes \C_p. 
\]
Furthermore, recall that the \'etale cohomology of $A$ is computed by the Tate module of $A$
\[
H^1_\et(A_{\ol{K}}, \Z_p) \cong  T_p(A)^\vee := \left(\varprojlim A[p^n](\ol{K})\right)^\vee,
\]
where $A[p^n]$ denotes the kernel of multiplication by $p^n$ on $A$.
With this in mind, Tate and Raynaud established the theorem in the following form.

\begin{satz}[Tate, Raynaud]
Let $A$ be an abelian variety with good reduction over $K$. There exist functorial $K$-linear isomorphisms
\begin{align*}
\rho^0_A \colon &H^0(A,\Omega^1_{A/K}) \to \Hom_{\Z_p[G_K]}(T_p(A),\C_p(1)),\\
\rho^1_A \colon &H^1(A,\clg{O}_A) \to \Hom_{\Z_p[G_K]}(T_p(A),\C_p).  
\end{align*}
\end{satz}

Let $A'$ be the dual abelian variety of $A$. By duality - specifically, the natural isomorphisms $T_p(A') \cong \Hom_{\Z_p[G_K]}(T_p(A), \Z_p(1))$ and $H^1(A',\clg{O}_{A'}) \cong H^0(A,\Omega^1_{A/K})$ - the proof reduces to the construction of $\rho_A^0$, and proving its injectivity.
Fontaine's construction of $\rho_A^0$ is based on his computation of the Tate module of the K\"ahler differentials of $O_{\ol{K}}$ over $O_K$. He showed that the morphism
\begin{align*}
   d\log \colon \mu_{p^\infty} &\to \Omega^1_{O_{\ol{K}}/O_K} \\
    \zeta &\mapsto d\log \zeta := \frac{d \zeta}{\zeta}
\end{align*}
induces an isomorphism of $G_K$-modules
\begin{equation}\label{eq_fontaine_tate}
\C_p(1) \xrightarrow{\sim} T_p(\Omega^1_{O_{\ol{K}}/O_K})\otimes \Q_p.
\end{equation}

This identification suggests the following approach for constructing $\rho^0_A$. 
Since $A$ has good reduction, its N\'eron model $\clg{A}$ is a proper smooth model for $A$ and there exists an isomorphism
\[
H^0(\clg{A},\Omega^1_{\clg{A}/O_K}) \otimes K  \to H^0(A,\Omega^1_{A/K}).
\]
Furthermore, the Tate module of $A$ lifts to $O_K$ in the following sense: 
\[T_p(A) \cong T_p(\clg{A}) =: \varprojlim_n \clg{A}[p^n](O_{\ol{K}}).
\]
Now, given an integral invariant differential $\omega \in H^0(\clg{A},\Omega^1_{\clg{A}/O_K})$ and $x \in T_p (\clg{A})$, 
the pullback $x^*\omega$
of $\omega$ along $x$ defines an element of the Tate module of $\Omega^1_{O_{\ol{K}}/O_K}$,
which, under the identification (\ref{eq_fontaine_tate}), corresponds to an element of $\C_p(1)$. \footnote{In the complex world, $x^*\omega$ corresponds to the integral of a differential form over a $1$-simplex in $H_1(X,\Z)$.}.
To summarize, $\rho_A^0$ is defined as the $K$-linear extension of
\begin{align}\label{eq_rho_0}
    H^0(\clg{A},\Omega^1_{\clg{A}/O_K}) &\to 
    \Hom_{\Z_p[G_K]}(T_p(\clg{A}), \C_p(1))\\
    \omega &\mapsto [x \mapsto x^*\omega] \nonumber.
\end{align}

Adopting this strategy to construct a version of the de Rham comparison isomorphism, that is, defining an (injective) map
\[
\rho_A \colon H^1_\dr(A) \to  \Hom_{\Z_p[G_K]}(T_p(A), B_\dr^+),
\]
presents us with two main challenges:
\begin{enumerate}[(i)]
    \item We need an interpretation of Fontaine's period ring $B_\dr$ in terms of K\"ahler differentials similar to (\ref{eq_fontaine_tate}).
    \item We need to represent all de Rham classes by $O_K$-linear differentials.
\end{enumerate}
The first piece of the puzzle was provided by Beilinson, who realized that $B_\dr^+$ can be expressed in terms of the derived de Rham complex of $O_{\ol{K}}$ over ${O_K}$, see \cite{Beil_de_Rham} and \ref{appendix_de_rham}.

\begin{satz}
    There is a $G_K$-equivariant isomorphism of filtered $\ol{K}$-algebras
    \[
    B_\dr^+ \cong (L\hat{\Omega}^\bullet_{O_{\ol{K}}/O_K})^\wedge_{(p)}\left[\frac{1}{p}\right],
    \]
    where the right-hand side is endowed with the Hodge filtration.
\end{satz}
Here, $L\hat{\Omega}^\bullet_{O_{\ol{K}}/O_K})^\wedge_{(p)}$ is the Hodge-completed de Rham complex, and $(-)^\wedge_{(p)}$ denotes the derived $p$-completion. The basic idea is that replacing $\Omega^1_{O_{\ol{K}}/O_K}$ with 
$L\hat{\Omega}^\bullet_{O_{\ol{K}}/O_K}$ allows us to combine $\rho^0_A$ and $\rho^1_A$ into a morphism
\[
\rho_A \colon H^0(A,\Omega^1_{A/K}) \to \Hom(T_p(A), B_\dr^+).
\]

Now, let us explain the solution to problem (ii) in the case of an abelian variety.
\begin{satz}
    Let $A$ be an abelian variety. There exists a natural isomorphism 
    \[
    \rom{Inv}(\Omega^1_{E(A)/K}) \to H^1_\dr(A),
    \]
    where the left-hand side denotes the $K$-vector space of invariant differentials on the universal vectorial extension of $A$.
\end{satz}

Concretely, this means that there is a natural way to represent de Rham classes by invariant differentials. Thus, we are left with the problem of finding a way to lift $K$-linear differentials on $E(A)$ to $O_K$-linear differentials on some integral model of $E(A)$. We achieve this in a unified way by using rigid analytic geometry. More precisely, we prove that every de Rham class of $A$ (after multiplication by $p^r$) is represented by an integral differential form on a large enough open subgroup of the rigid analytic variety $E(A)^\rig$ (Lemma \ref{lem_qc_subgrp_A}). 
Then, we define $\rho_A$ roughly as follows. For $\omega \in H_\dr(A)$ and $x = (x_n)_n \in T_p(A)$ choose a system of lifts 
$\hat{x}_n \in E(A)^\rig(\C_p)$ and set $\rho_A(\omega)(x) = \varprojlim_n \hat{x}^*_n \omega'$, 
where $\omega'$ is an integral differential form representing $\omega$ (see Construction \ref{rigid_construction_abelian} for more details).

\subsection{Outline}
The text is organized as follows. In Section \ref{sect-comp-iso}, we recall Beilinson's construction of the $p$-adic de Rham comparison isomorphism and explain how the Poincar\'e Lemma can be used to give a canonical construction of the Hodge-Tate comparison isomorphism for abelian varieties (Section \ref{sect_hodge_tate}).

In Section \ref{sect_comp_formula}, we give an explicit construction of the de Rham comparison isomorphism for $1$-motives and prove the main theorem (Theorem \ref{thm_1_mot_intro}).
Along the way, we introduce the notion of integral differential forms on rigid analytic spaces and prove the key technical result about differentials on the (rigid) universal extension of a $1$-motive (Proposition \ref{prop_subgrp_1mot}). We end this section with the comparison between our construction and Beilinson's construction of the comparison isomorphism (Theorem \ref{thm_1-mot_comp_formula}).

In Section \ref{sect_motivic_period_ring}, we introduce a new ring of motivic $p$-adic periods and explain the relation between these motivic periods and the periods of the comparison isomorphism constructed in the previous section. Finally, we apply our construction to compute the periods of several examples of $1$-motives and curves.

In the Appendix, we included recollections on the categories of $1$-motives and rigid $1$-motives and their realizations (\ref{sect_1mot_cat}). Furthermore, we discuss the derived logarithmic de Rham complex and its relation to Fontaine's period rings (Theorem \ref{thm_period_ring_comp}). Finally, we give a construction of the logarithm map of $B_\stb$ using the description in terms of the derived logarithmic de Rham complex (\ref{sect_log_map}).

\subsection{Acknowledgments}
This paper is part of my PhD Thesis at the University of Z\"urich. I want to dearly thank my PhD advisor, Joseph Ayoub, for suggesting this project to me, for his steady guidance, the incredible insight he shared with me, and his endless patience. Moreover, I feel extremely fortunate to have been part of the algebraic geometry group at the University of Z\"urich, which always created a warm and welcoming atmosphere. 

\subsection{Notation and conventions}
We fix a prime number $p$.

\begin{defn2}
    A \(p\)-adic field is a complete, discretely valued non-archimedean extension of $\Q_p$ with perfect residue field.
\end{defn2}
We will use the following notations for a $p$-adic field $K$.

\begin{itemize}
    \item $O_K$ denotes the ring of integers of $K$. It is a discrete valuation ring with uniformizer $\varpi$ and residue field $k$.
    \item We write $\C_p$ for the completion of the algebraic closure $\hat{\ol{K}}$ of $K$.
    \item The absolute Galois group of $K$ is denoted by $G_K$.
    \item Write $W = W(k)$ for the ring of Witt vectors of $k$.
    \item We write
    $A_\infi$ for the ring of Witt vectors, $W(O_{\C_p}^\flat)$, of $O_{\C_p}^\flat := \lim_{x \mapsto x^{p}} O_{\C_p}$, the tilt of the perfectoid ring $O_{\C_p}$. Denote the maximal ideal of $O_{\C_p}^\flat$ by $\frk{m}^\flat$.
\end{itemize}

For a complex $K \in D(\Z_p)$ we write 
\[
K^\wedge_{(p)} := R\lim_n K \otimes^L \Z/p^n
\]
for the derived $p$-completion of $K$. If $K$ has $\Z_p$-flat terms, then $K^\wedge_{(p)}$ is computed by the termwise $p$-adic completion of $K$.

Unless specified otherwise, all schemes are assumed to be quasi-compact and separated. All rigid varieties are assumed to be separated but not necessarily quasi-compact. Let $S$ be a scheme (resp. rigid variety), we denote by $\Sm_S$ the category of smooth schemes over $S$ (resp. by $\SmR_S$ the category of smooth rigid varieties over $S$). A $K$-variety is a reduced, separated scheme of finite type over a field $K$.

A prelog scheme is a pair $(X,\alpha \colon M \to \clg{O}_X)$, where $X$ is a scheme, $M$ is a sheaf of monoids in the \'etale topology on $X$, and $\alpha$ is a monoid morphism (where we consider $\clg{O}_X$ as a sheaf of monoids via multiplication). We often write $(X,M)$ or $(M \to X)$. A prelog scheme is a log scheme if the induced morphism $\alpha \colon \alpha^{-1}\clg{O}_X^\times \to \clg{O}_X^\times$ is an isomorphism. If $D \subset X$ is a closed subscheme with open complement $j \colon U \to X$, we write $(X,D)$ for the (pre)log-scheme $(X,j_* \clg{O}_U^\times \cap \clg{O}_X \to \clg{O}_X)$.

For a scheme $X$, we write $X_\et$ for the small \'etale site of $X$.
We write $\SH(X)$ for the stable motivic homotopy category and $\otimes$ for the symmetric monoidal structure on $\Hm(X)_*$ and $\SH(X)$. Moreover, we denote by $\DA_\et(S,\Lambda)$ (resp. $\RDA_\et(S, \Lambda)$) the category of \'etale (rigid) $S$-motives with coefficients in a ring $\Lambda$. In both cases, motives form a closed monoidal category; we denote the unit object by $\Lambda$ or by $\UN$.

\section{Beilinson's $p$-adic de Rham comparison isomorphism}\label{sect-comp-iso}
In this subsection, we recall Beilinson's construction of the de Rham to \(p\)-adic \'etale comparison isomorphism for smooth schemes following \cite{Beil_de_Rham} and \cite{SzaBei18}. 
To do this, we consider the following categories of "good" compactifications \cite[Section 2.2]{Beil_de_Rham}.
\begin{enumerate}
    \item \textit{Arithmetic pairs}: Let $K$ be a \(p\)-adic field. An arithmetic pair over $K$ is a pair $(U,\overline{U})$ 
    consisting of a $K$-variety $U$, a reduced proper flat $O_K$-scheme $\overline{U}$ together with an open embedding $j\colon U \to \overline{U}$ such that $U$ is dense in $\overline{U}$. Denote the category of arithmetic pairs by $\Var_K^\cc$.    
    For an arithmetic pair $(U,\ol{U})$ set $O_{K_U} := \Gamma(\ol{U},\clg{O}_{\ol{U}})$ and $K_U := \Gamma(\ol{U}_K,\clg{O}_{\ol{U}_K})$. The ring $K_U$ is a product of finite extensions of $K$, and if $\ol{U}$ is normal, $O_{K_U}$ is the product of the corresponding rings of integers. 
    We say that $(U,\ol{U})$ is a \textit{semistable pair}, or ss-pair, if $\ol{U}$ is regular, $D:= \ol{U} \setminus U$ is a divisor with normal crossings on $\ol{U}$ and the special fibers of 
    $\ol{U} \to \Spec O_{K_U}$ are reduced. Denote the full subcategory of ss-pairs by $\Var_K^\ses$.
    \item \textit{Arithmetic $\ol{K}$-pairs}: Let $K$ be a \(p\)-adic field with algebraic closure $\ol{K}$. An arithmetic $\ol{K}$-pair is a pair $(V,\overline{V})$ 
    consisting of a $\ol{K}$-variety $V$, a reduced proper flat $O_{\ol{K}}$-scheme $\overline{V}$ together with an open embedding $j\colon V \to \overline{V}$ such that $V$ is dense in $\overline{V}$. Denote the category of arithmetic $\ol{K}$-pairs by $\Var_{\ol{K}}^\cc$.
    A connected pair $(V,\ol{V})$ is said to be \textit{semistable} if there exists an ss-pair $(U,\ol{U})$ and a $\ol{K}$-point $\alpha \colon K_U \to \ol{K}$ such that $(V,\ol{V})$ is isomorphic to the base change of $(U,\ol{U})$ over $K_U$ to $\ol{K}$ along $\alpha$. 
    An arithmetic $\ol{K}$-pair $(V,\ol{V})$ is semistable if all its connected components are.
    Denote the full subcategory of semistable pairs by $\Var_{\ol{K}}^\ses$.
\end{enumerate}

\begin{defn}
    Let $\clg{P}$ be any of the categories defined above. The $h$-topology on $\clg{P}$ is the finest
    topology for which the restriction of an $h$-sheaf on $\Var_K$ (or $\Var_{\ol{K}}$) to $\clg{P}$ along the 
    forgetful functor \[(U,\ol{U})\mapsto U,\] is a sheaf.
\end{defn}
More explicitly, a family
\[
\{(U_i, \ol{U_i}) \to (U,\ol{U})\}
\]
is an $h$-cover in $\clg{P}$ if and only if
\[
\{U_i \to U\}
\]
is an $h$-cover of varieties, see \cite[\href{https://stacks.math.columbia.edu/tag/0EU5}{Tag 0EU5}]{stacks-project} for more details on the $h$-topology.

The following theorem \cite[Theorem 2.6]{Beil_de_Rham}, which is a consequence of de Jong's Theorem on alterations \cite{deJongAlt}, is the backbone of Beilinson's construction of the comparison isomorphism. 
We fix a \(p\)-adic field $K$ for the rest of this section.
\begin{theorem}\label{eq-h-sheaf}
    The forgetful functor induces an equivalence of categories between the category of $h$-sheaves on $\Var_K$, respectively on $\Var_{\ol{K}}$, and the category of $h$-sheaves on either of $\Var_K^\cc, \Var_K^\ses$, respectively on either of $\Var_{\ol{K}}^\cc, \Var_{\ol{K}}^\ses$.
\end{theorem}

Let $(U,\ol{U})$ be an arithmetic pair. We view $\ol{U}$ as a log scheme with the canonical log structure given by 
\[
\can := (\clg{O}_{\ol{U}} \cap j_*\clg{O}_{U}^\times \to \clg{O}_{\ol{U}}).
\]
Furthermore, we denote by
\begin{equation*}
    L\hat{\Omega}^\bullet_{(\ol{U},\can)/O_K}
\end{equation*}
the Hodge-completed derived logarithmic de Rham complex of the log scheme 
$\ol{U}$ over $(\Spec O_K,O_K^\times)$,
which we consider as the projective system $L\hat{\Omega}_{(\ol{U},\can)/O_K}^\bullet/F^i$, see  \ref{appendix_de_rham} for more details on the derived logarithmic de Rham complex.

\begin{defn}
We define a pro-system of presheaves of
filtered $E_\infty$-$O_{\overline{K}}$-algebras on $\Var_{\ol{K}}^\ses$ by 
\begin{equation*}
    (U, \ol{U}) \mapsto \RS(\ol{U}, L\hat{\Omega}^\bullet_{(\ol{U},\can)/O_K}),
\end{equation*}
and we denote the associated $h$-sheaf on $\Var_{\ol{K}}^\ses$ by $\clg{A}_{\dr}$.\\ 
\end{defn}

Let $A_\dr$ be the $p$-adically complete ring $\clg{A}_{\dr}(\Spec \overline{K},\Spec O_{\ol{K}})^\wedge_{(p)}$, see \ref{appendix_period_rings} for more details on period rings, derived \(p\)-completions and the fact that $A_\dr$ is a classical ring. 
The structure map of semistable pairs induces a morphism (of pro-systems) of $h$-sheaves $A_\dr \to (\clg{A}_{\dr})^\wedge_{(p)}$, where $A_\dr$ is viewed as a constant sheaf.

\begin{theorem}[Beilinson's Poincar\'e Lemma, {\cite[Theorem 3.3]{Beil_de_Rham}}]
The induced maps
\begin{equation*}
    A_\dr/F^i \to (\clg{A}_{\dr}/F^i)^\wedge_{(p)},
\end{equation*}
where $F^i$ denotes the $i$'th step of the Hodge Filtration on $\clg{A}_\dr$, form a compatible system of quasi-isomorphisms of $h$-sheaves on $\Var_{\overline{K}}$ for all $i \geq 0$.
\end{theorem}

\begin{proof}
    We provide an outline of the proof. The theorem follows by looking at the Hodge filtration from the following two claims using the derived Nakayama Lemma:
    For any $(U,\ol{U}) \in \Var_{\ol{K}}^\ses$, there exists an $h$-cover $f \colon (V,\ol{V}) \to (U,\ol{U})$ such that
    \begin{enumerate}[(i)]
        \item The induced map
        \[
        \tau^{\geq 1} \RS(\ol{U},\clg{O}_{\ol{U}}) \to  \RS(\ol{V},\clg{O}_{\ol{V}})
        \]
        is \(p\)-divisible in $D(O_{\ol{K}})$\footnote{We use cohomogical notation for the trunaction.}.
        \item For $i> 0$, the induced map
        \[
        \RS(\ol{U}, \Omega^i_{(\ol{U},\can)/(O_{\ol{K}},\can)}) \to \RS(\ol{U}, \Omega^i_{(\ol{U},\can)/(O_{\ol{K}},\can)})
        \]
        is \(p\)-divisible in $D(O_{\ol{K}})$.
    \end{enumerate}
    The first claim is in Bhatt's Thesis \cite{bhatt_thesis}. The second claim is immediately reduced to the case of $i=1$ and, as we will explain next, to the affine case. 
    In the affine case, extracting \(p\)'th roots gives the desired $h$-cover.
    Let $(U, \ol{U})$ be a semistable pair and assume that there exists an affine open cover 
    $\{U_i \to \ol{U_i}\}$ and proper surjective maps $\pi_i \colon V_i \to U_i$ such that 
    \[
    \pi_i^* \Omega^1_{(U_i,\can)/(O_{\ol{K}},\can)} \to \Omega^1_{(V_i,\can)/(O_{\ol{K}},\can)}
    \]
    is divisible by $p$, where all the log structures are pulled back from $\ol{U}$.
    By Nagata, we can find a proper surjection $\pi \colon \ol{V} \to \ol{U}$
    that factors through $\pi_i$ over $U_i$, and, by de Jong, we can ensure that $(\ol{V},\can)$ defines a semistable pair. Since $\Omega^1_{(\ol{V},\can)/(O_{\ol{K}},\can)}$ is $\Z_p$-flat it follows that
    \[
    \pi_i^* \Omega^1_{(\ol{U},\can)/(O_{\ol{K}},\can)} \to \Omega^1_{(\ol{V},\can)/(O_{\ol{K}},\can)}
    \]
    is divisible by $p$.
    We refer to \cite[Lemma 10.13]{bhatt2012padicderivedrhamcohomology} for more details.
\end{proof}

The construction of the de Rham comparison isomorphism relies on the following facts. The first proposition follows from $h$-descent for de Rham cohomology over fields of characteristic zero.
\begin{prop}[{{\cite[Proposition 3.4]{Beil_de_Rham}}}]
    Let $X$ be a smooth variety over $\ol{K}$. The complex
    \begin{equation*}
        \RS_h(X, \clg{A}_\dr) \otimes \Q_p
    \end{equation*}
    computes de Rham cohomology of $X$ endowed with Deligne's Hodge filtration.
\end{prop}

\begin{prop}
The complex 
\begin{equation*}
    A_\dr/F^i \otimes \Q_p := \left(\varprojlim_n \clg{A}_\dr(\Spec O_{\ol{K}},\Spec \ol{K})/F^i \otimes^L \Z/p^n\right) \left[\frac{1}{p}\right]
\end{equation*}
is concentrated in degree zero, and its zeroth homology is
given by the ring $B_\dr^+/F^i$, where $B_\dr^+$ denotes Fontaine's period ring endowed with its natural filtration $F^\bullet$.
\end{prop}
\begin{proof}
    See Appendix \ref{thm_period_ring_comp}.
\end{proof}

\begin{cor}
Let $X$ be a smooth variety over $\ol{K}$. There exists a natural filtered quasi-isomorphism 
\begin{align*}
    \RS_{\et}(X, \Z_p) \hat{\otimes}_{\Z_p} B_\dr^+ &\to \varprojlim_i
    \left(\RS_h(X, (\clg{A}_\dr/F^i)^\wedge_{(p)})\otimes \Q_p\right) \\
    &= \varprojlim_i\left(\left(\varprojlim_n \RS_h(X,\clg{A}_\dr/F^i \otimes^L \Z/p^n)\right)\left[\frac{1}{p}\right]\right)  .
\end{align*}

\end{cor}
\begin{proof}
Since $A_\dr/F^i \otimes^L \Z/p^n$ is a constant torsion sheaf and \'etale cohomology with torsion coefficients satisfies $h$-descent, see \cite[Theorem 10.7]{suslin_veov_singhom}, the natural map
\begin{equation*}
    \RS_\et(X, \Z_p)\otimes^L_{\Z_p} (A_\dr/F^i\otimes^L \Z/p^n) \cong \RS_\et(X, A_\dr/F^i \otimes^L \Z/p^n) \to \RS_h(X, A_\dr/F^i \otimes^L \Z/p^n)  
\end{equation*}
is a quasi-isomorphism. Moreover, the Poincar\'e Lemma gives a quasi-isomorphism
\begin{equation*}
        \RS_h(X, A_\dr/F^i \otimes^L \Z/p^n) \xrightarrow{\sim} \RS_h(X, \clg{A}_\dr/F^i \otimes^L \Z/p^n). 
\end{equation*}
Now, taking derived \(p\)-completion and inverting \(p\) yield the desired quasi-isomorphism
\begin{equation*}
    \RS_\et(X, \Z_p)\hat{\otimes}_{\Z_p} B_\dr^+/F^i \xrightarrow{\sim} \RS_h(X, (\clg{A}_\dr/F^i)^\wedge_{(p)})\otimes \Q_p.
\end{equation*}
\end{proof}

Let $X$ be a smooth variety over $K$. Combining everything above 
gives a natural filtered $G_K$-equivariant morphism 
\begin{equation*}
    \RS_\dr(X)  \to \RS_h(X_{\ol{K}}, \clg{A}_\dr)\otimes \Q_p \to \RS_h(X_{\ol{K}}, (\clg{A}_\dr)^\wedge_{(p)} )\otimes \Q_p
    \xrightarrow{\sim} \RS_{\et}(X_{\overline{K}}, \Z_p) \otimes^L B_\dr^+.
\end{equation*}
We denote the $B_\dr$-linear extension of the above composition by $\rho_\dr$, where $B_\dr := \rom{Frac}(B_\dr^+)$.

\begin{theorem}[{{\cite[Theorem 3.6]{Beil_de_Rham}}}]
    The morphism $\rho_\dr$ is a filtered quasi-isomorphism for every smooth variety $X$ over $K$.
\end{theorem}
The theorem follows from the computation of $\rho_\dr$ in the case of $\G_m$ using Poincar\'e duality, Gysin morphisms, and compatibility of $\rho_\dr$ with Chern classes.
We refer to the original paper \cite{Beil_de_Rham} for further details. Finally, we record that the de Rham sheaf $\clg{A}_\dr$ and hence the comparison isomorphism are representable in the motivic category $\DA_\et(K,\mathbb{Z})$.

\begin{cor}\label{cor_adr_reprs}
    The presheaf 
    \begin{equation*}
        X \mapsto \RS_h(X_{\ol{K}}, \clg{A}_\dr)
    \end{equation*}
    and, consequently, the comparison isomorphism $\rho_\dr$, are representable in $\DA_\et(K)$. 
\end{cor}
\begin{proof}
    Denote the above presheaf by $F$. To show that $F$ is representable in $\DA^{\eff}_\et(K)$, we need to prove that $F$ is $\A^1$-invariant and satisfies \'etale descent.
    The latter holds by definition of $F$.
    To deal with $\A^1$-invariance, it suffices to check that $F \otimes \Z[\frac{1}{p}]$ and the derived \(p\)-completion $F^{\wedge}_{(p)}$ are $\A^1$-invariant. To see this, observe that a presheaf $G$ is zero if $G \otimes \Z[\frac{1}{p}]$ and $G^\wedge_{(p)}$ are zero. By definition derived $p$-complete sheaves are right orthogonal to sheaves of complexes of $\Z[\frac{1}{p}]$-modules, see \cite[Section 3]{BS_proetale}. In particular, we have an exact triangle
    \[
    R\Hom(\Z[\frac{1}{p}], G) \to G \to G^{\wedge}_{(p)} \xrightarrow{+1}.
    \]
    Thus, if $G^\wedge_{(p)} \simeq 0$ it follows that $G \simeq R\Hom(\Z[\frac{1}{p}], G)$ which is naturally a $\Z[\frac{1}{p}]$-module and hence $G \simeq G\otimes \Z[\frac{1}{p}] \simeq 0$.
    Finally, $F \otimes \Z[\frac{1}{p}]$, and $F^{\wedge}_{(p)}$ are $\A^1$-invariant, because
    $F \otimes \Z[\frac{1}{p}]$ computes de Rham cohomology and $F^{\wedge}_{(p)}$ computes \'etale cohomology.
    
    Moreover, to show that $F$ is representable by a $\PR^1$-spectrum, it suffices to show that the natural map
    \[
    F \to \Omega_{\PR^1} F,
    \]
    where $\Omega_{\PR^1} F$ denotes the $\PR^1$-loop space of $F$, is a quasi-isomorphism. By the same argument as above, this follows from the corresponding facts for de Rham and \'etale cohomology.
\end{proof}
 
\subsection{The Hodge-Tate comparison for abelian varieties with good reduction}\label{sect_hodge_tate}
Let us explain how to use Beilinson's Poincar\'e Lemma to recover Fontaine's formula for the Hodge-Tate comparison isomorphism. The construction below will strongly guide our intuition in the de Rham case.

Let $A$ be an abelian variety with good reduction. Recall the following form of the Hodge-Tate comparison theorem for abelian varieties from the introduction.

\begin{theorem}[Tate, Raynaud]
Let $A$ be an abelian variety with good reduction over $K$. There exist functorial $K$-linear isomorphisms
\begin{align*}
\rho^1_A \colon &H^1(A,\clg{O}_A) \to \Hom_{\Z_p[G_K]}(T_p(X),\C_p)  \\
\rho^0_A \colon &H^0(A,\Omega^1_{A/K}) \to \Hom_{\Z_p[G_K]}(T_p(X),\C_p(1)).  
\end{align*}
\end{theorem}

As explained in the introduction, $\rho^1_A$ can be constructed from $\rho_{A'}^0$ by using the identification 
\[
T_p(A) \cong \Hom_{\Z_p[G_K]}(T_p A', \Z_p(1)),
\]
where $A'$ denotes the dual abelian variety of $A$. 
However, we will use Beilinson's Poincar\'e Lemma to give a canonical construction of both $\rho^1_A$ and $\rho^0_A$. 

\begin{construction}\label{constr_HT_comp}
Let $A$ be an abelian variety with good reduction over $K$ with N\'eron model $\clg{A}$.
Since $A$ has good reduction, the Tate module of $A$ admits an integral model 
\[
T_p \clg{A} :=
\varprojlim_{n} \clg{A}[p^n].
\]
Moreover, we consider the perfectoid cover of $\clg{A}$ given by
\[
\varprojlim_{x \mapsto p \cdot x} \clg{A} =: \clg{A}_\infty  \to \clg{A},
\]
which is a torsor under $T_p\clg{A}$.
Note that $\clg{A}_\infty$ exists as a scheme since multiplication by \(p\) is a finite morphism \cite[Tag 01YX]{stacks-project} and we have a short exact sequence of group schemes
\begin{equation}\label{ses_perf_cover}
0 \to T_p \clg{A} \to \clg{A}_\infty \to \clg{A} \to 0.
\end{equation}

Since we will be working on the \'etale or $h$-site, it is better to view $T_p \clg{A}$ and $\clg{A}_\infty$ as pro-objects in $\Sh_{\et}(\Var_K)$. In particular, (\ref{ses_perf_cover}) gives rise to a pro-system of short exact sequences
\footnote{Recall that morphisms between pro-objects in a category $\clg{C}$ are given by $
\Hom_{\clg{C}}((X_i),(Y_j)) = \varinjlim_i \varprojlim_j \Hom_{\clg{C}}(X_i,X_j).$
By abuse of notation, we write $\Hom_{\clg{C}}(-,-)$ for the set of morphisms between pro-objects when there is no risk of confusion.} of $h$-sheaves on $\Var_K^{\ses}$ (or equivalently on $\Var_{K}$).

\textbf{Step 1: $\rho^1_A$.} Let $\alpha \in H^1(A,\clg{O}_A)$. Our goal is to construct a continuous $G_K$-equivariant morphism
\[
\rho^1_A(\alpha) \colon T_p A \to \C_p,
\]
which depends $K$-linearly on $\alpha$.
Since $\clg{A}$ is a proper model for $A$, it suffices to construct $\rho^1_A(\alpha)$ for $\alpha \in H^1(\clg{A},\clg{O}_{\clg{A}})$.
Moreover, because $\clg{A}$ is an abelian scheme every $\G_a$-torsor over $\clg{A}$ admits the structure of a
vectorial extension and therefore we can represent every $\alpha \in H^1(\clg{A},\clg{O}_{\clg{A}})$ by a morphism in the derived ($\infty$-)category of Zariski sheaves on $\Var_K^\ses$
\[
\alpha \in \Hom_{\PSh(\Var_{K}^\ses,\Z)}(\clg{A},L_{\Zar}\clg{O}[1]).
\]
After base-change\footnote{We write $\clg{A}_{\ol{K}}$ for the base change $\clg{A} \times_{O_K} O_{\ol{K}}$.}, $h$-sheafification and derived \(p\)-completion, $\alpha$ gives rise to a morphism
\[
\alpha' \in \Hom_{\Sh_{\rom{h}}(\Var_{\ol{K}},\Z)}(\clg{A}_{\ol{K}}, (\gr^0_F \clg{A}_\dr)^\wedge_{(p)}[1])
\cong \Hom_{\Sh_{\rom{h}}(\Var_{\ol{K}},\Z)}(\clg{A}_{\ol{K}},O_{\C_p}[1]),
\]
where we used the Poincar\'e Lemma in the last step and $O_{\C_p}$ denotes the pro-system of constant $h$-sheaves $(O_{\ol{K}}/p^n)_n$.
The long exact sequence associated to (\ref{ses_perf_cover}) reads as
\[
    \begin{tikzcd}
        0 \ar[r] & \Hom(\clg{A}_{\ol{K}},O_{\C_p}) \ar[r] &
        \Hom(\clg{A}_{\infty,\ol{K}},O_{\C_p}) \ar[r] & \Hom(T_p\clg{A}_{\ol{K}},O_{\C_p}) \ar[r, "\delta"] & \hphantom{0}\\
        \hphantom{\cdots} \ar[r,"\delta"] 
        & \Hom(\clg{A}_{\ol{K}},O_{\C_p}[1])\ar[r] 
        & \Hom(\clg{A}_{\infty,\ol{K}},O_{\C_p}[1]) \ar[r]
        &  \Hom(T_p\clg{A}_{\ol{K}},O_{\C_p}[1])\ar[r] & \cdots.
    \end{tikzcd}
\]
Since $\Hom(\clg{A}_\infty, -)$ is \(p\)-divisible, morphisms from $\clg{A}_{\infty,\ol{K}}$ to the constant sheaf $O_{\C_p}/p^n$ vanish and we are left with an isomorphism
\[
\delta \colon \Hom_{\Sh_{\rom{h}}(\Var_{\ol{K}},\Z)}(T_p\clg{A}_{\ol{K}},O_{\C_p}) \to \Hom_{\Sh_{\rom{h}}(\Var_{\ol{K}},\Z)}(\clg{A}_{\ol{K}},O_{\C_p}[1]).
\]
Since $h$-hypercovers satisfy cohomological descent for torsion \'etale sheaves, we can identify the right-hand side with
\[
\Hom_{\Sh_\et(\Var_K,\Z)}(T_p \clg{A}, O_{\C_p}) = \varprojlim_n \Hom_{\Z[G_K]}(A[p^n](\ol{K}), O_{\C_p}/p^n) = \Hom_{\Z_p[G_K]}(T_pA,O_{\C_p}).
\]
Under these identifications, we define $\rho^1_A$ as
\[\rho^1_A(\alpha) := \delta^{-1}(\alpha'),\]
compare also with \cite[Section 2.2]{bhatt_arizona}.

\textbf{Step 2: $\rho^0_A$.} Let $\omega \in H^0(A, \Omega^1_{A/K})$.
As before, we want to construct a continuous $G_K$-equivariant map
\[
\rho^0_A(\omega) \colon T_p A \to \C_p(1)
\]
which depends $K$-linearly on $\omega$. Since $A$ has good reduction, we may assume that $\omega$ is lifts to an invariant differential form on $\clg{A}$, in other words, $\omega$ is contained in the image of $H^0(\clg{A},\Omega^1_{\clg{A}/O_K}) \to H^0(\clg{A},\Omega^1_{\clg{A}/O_K}) \otimes_{O_K} K \cong H^0(A, \Omega^1_{A/K})$, and hence $\omega$ corresponds to a morphism of presheaves
\[
\omega \in \Hom_{\PSh(\Var_{K}^\ses,\Z)}(\clg{A},\Omega^1_{\clg{A}/O_K}).
\]
By Lemma \ref{lem_invariant_differentials_faltings}, we may further assume that $\omega$ lifts to an invariant differential on
$\clg{A}_{O_{\ol{K}}}$ over $O_K$ (not only over $O_{\ol{K}}$). Thus, $\omega$ induces, after $h$-sheafification and derived \(p\)-completion, a morphism
\[
\omega' \in \Hom_{\Sh_{\rom{h}}(\Var_{\ol{K}},\Z)}(\clg{A}_{\ol{K}},(\gr^1_F \clg{A}_\dr)^\wedge_{(p)}[1])\cong  \Hom_{\Sh_{\rom{h}}(\Var_{\ol{K}},\Z)}(\clg{A}_{\ol{K}},T_p \Omega^1_{O_{\ol{K}}/O_K}[1]).
\]
As before, the long exact sequence associated to (\ref{ses_perf_cover}) gives an isomorphism
\begin{equation}\label{eq_delta}
\delta \colon \Hom_{\Sh_{\rom{h}}(\Var_{\ol{K}},\Z)}(T_p \clg{A}_{\ol{K}}, T_p \Omega^1_{O_{\ol{K}}/O_K})
\to \Hom_{\Sh_{\rom{h}}(\Var_{\ol{K}},\Z)}(\clg{A}_{\ol{K}},T_p \Omega^1_{O_{\ol{K}}/O_K}[1]),   
\end{equation}
where we used the Poincar\'e Lemma and the following identifications from  \ref{appendix_period_rings}: 
\[
\gr^1_F L\Omega^\bullet_{O_{\ol{K}}/O_K} \simeq L_{O_{\ol{K}}/O_K}[-1] \simeq \Omega^1_{O_{\ol{K}}/O_K}
\]
and 
\[
(L_{O_{\ol{K}}/O_K})^\wedge_{(p)} \simeq T_p(\Omega^1_{O_{\ol{K}}/O_K})[1],
\]
see Lemma \ref{lem_cotangent_O_K}. As in the previous step, we can identify the left-hand side of (\ref{eq_delta}) with
\[
\Hom_{\Z_p[G_K]}(T_p A, T_p \Omega^1_{O_{\ol{K}}/O_K}).
\]
Finally, we define 
\[
\rho^0_A(\omega) = \delta^{-1}(\omega').
\]
A diagram chase, where we use the Breen-Deligne resolution of $\clg{A}$ to compute $R\Hom(-,\gr^1\clg{A}_\dr)$, shows that $\rho^0_A$ defined in this way agrees with Fontaine's construction from the introduction.
\end{construction}

\newpage

\section{The comparison theorem for 1-motives}\label{sect_comp_formula}
In this section, we give an explicit construction of the $p$-adic de Rham comparison isomorphism for $1$-motives, see Theorem \ref{prop_cdr}. In particular, we show that our construction recovers Beilinson's comparison isomorphism, as stated in Theorem \ref{thm_1-mot_comp_formula}.
Combining the two results implies the main theorem.

\begin{theorem}\label{thm_1_mot_intro}
    Let $M,M'$ be $1$-motives. The $p$-adic de Rham comparison isomorphism extends to $1$-motives and is functorial with respect to morphisms
    \begin{equation*}
        M^\rig \to M'^\rig
    \end{equation*}
    in the category $D^b_\fppf(\Spa(K,O_K),\Q)$.
\end{theorem}

\begin{remark}
    The above theorem says in particular that the comparison isomorphism is functorial from morphisms of rigid $1$-motives, i.e., morphisms in the category $\clg{M}_1(K)^\rig\otimes{\Q}$ as introduced in Appendix \ref{def_rig_1mot}.
\end{remark}

Before we proceed with the construction, let us fix some notation. Let $K$ be a $p$-adic field with ring of integers $O_K$ and uniformizer $\varpi \in O_K$. 
For a $K$-variety $X$, we write
$X^\rig$ for the associated rigid analytic variety. For a formal $O_K$-scheme $\scr{X}$, we write $\scr{X}_\eta$ for the Raynaud generic fiber of $\scr{X}$. 
If $\scr{X}$ is the formal completion of an $O_K$-scheme with generic fiber $X$ and special fiber $\scr{X}_0$,
then there is a canonical map $\scr{X}_\eta \to X^\rig$ which is an open immersion in general and an isomorphism if $\scr{X}$ is proper.
Since our construction uses rigid geometry, we start by recalling some facts about differentials on rigid analytic spaces.

\subsection{Differentials on rigid analytic spaces} 
Let $X$ be a rigid analytic variety locally of finite type over $K$. 
We recall some basic facts about the sheaf of continuous differentials $\Omega^1_{X/K}$ on $X$.
If $U = \Spa(B,B^\circ)$ is an open affinoid of $X$ then $H^0(U,\Omega^1_{X/K})$ is given by the module of continuous differentials of $B$ over $K$. 
Moreover, if $X$ has a formal model $\scr{X}$, then $\Omega^1_{X/K}$ naturally identifies with 
\[
\Omega^{1,\cont}_{\scr{X}/O_K} \left[\frac{1}{p}\right] := 
\left(\varprojlim_n \Omega^1_{\scr{X}/p^n/(O_K/p^n)} \right)\left[\frac{1}{p}\right].
\]
We say $\scr{X}$ is a \textit{strict semistable model}, if every closed point $x \in \scr{X}_0$ of the special fiber admits an open neighborhood which is smooth over 
\[
\Spf O_K\langle t_1,\dots,t_r\rangle / (t_1\cdots t_r - \varpi)
\]
for some $r \geq 0$.
In this case, the irreducible components $D_1,\dots,D_m$ of $\scr{X}_0$ are smooth divisors and $\scr{X}$ endowed with the canonical log-structure $(\scr{X},\can)$\footnote{This agrees with the definition given below.} associated to 
\[
D_1 + \cdots + D_m
\]
is log smooth over $(O_K,\can)$ \cite[Proposition 1.3]{hartl_rig_picard}.
In particular, 
\[
\Omega^{1,\cont}_{\scr{X}/O_K} \left[\frac{1}{p}\right]\cong \Omega^{1,\cont}_{(\scr{X},\can)/(O_K,\can)} \left[\frac{1}{p}\right]
\]
is locally free and $X$ is smooth. 

\begin{defn}
We define the sheaf of \textit{integral} differential forms $\Omega^{1,+}_{X/K} \subset \Omega^1_{X/K}$ as the sheafification of the presheaf defined on admissible open affinoids $U = \Spa(B,B^\circ)$ by
\[
U \mapsto \Omega^{1,\cont}_{(B^\circ,\can)/(O_K,\can)},
\]
where the canonical log-structure on $B^\circ$ is the log-structure associated to $(B^\circ[\frac{1}{\varpi}])^\times \cap B^\circ \to B^\circ$.
\end{defn}

Let $X$ be a smooth $K$-variety with associated rigid analytic variety $X^\rig$. Assume that $X$ admits a strict semistable compactification $P$ over $O_K$. 
In the language of Section 1, $(X,P)$ is a strict semistable pair over $O_K$.
In particular, $P \setminus X = D$ is a strict normal crossing divisor 
and $D$ decomposes into horizontal and vertical parts $D = D_h \cup D_v$. In particular, 
$(P \setminus D_v)\times_{O_K} K = X$. Furthermore, we write $\scr{X}$ for the formal completion of
$P \setminus D_v$ along the special fiber.

\begin{lemma}\label{lem_log_forms_ssmodel}
    Let $X$ and $P$ be as above. Suppose that $U$ is a quasi-compact open subset of $X^\rig$ containing $\scr{X}_\eta$. Then, there exists a natural injective map
    \[
    H^0(P,\Omega^1_{(P,D)/(O_K,\can)}) \to H^0(U,\Omega^1_{U/K}).
    \]
\end{lemma}
\begin{proof}
    It suffices to prove the lemma for $U = \scr{X}_\eta$.
    Since $P$ is flat and proper over $O_K$, $p$-completion induces an isomorphism
    \[
    H^0(P,\Omega^1_{(P,D)/O_K}) \xrightarrow{\simeq} H^0(\hat{P}, \Omega^{1,\cont}_{(\hat{P},D)/O_K}).
    \]
    It remains to see that the composition  
    \[
    H^0(\hat{P}, \Omega^{1,\cont}_{(\hat{P},D)/O_K}) \to H^0(\scr{X}, \Omega^{1,\cont}_{(\hat{P},D)/O_K}) \to H^0(\scr{X}_\eta, \Omega^1_{X/K})
    \]
    is injective. The first map is injective because $\scr{X}$ is a dense open subscheme of $(\hat{P},D)$. The second map is injective because 
    \(
    \Omega^{1,\cont}_{(\hat{P},D)/O_K}
    \)
    is locally free (using that the canonical log structure on $\scr{X}$ is the log structure pulled back from $\hat{P}$ along $\scr{X} \to \hat{P}$).
\end{proof}

\subsection{Invariant differentials}
Let us also recall some useful facts about invariant differentials on group schemes. The reader may skip this section and return to it at an appropriate time.

Let $G$ be a group scheme over a scheme $S$. Let $T$ be an $S$-scheme and let $g \in G(T)$ be a $T$-point of $G$. The left translation by $g$ is defined as the composition
\[
L_g \colon G_T \xrightarrow{\sim} T \times_T G_T \xrightarrow{g_t \times \id} G_T \times_T G_T
\xrightarrow{m_T} G_T,
\]
where $m$ denotes the multiplication map of $G$.
Note that every left translation is obtained by base change from the universal left translation
\[
L_{\id} \colon G \times_S G \to G \times_S G, \quad (x,y) \mapsto (x, m(x,y)).
\]
\begin{defn}
We say a global section $\omega \in H^0(G,\Omega^i_{G/S})$ is \textit{left-invariant} if $L_g^* \omega_T = \omega_T$ under the canonical isomorphism
\[
L_g^* \Omega^i_{G_T/T} \xrightarrow{\sim} \Omega^i_{G_T/T},
\]
for all $T \in \Sch_S$ and all $g \in G(T)$.
If $G$ is abelian, we simply say that $\omega$ is \textit{invariant} and we denote the set of invariant 1-forms on $G$ by $\rom{Inv}(G) \subset H^0(G,\Omega^1_{G/S})$.
\end{defn}

\begin{prop}[{}{\cite[Prop 4.2.1]{Bosch90}}]\label{prop_invariant_form_from_e}
    Let $e \colon S \to G$ be the unit section of $G$. Then, for every $\omega_0 \in H^0(S,\epsilon^*\Omega^i_{G/S})$, there exists a unique left-invariant differential form $\omega \in H^0(G,\Omega^i_{G/S})$ such that $e^*\omega = \omega_0$.
\end{prop}

\begin{lemma}\label{lem_invdiff_grp_morph}
    Let $G$ be a commutative group scheme and let $\omega \in H^0(G,\Omega^1_{G/S})$ be a
    differential form. Then $\omega$ is invariant if and only if $\omega$ satisfies
    \[
    m^*\omega = \pr_1^*\omega + \pr_2^*\omega \in H^0(G \times_S G,\Omega^1_{G\times_S G/S}),
    \]
    where $m \colon G \to G$ is the group law of $G$. Equivalently, $\omega$ is invariant if and only if it defines a morphism of presheaves of abelian groups on $\Sch_S$
    \[
    \omega \colon G \to \Omega^1_{(-)/S}.
    \]
\end{lemma}
\begin{proof}
    We analyze left-invariance and right-invariance for the universal point $\id \colon G \to G$ more carefully. Recall that $H^0(G \times_S G, \Omega^1_{G\times_S G/S})$ decomposes as a direct sum \[\pr_1^* H^0(G,\Omega^1_{G/S}) \oplus \pr_2^* H^0(G,\Omega^1_{G/S}).\]
    Left translation with respect to the universal point sits in the following commutative diagram
    \[
    \begin{tikzcd}
    L_\id \colon G \times_S G \ar[r , "\Delta \times \id"] \ar[dr, swap, "\pr_1 \times \pr_1"] &
    (G \times_S G) \times_G (G \times_S G) \ar[r, "m_G"] \ar[d, "\pr_1 \times \pr_1"] &
    G \times_S G \ar[d, "\pr_1"] \\
    & G \times_S G \ar[r, "m"] & G.  
    \end{tikzcd}
    \]
    It follows that
    \[
    L_\id^* \omega_G = L_\id \pr_1^* \omega = (\pr_1 \times \pr_1)^* m^* \omega \in H^0(G \times_S G, \Omega^1_{G \times_S G/S}).
    \]
    Hence $\omega$ is left invariant if and only if the image of $m^* \omega_G$ under the projection to the factor $\pr^*_1 H^0(G,\Omega^1_{G/S}) \subset H^0(G \times_S G, \Omega^1_{G \times_S G/S})$ is given by $\pr_1^*\omega$.
    
    Right-translation along the universal points sits in a similar diagram; more precisely, the universal right translation is obtained from the universal left translation by swapping the two factors of $G \times_S G$ and using the commutativity of $G$.
    Thus, the roles of $\pr_1$ are $\pr_2$ interchanged and we find that 
    $\omega$ is right invariant if and only if the image of $m^* \omega_G$ under the projection to the factor $\pr^*_2 H^0(G,\Omega^1_{G/S}) \subset H^0(G \times_S G, \Omega^1_{G \times_S G/S})$ is given by $\pr_2^*\omega$.
    \end{proof}

\begin{defn}
    Let $G/S$ be a commutative group scheme. Let $f \colon S \to S'$ be a morphism of schemes. We say that a differential form $\omega \in H^0(G, \Omega^1_{G/S'})$ is \textit{invariant relative to $S'$} if it defines a morphism of presheaves of abelian groups on $\Sch_S$
    \[
    G \to f^*\Omega^1_{(-)/S'},
    \]
    where $f^* \Omega^1_{(-)/S'}$ the restriction of $\Omega^1_{(-)/S'}$ to $\Sch_S$ along the forgetful functor $\Sch_S \to \Sch_{S'}$. 
    Equivalently, we say that $\omega$ is invariant if it satisfies
     \[
    m^*\omega = \pr_1^*\omega + \pr_2^*\omega \in H^0(G \times_S G,\Omega^1_{G\times_S G/S'}).
    \]
\end{defn}

\begin{lemma}\label{lem_invariant_differentials_faltings}
    Let $G/S'$ be a commutative group scheme. Let $f \colon S \to S'$ be a morphism of schemes. The natural
    map 
    \[
    H^0(G,\Omega^1_{G/S'}) \to H^0(G \times_{S'} S, \Omega^1_{G\times_{S'} S/S'})
    \]
    preserves invariant differentials.
\end{lemma}

\begin{proof}
    Recall that $G$ admits the following (partial) resolution in the category of presheaves
    \[
    \Z[G\times_{S'} G] \xrightarrow{m - \pr_1 -\pr_2} \Z[G] \to G \to 0.
    \]
    In particular, $\omega \in H^0(G,\Omega^1_{G/S'}) = \Hom(\Z[G], \Omega^1_{(-)/S'})$ is invariant if and only if
    \[
    \omega \in \ker(\Hom(\Z[G],\Omega^1_{(-)/S'}) \to \Hom(\Z[G\times G]), \Omega^1_{(-)/S'}).
    \]    
    The projection map $G_S := G\times_{S'} S \to G$ induces a morphism between the above resolutions for $G$ and $G_S$. This gives the following commutative diagram with exact rows
    \[
    \begin{tikzcd}[column sep = small]
        0 \ar[r]  & \Hom(G,\Omega^1_{(-)/S}) \ar[r] \ar[d] & \Hom(\Z[G],\Omega^1_{(-)/S})  \ar[r] \ar[d]
        & \Hom(\Z[G\times_S' G],\Omega^1_{(-)/S}) \ar[d]\\
        0 \ar[r] & \Hom(G_S,f^*\Omega^1_{(-)/S}) \ar[r]  & \Hom(\Z[G_S],f^*\Omega^1_{(-)/S})  \ar[r] 
        & \Hom(\Z[G_S \times_S G_S],f^*\Omega^1_{(-)/S})
    \end{tikzcd}
    \]
    and the claim follows.
\end{proof}

\subsection{Construction of the de Rham comparison isomorphism}

The goal of this section is to present a construction of the $p$-adic comparison isomorphism for $1$-motives in terms of rigid analytic geometry (Construction \ref{constr_1mot_formula}).

We start with the case of an abelian variety $A$ over $K$.
Before we get into the construction, we recall some useful facts about the interaction between the analyticification functor and the de Rham and \'etale realization functors.
The analytification functor
\[ X \mapsto X^\rig\]
induces a bijection $A(K) \xrightarrow{\sim} A^\rig(K)$, which gives rise to an isomorphism $T_{p}(A) \cong T_{p}(A^\rig)$. 
Moreover, the associated rigid analytic variety $A^\rig$ has a universal vectorial extension $E(A^{\rig})$, \cite{maculan2022universal}, which is isomorphic to $E(A)^\rig$, see remark below.
This allows us to identify $H_\dr(A)$ with the space of invariant continuous differentials on $E(A)^\rig = E(A^\rig)$ via the natural isomorphism 
\[T_e E(A) \xrightarrow{\sim} T_e E(A)^\rig.\]

\begin{remark}
One can show that $E(A^\rig)$ is naturally isomorphic to $E(A)^\rig$ as follows. 
Both $E(A)$ and $E(A^\rig)$
canonically identify with the moduli space of homogeneous (analytic) line bundles 
on the dual abelian variety $A'$, respectively on $A'^\rig$,
endowed with a connection, see \cite[Prop3.2.3]{MM_Uniext_74} and 
\cite[Section 3.6]{maculan2022universal}
\footnote{A homogeneous line bundle on $A'$ is a line bundle together with an isomorphism of $\clg{O}_{A'}$-modules $\mathrm{pr}^*_1 L \otimes \mathrm{pr}^*_2L \to \mu^* L$.}. 
Sending a line bundle on $A'$ to the associated line bundle on $A'^\rig$ induces a map 
$E(A)^\rig \to E(A^\rig)$.
The fact that this map is an isomorphism follows from GAGA after identifying connections on a line bundle $L$ with sections of the Atiyah extension of $L$, see \cite[Proposition A.4]{maculan2022universal}.  
\end{remark}

Recall that the universal extension $E(A)$ is the extension of $A$ by the vectorial group scheme\footnote{We adopt the convention that the vectorial group associated with a finitely generated projective module $M$ is covariant in $M$, i.e., $V(M) = \Spec (\Sym M^\vee)$.} attached to $V(\omega_{A'})$
corresponding to 
\[
\id \in \Hom_K(\omega_{A'}, \omega_{A'}) \cong H^1(A,\omega_{A'}) \cong \Ext^1(A,V(\omega_{A'})).
\]

\begin{lemma}\label{lem_qc_subgrp_A}
    Let $A$ be an abelian variety with semistable reduction. Then there exists a quasi-compact open subgroup $\ol{E}$ of $E(A)^\rig$, such that for every $\omega \in H_\dr(A)$, there exists $r\geq 0$ such that
    $p^r \omega$ is integral on $\ol{E}$, i.e.,
    \[
      p^r \omega \in H^0(\ol{E}, \Omega^{1,+}_{E(A)^\rig/K}) \subset H^0(\ol{E},\Omega^1_{E(A)^\rig/K}).
    \]
    Moreover, $\ol{E}$ is a subtorsor of $E(A)^\rig$ under a quasi-compact open subgroup $\ol{V} \subset V(\omega_{A'})$,
\end{lemma}

\begin{proof}
    Recall that $A$ has a proper integral model $P$ over $O_K$ that 
    satisfies:
    \begin{itemize}
        \item $P$  contains the N\'eron model $\clg{A}$ of $A$ as an open subscheme;
        \item The codimension of $P\setminus \clg{A}$ inside $P$ is $\geq 2$;
        \item The special fiber $P_s$ is a normal crossing divisor and $(P,P_s)$ is log smooth over $(O_K,\can)$.
        \item The formal completion $\hat{P}$ is a strict semistable model of $A^\rig$.
    \end{itemize}
    These models were constructed by Mumford \cite{mumford_deg_ab72} in the case where the identity component of $\clg{A}_0$ is a split torus and
    by Faltings-Chai in the general case \cite[Chapter III]{faltings_chai}. 
    
    Let $L_0$ be the $O_K$-module of invariant differentials on $\clg{A}'^0$
    \begin{equation*}\label{F^0-identification}
    L_0 := \Hom_{O_K}(\Lie(\clg{A}'^0), O_K) \subset \omega_{A'},
    \end{equation*}
    where $\clg{A}'$ is the N\'eron model of the dual abelian variety $A'$.
    Recall that $\clg{A}$ admits a canonical extension, $E(\clg{A})$, by the vectorial scheme
    $V(L_0)$. It is defined as the smooth $O_K$-group scheme representing rigidified $\G_m$-extensions of
    $\clg{A}'^0$, 
    see \cite[Section 5]{MM_Uniext_74}. 
    
    \textbf{Step 1:} By Lemma \ref{lem_torsor_lattice} below, there exists a $V(p^{-l} L_0)$-torsor $E$ over $P$, for some $l \geq 0$, such that 
    \begin{enumerate}[(i)]
        \item The restriction of $E$ to $\clg{A}$ is a vectorial extension of $\clg{A}$;
        \item $E \times_{O_K} K$ is isomorphic the universal vectorial extension of $A$.
    \end{enumerate}
    We can ensure the second condition by choosing the torsor $E$ in the proof of Lemma \ref{lem_torsor_lattice}, such that $E$ extends $p^l\cdot E(\clg{A})$.
    
    \textbf{Step 2:} Define $\ol{E} = \hat{E}_\eta$.
    We claim that $\ol{E}$ is a subtorsor of $E(A)^\rig$ under a quasi-compact open subgroup $\ol{V} \subset V(\omega_{A'})$. 
    Indeed, $\ol{E}$ is a $\ol{V} = V(p^{-l} L_0)_\eta$-torsor because $E$ is a $V(p^{-l} L_0)$-torsor. Moreover, 
    since $E$ is an integral model of $E(A)$, the following diagram commutes
    \begin{equation*}
    \begin{tikzcd}
        \ol{V} \ar[d] \ar[r] & \ol{E} \ar[d] \ar[r] & A^\rig \ar[d]\\
        V(\omega_{A'}) \ar[r] & E(A)^\rig \ar[r] & A^\rig.
    \end{tikzcd}
    \end{equation*}
    Lemma \ref{lem_torsor_grp} shows that $\ol{E}$ is a subgroup of $E$.

    \textbf{Step 3:} Let $\omega \in H_\dr(A)$. We show that there exists $r\geq 0$ such that $p^r \omega$ is integral on $\ol{E}$.
    First, there exists
    $r\geq 0$ such that $p^r \omega$ lifts to a unique element
    \[\omega' \in \Lie(E_{|\clg{A}})^\vee,\] 
    because $\Lie(E(A))^\vee \cong \Lie(E_{|\clg{A}})^\vee \otimes K$
    and $\Lie(E_{|\clg{A}})^\vee$ is a free finitely generated $O_K$-module.  
    Furthermore, by Proposition \ref{prop_invariant_form_from_e}, $\omega'$ extends uniquely to an invariant differential
    \[
    \omega' \in H^0(E_{|\clg{A}}, \Omega^1_{(E,E_s)/O_K})
    \]
    since $E_{|\clg{A}} \subset E$ is a group scheme.
    Finally, $\omega'$ extends uniquely to a log differential form 
    \[\omega' \in H^0(E,\Omega^1_{(E,E_s)/(O_{K},\can)})\]
    because $E_{|\clg{A}}$ has complement of codimension $\geq 2$ in $E$ and $\Omega^1_{(E,E_s)/(O_{K},\can)}$ is locally free (because $(E,E_s)$ is log smooth over $(O_K,\can)$). 
    By Lemma \ref{lem_log_forms_ssmodel}, we view $\omega'$ as an integral differenttial form 
    \[
    \omega' \in H^0(\ol{E}, \Omega^{1,+}_{E(A)^\rig/K}).
    \]
    \end{proof}

\begin{lemma}\label{lem_torsor_lattice}
    Let $B$ be a proper $O_K$-scheme with generic fiber $B_K$. Let $L$ be a finitely generated free $O_K$-module. For every $E_K \in H^1(B_K, V(L\otimes K))$, there exists $l \geq 0$ such that $E_K$ extends to a $V(p^{-l}L)$-torsor over $B$, in other words, $E_K$ is in the image of the natural map
    \[
    H^1(B,V(p^{-l} L)) \to H^1(B_K, V(L\otimes K)); \qquad E \mapsto E \times_{O_K} K.
    \]
\end{lemma}

\begin{proof}
Since $H^1(B,V(L))$ is finitely generated, there exists $l \geq 0$ such that $p^l\cdot E_K$ is in the image of the natural map 
\[
H^1(B,V(L)) \to H^1(B_K, V(L\otimes K)).
\]
Let $E$ be a preimage of $p^l \cdot E_K$.
We claim that 
the pushout $\phi_* E$ of $E$ along $\phi \colon V(L) \to V(p^{-l}L)$ extends $E_K$.
Indeed, we can represent $E$ be a \v{C}ech cocycle 
\[
f_{ij} \in V(L)(U_{i}\cap U_j) = \Hom_{O_K}(L^\vee, B_{ij})
\]
on an open cover $\clg{U} = (U_i)_i$ of $B$,
with $B_{ij} = \Gamma(U_{i}\cap U_j, \clg{O}_B)$. Since $E$ extends $p^l\cdot E_K$ the image of $f_{ij}$ under the map 
\[
\kappa_L \colon \Hom_{O_K}(L^\vee, B_{ij}) \to \Hom_{K}(L^\vee\otimes K, B_{ij} \otimes_{O_K} K) 
\]
is a \v{C}ech cocycle representing $p^l\cdot E_K$. Let $\phi \colon L \to p^{-l} L$ be the inclusion. The lemma follows now from the following observation 
\[
\kappa_L (f_{ij}) = p^l \kappa_{p^{-l}L} (\phi_* f_{ij}) := p^l\kappa_{p^{-l}L}(f_{ij} \circ \phi^\vee) \in \Hom_{O_K}(L^\vee\otimes K, B_{ij} \otimes_{O_K} K). 
\]
Indeed, let $e_1,\dots,e_r$ be a set of generators for $L$ and $f_i := p^{-l} e_i$ be a set of generators for $p^{-l}L$, then $\phi(e_i) = p^l f_i$ and the claim follows by a direct computation.
\end{proof}

\begin{lemma}\label{lem_torsor_grp}
    Let $A$ be an abeloid variety over $K$, that is, a proper smooth connected analytic group over $K$. Let $E \in H^1(A, \clg{O}_{A}^+)$ and suppose that $E$ satisfies
    \[
    \pr_1^* E + \pr_2^* E = m^* E \in H^1(A \times A, \clg{O}_{A\times A}^+)\otimes_{O_K} K \xrightarrow{\sim} H^1(A\times A,\clg{O}_{A\times A}),
    \]
    where $m$ denotes the group law of $A$.
    Then $E$ admits a unique group structure making it into an extension of $A$ by $\mathbb{B}^1$. 
\end{lemma}
\begin{proof}
    \textbf{Step 1:} We claim that 
    \[
    \delta := \pr_1^* E + \pr_2^* E - m^*E = 0 \in H^1(A \times A, \clg{O}_{A\times A}^+).
    \]
    By assumption, there exists $k \geq 0$ such that $p^k \delta = 0$. Thus, it suffices to show that
    multiplication by $p$ is injective on $H^1(A \times A, \clg{O}_{A\times A}^+)$.
    Indeed, applying $\RS(A,-)$ to the short exact sequence
    \[
    0\to \clg{O}_{A\times A}^+ \xrightarrow{p} \clg{O}_{A\times A}^+ \to \clg{O}_{A\times A}^+/p \to 0 
    \]
    gives a long exact sequence
    \[
    0 \to O_K \xrightarrow{p} O_K \to O_K/p \to H^1(A\times A,\clg{O}_{A\times A}^+) \xrightarrow{p} 
     H^1(A\times A,\clg{O}_{A\times A}^+) \to  H^1(A\times A,\clg{O}_{A\times A}^+/p) \to \cdots,
    \]
    where we used that $A\times A$ is proper. The claim follows because $O_K \to O_K/p$ is surjective.

    \noindent \textbf{Step 2:}
    The rest of the proof is an adaptation of the argument in \cite[Chapter VII. Theorem 5]{SerreAGCF}. Let $E'$ be the $\mathbb{B}^1$-torsor representing $m^*E$. 
    By assumption, $E'$ is isomorphic to $E\times E$ with torsor structure induced by the addition $\mathbb{B}^1\times \mathbb{B}^1 \to \mathbb{B}^1$. 
    Let $g$ be the composition $E \times E \to E' \to E$. Then, we have a commutative diagram
    \[
    \begin{tikzcd}
        E\times E \ar[r, "g"] \ar[d] &E \ar[d]\\
        A \times A \ar[r] & A.
    \end{tikzcd}
    \]
    One can easily verify the identity
    \[
    g(bc, b'c') = (b+b')g(c,c') 
    \]
    for $c,c' \in E(T)$ and $b,b' \in \mathbb{B}^1(T)$ for all $K$-varieties $T$. Choose a $K$-point $e \in E(K)$ projecting to $e \in A(K)$. After translating $e$ by some $K$-point of $\mathbb{B}^1$ we may assume that $g(e,e) =e$. We claim that 
    \[
    g \colon E\times E \to E
    \]
    is the desired group structure. This follows exactly as in loc. cit. using that all analytic functions $A \to \mathbb{B}^1$ are constant.
\end{proof}

\begin{construction}\label{rigid_construction_abelian}
Let $A$ be an abelian variety with semistable reduction.
Choose an integral $O_K$-lattice $H \subset H_\dr(A)$, in the sense that
\[
H \subset H^0(\ol{E},\Omega^{1,+}_{E(A)^\rig/K}).
\]
For example, if $H'$ is any lattice in $H_\dr(A)$, then, by Lemma \ref{lem_qc_subgrp_A}, there exists $r \geq 0$ such that $p^r H'$ is integral.

We are going to construct an $O_K$-linear map
\[
 H \to \Hom_{\Z_p[G_K]}(T_p(A),B_\dr^+),\quad \omega \mapsto c_\omega.
\]
Fix $\omega \in H$. Let $x:=(x_n)_n \in T_p(A)$ and choose lifts 
$\hat{x}_n \in \ol{E}(\C_p)$ of $x_n$ with $\hat{x}_0 = 0$. Roughly speaking, we will define $c_\omega(x)$ as the pullback of $\omega$ along these lifts.

More precisely, the restriction of $\omega$ to $\ol{V}$ is of the form $df_\omega$ for a unique linear form 
\[f_\omega \colon \ol{V} \to \G_a,\]
by the invariance of $\omega$.
Moreover, there exists $\lambda_n \in \ol{V}(\C_p)$ such that $p \hat{x}_{n} = \lambda_{n} + \hat{x}_{n-1}$ for $n \geq 1$. 
It follows that the elements $\hat{x}_n^* \omega \in \Omega^1_{O_{\C_p}/O_K}$ satisfy the relation \footnote{The invariance of $\omega$ implies that $[p]^* \omega = p \omega$, where $[p]$ denotes the multiplication by $p$ map.}
\[
p \hat{x}_{n}^* \omega = df_\omega(\lambda_{n}) + \hat{x}_{n-1}^* \omega
\]
for all $n\geq 1$.
By induction, we have
\[
p^{n} \hat{x}_{n}^* \omega= df_\omega(\lambda_1) +p df_\omega(\lambda_2) + \cdots + p^{n-1} df_\omega(\lambda_{n}) = df_\omega(s_n),
\]
where $s_n :=  f_\omega(\lambda_1) + p f_\omega(\lambda_2) + \cdots + p^{n-1} f_\omega(\lambda_{n}) \in O_{\C_p}$.
Since $\omega$ comes from a log form on a log smooth model of $\ol{E}$, the pullback $\hat{x_n}^*\omega$ naturally lives in $H^1(L\hat{\Omega}^\bullet_{(O_{\C_p},\can)/(O_K,\can)})$. Therefore, the pair 
\[(\hat{x}_n^* \omega, -s_n) \in \Omega^1_{O_{\C_p}/O_K}\oplus O_{\C_p}\] 
defines an element of
\[
A_\dr \otimes^L \Z/p^n = H^0(L\hat{\Omega}^\bullet_{(O_{\C_p},\can)/(O_K,\can)} \otimes^L \Z/p^n).
\]
To see this recall that $L\hat{\Omega}^\bullet_{(O_{\C_p},\can)/(O_K,\can)} \otimes^L \Z/p^n$ is represented by the complex
\[
L\hat{\Omega}^\bullet_{(O_{\C_p},\can)/(O_K,\can)}[1] \oplus L\hat{\Omega}^\bullet_{(O_{\C_p},\can)/(O_K,\can)}
\]
with differential in degree $m$ given by
\[
\begin{pmatrix}
    -d_{m+1} & 0 \\
    p^n & d_m
\end{pmatrix},
\]
where $d_\bullet$ is the differential of $L\hat{\Omega}^\bullet_{(O_{\C_p},\can)/(O_K,\can)}$.
Finally, we define $c_\omega(x)$ as the sequence
\[ c_{\omega}(x) := ([\hat{x}_n^* \omega, -s_n])_n \in \left( \varprojlim_n A_\dr \otimes^L \Z/p^n \right) \left[\frac{1}{p} \right] =  B_\dr^+,\]
which is well-defined because
\[
s_{n-1} - s_{n} = -p^{n-1} f_\omega(\lambda_n)\quad  \text{and}\quad  \hat{x}_n^* \omega - \hat{x}_{n-1}^* \omega = df_\omega(\lambda_n).
\]
\end{construction}

\begin{prop}
    The element $c_{\omega}(x) \in B_\dr^+$ is independent of the choice of lifts $\hat{x}_n$ and the choice of $\ol{E}$. In particular, the map $c_\omega$ defines a morphism of $G_K$-modules.
    Moreover, the resulting map
    \[
    c_\dr \colon H_\dr(A) \otimes B_\dr \cong H \otimes_{O_K} B_\dr \to \Hom_{G_K}(T_p(A), B_\dr) \cong H^1_\et(A_{\ol{K}},\Q_p) \otimes_{\Q_p} B_\dr
    \]
    is independent of the lattice $H$.
\end{prop}

\begin{proof}
Let $x = (x_n)_n \in T_p(A)$. We start by proving the independence of \(\omega\) of the choice of lifts \(\hat{x}_n\). Replacing \(\hat{x}_n\) with \(\lambda \hat{x}_n\) for some $\lambda \in \ol{V}$ changes the pair 
$(\hat{x}_n^* \omega', -s_n)$ to  
\[
\big(\hat{x}_n^* \omega' + df(\lambda), -s_n - p^n f(\lambda)\big) = 
(\hat{x}_n^* \omega, s_n) + \big(df(\lambda),-p^n f(\lambda)\big),
\]  
which represents the same class in \(H^0(A_\dr \otimes \mathbb{Z}/p^n)\).

We may assume that $\omega$ lifts to an integral form on $\ol{E}$. Furthermore, assume that \(\overline{E}\) is contained in a quasi-compact open torsor \(V' \to \ol{E}' \to A^\rig\) satisfying the assumptions of Lemma \ref{lem_qc_subgrp_A}.  In particular, that \(p^r \omega\) lifts to a differential \(\omega'\) on \(\ol{E}'\).
By the uniqueness in Lemma \ref{lem_qc_subgrp_A}, the restriction of \(\omega'\) to \(\overline{E}\) coincides with \(p^r\omega\). 
Since \(c_{\omega'}\) is independent of the choice of lifts \(\hat{x}_n\), we can choose lifts lying in \(\overline{E} \subset \ol{E}'\) and deduce that  
\[
c_{\omega'}(x) = p^rc_{\omega}(x).
\]  
In general, for two subtorsors \(\overline{E}\) and \(\overline{E}'\) as in the assertion, we can find a third quasi-compact subtorsor \(\overline{E}''\) containing both.

We check the additivity of $c_\omega$, the $G_K$-equivariance follows analogously. Let $x=(x_n)_n, y= (y_n)_n \in T_p(A)$, with lifts $(\hat{x}_n)_n, (\hat{y}_n)_n \in \ol{E}(\C_p)$. Then $\hat{z}_n = \hat{x}_n + \hat{y}_n$ lifts $x_n + y_n$. Since $\omega$ is an invariant form, we have
\[
\hat{z}_n^* \omega = \hat{x}_n^* \omega + \hat{y}_n^* \omega,
\]
and it follows that
\[
c_\omega(x+y) = c_\omega(x) + c_\omega(y).
\]
\end{proof}

Our next goal is to generalize the above construction to all $1$-motives. For a $1$-motive $M = [u \colon Y \to G]$, we write
$E(M) = [u^\sharp \colon Y \to G^\sharp]$ for its universal extension and $M' = [u' \colon Y' \to G']$ for its Cartier dual.
The extension $E(M)$ is an extension of $M$ by the vectorial group associated with $\omega_{G'} = \Lie(G')^\vee$, Appendix \ref{lemma_uni_ext_M}.
As in the abelian case, $E(M)^\rig = [u^\sharp \colon Y \to G^{\sharp,\rig}]$ is the universal extension of $M^\rig$; this follows from $\Ext^1(\G_m^\rig,\G_a^\rig) =0$.
We remark that it suffices to consider $1$-motives up to isogeny since the de Rham and \'etale realizations factor through the isogeny category $\clg{M}_{1}(K)\otimes\Q$.

\begin{defn}[{}{\cite[Section 4]{raynaud_1-mot}}]
    Let $M = [u \colon Y \to G]$ be a $1$-motive over $K$. We say that $M$ has \textit{semistable reduction} if the following two conditions are satisfied:
    \begin{enumerate}[i)]
        \item $Y$ is locally constant and extends to a locally constant $O_K$-group scheme. Equivalently, $Y$ corresponds to a finite unramified $G_K$-representation on $\Z^k$, where $k$ is the rank of $Y$.
        \item $G$ extends to a smooth group scheme over $O_K$, whose special fiber is a semi-abelian variety.
    \end{enumerate}
\end{defn}

In particular, Grothendieck's semistable reduction theorem for abelian varieties implies that every $1$-motive has potentially semistable reduction.
The following proposition shows that invariant differential forms $\omega \in H_\dr(M)$ are representable by integral differential forms on a sufficiently large open subgroup of $G^{\sharp,\rig}$. 

\begin{prop}\label{prop_subgrp_1mot}
    Let $M$ be a $1$-motive with semistable reduction.
    Up to replacing $M$ by an isogenous $1$-motive, there exists an open subgroup $\ol{E} \subset G^{\sharp,\rig}$ such that
    \begin{enumerate}
        \item $\ol{E}$ is a torsor over $G^\rig$ under a quasi-compact open subgroup $\ol{V}\subset V(\omega_{G'})$, 
        \item $\ol{E}(\C_p)$ contains $u^\sharp(Y(\C_p))$,
        \item For every $\omega \in H_\dr(M)$ there exists $r \geq 0$ such that $p^r\omega$ is integral on $\ol{E}$, i.e.,
        \[p^r\omega \in H^0(\ol{E},\Omega^{1,+}_{G^{\sharp,\rig}/K}) \subset H^0(\ol{E},\Omega^1_{G^{\sharp,\rig}/K}).\]
    \end{enumerate}
\end{prop}

\begin{proof}
    Let $G$ be an extension of an abelian variety $A$ with semistable reduction by a split torus $T$ of rank $s$. That is, $G$ corresponds to a morphism 
    \[u' \colon \Z^s \cong X^*(T) \to A'(K),\]
    where $A'$ is the dual abelian variety of $A$.
    
    Let $\clg{A}$ be the N\'eron model of $A$. Recall that the N\'eron model $\clg{A}'$ of $A'$ represents extensions of $\clg{A}^0$ by $\G_m$ \cite[Lemma 5.1]{MM_Uniext_74}, and let 
    \[\G^s_{m,O_K}= \clg{T} \to \clg{G} \to \clg{A}^0\]
    be the extension of $\clg{A}^0$ corresponding to the morphism $\Z^s \to \clg{A}'$ lifting $u'$. 
    By \cite{bosch_grothendieckconj}, $\clg{G}$ extends to an extension of the N\'eron model $\clg{A}$ by $\clg{T}$ if and only if the image of $u'$ is contained in $\clg{A}'^0$.
    Since the group of connected components of $\clg{A}'$ is finite, we may replace $M$ by an isogenous $1$-motive to ensure that $\clg{G}$ extends to all of $\clg{A}$. 
    
    Let $P_A$ be a proper semistable model of $A$ as in the proof of Lemma \ref{lem_qc_subgrp_A}.
    We will construct a semistable compactification of $G$. 
    First, note that $\clg{G}$ extends to a unique $\clg{T}$-torsor $\clg{H}$ over $P_A$.
    To see this, identify $\clg{G}$ with a product of line bundles and use the fact that the map 
    $\Pic(P_A) \to \Pic(\clg{A})$ is an isomorphism, 
    because $P_A$ is regular and $P_A \setminus \clg{A}$ has codimension two.
    
    The theory of toroidal embeddings provides an embedding of $\clg{T}$ into a proper variety 
    $\clg{T}_\Sigma$, associated to a rational polyhedral cone decomposition $\Sigma$ of $X_*(T)\otimes \R$, see \cite[IV.2 Theorem 2.5]{faltings_chai}. 
    By functoriality of toroidal embeddings, there exists an embedding of $\clg{H}$ into a proper $O_K$-variety $P_\Sigma$, 
    such that \[D_h:= P_\Sigma \setminus \clg{H}\] is a strict normal crossing divisor relative to $P_A$ and hence 
    $P_\Sigma$ is a semistable compactification of $G$.
    
    In particular, the map $\clg{G} \to \clg{A}$ extends to a map $P_\Sigma \to P_A$ and the action of $\clg{T}$ on $\clg{G}$ extends to an action of $\clg{T}$ on $P_\Sigma$. 
    We can summarize the situation as follows: 
    \[
    \begin{tikzcd}
    T \ar[d]\ar[r, hookrightarrow] & \clg{T} \ar[d] \ar[r, equals] & \clg{T}\ar[d] \ar[r, hookrightarrow] & \clg{T}_\Sigma \ar[d] \\
    G \ar[d] \ar[r, hookrightarrow]& \clg{G} \ar[d]\ar[r, hookrightarrow]& \clg{H} \ar[d]\ar[r, hookrightarrow] & P_\Sigma \ar[d] \\
    A \ar[r, hookrightarrow]& \clg{A} \ar[r, hookrightarrow] & P_A \ar[r,equals] & P_A.
    \end{tikzcd}
    \]
    In the next steps, we construct an open subgroup $\ol{E} \subset G^{\sharp,\rig}$ with the desired properties.
    
    \textbf{Step 1:} As in Lemma \ref{lem_qc_subgrp_A}, let $E_A$ be the $V(p^{-l} L_0)$-torsor over $P_A$ extending $E(A)$
    with $L_0 = \Hom_{O_K}(\Lie(\clg{A}'^0, O_K)$.
    Set $E_G = E_A \times_{P_A} P_\Sigma$, then by construction
    \[
    E_G \times_{P_\Sigma} G = E(G)
    \]
    is the universal vectorial extension of $G$.
    \textbf{Step 2:} 
    Let $k$ be the rank of $Y$. Let $\clg{G}'$ be the $\G_m^{k}$-torsor over $\clg{A}'^0$ defined by the morphism $u \colon Y \to \clg{A}$, and consider the free $O_K$-module 
    \[L_1 := \Hom_{O_K}(\Lie(\clg{G}'),O_K).\]
    Then, we define $E_\Sigma$ as the pushout of $E_G$ along $V(p^{-l} L_0) \to V(p^{-l}L_1)$
    \[
    \begin{tikzcd}
        V(p^{-l}L_0)  \arrow[dr, phantom, "\ulcorner"] \ar[d] \ar[r] &E_G \ar[d] \ar[r] &P_\Sigma \ar[d, equals]\\
        V(p^{-l}L_1) \ar[r] &E_\Sigma \ar[r] &P_\Sigma.
    \end{tikzcd}
    \]
    Recall that $P_\Sigma$ is a semistable compactification of $G$ and hence the Raynaud generic fiber of $P_\Sigma$ is a compactification of $G^\rig$. We define $\ol{E} \to G^\rig$ via the following pullback square
    \begin{equation*}
    \begin{tikzcd}
        \ol{V} \ar[r] \ar[d] & \ol{E} \ar[d]  \ar[r] \arrow[dr, phantom, "\lrcorner"] & G^\rig \ar[d]\\
        \ol{V} \ar[r] & E_{\Sigma,\eta} \ar[r] & \hat{P_{\Sigma}}_{\eta}.
    \end{tikzcd}
    \end{equation*}
    Using the description of $G^{\sharp}$ in \ref{lemma_uni_ext_M}, we find a second representation of $\ol{E}$ as the pushout
    \begin{equation*}
    \begin{tikzcd}
        V(p^{-l}L_0)_\eta \ar[r] \ar[d] \arrow[dr, phantom, "\ulcorner"] & (E_A)_\eta \times_{A^\rig} G^\rig \ar[d]  \ar[r] & G^\rig \ar[d, equals]\\
        V(p^{-l}L_1)_\eta \ar[r] & \ol{E} \ar[r] & G^\rig.
    \end{tikzcd}
    \end{equation*}
    It follows immediately that $\ol{E}$ is a subgroup of $G^{\sharp,\rig}$.
    
    \textbf{Step 3:} It remains to show that $p^r\omega$ is representable by an integral differential form on $\ol{E}$ for some $r \geq 0$. 
    As in the abelian case, there exists $r\geq 0$ such that $p^r\omega$
    extends to an invariant differential form $\omega'$ on the $O_K$-group scheme
    \[
    E_\Sigma \times_{P_\Sigma} \clg{G} \subset E_\Sigma.
    \]
    Further, the form $\omega'$ extends to a log form on $E_\Sigma \times_{P_\Sigma} \clg{H} \subset E_\Sigma$ (with the pullback log structure from $P_A$) using log smoothness of $E_\Sigma$ and the fact that $\clg{H} \setminus \clg{G}$ has codimension 2. 
    Thus, it remains to show that $\omega'$ extends to a log form on all of $E_\Sigma$.
    To do this, we can work locally on $P_A$. Since $\G_m$- and $\G_a$-torsors are Zariski locally trivial there exists an open cover $(U_i)_i$ of $P_A$ such that 
    \[
    P_\Sigma \times_{P_A} U_i \cong  \clg{T}_\Sigma \times_{O_K} U_i,
    \]
    and
    $E_\Sigma \to P_\Sigma$ splits over $P_\Sigma \times_{P_A} U_i$ as
    \[
    E_\Sigma \times_{P_A} U_i \cong \clg{T}_\Sigma \times_{O_K} U_i \times_{O_K} V(L_1) \to \clg{T}_\Sigma \times_{O_K} U_i.
    \]
    Under this identification the embedding $E_\Sigma \times_{P_\Sigma} \clg{H} \to E_\Sigma$ corresponds to the toroidal embedding
    \[
    \clg{T} \times_{O_K} U_i \times_{O_K} V(L_1) \to \clg{T}_\Sigma \times_{O_K} U_i \times_{O_K} V(L_1).
    \]
    Hence, we can write $\omega' \in H^0(E_\Sigma \times_{P_\Sigma} \clg{H}, \Omega^{1}_{(E_\Sigma,D)/(O_K,\can)})$ locally as 
    \[
    f_1\omega_1 + f_2 \omega_2 + f_3 \omega_3 
    \]
    for differential forms $\omega_1, \omega_2, \omega_3$ on $\clg{T}$, resp. $U_i$, resp. $V(L_1)$, and
    \begin{align*}
        f_1 &\in H^0(U_i,\clg{O}_{U_i}) \otimes_{O_K} H^0(V(L_1),\clg{O}_{V(L_1)})\\
        f_2 &\in H^0(\clg{T},\clg{O}_{\clg{T}}) \otimes_{O_K} H^0(V(L_1),\clg{O}_{V(L_1)})\\
        f_3 &\in H^0(\clg{T},\clg{O}_{\clg{T}}) \otimes_{O_K} H^0(U_i,\clg{O}_{U_i}).
    \end{align*}
    By construction, $\omega'$ is invariant under the action of $\clg{T}$, 
    thus $\omega_1$ is an invariant form on $\clg{T}$ and thus extends to a log form on $\clg{T}_\Sigma$. Further, since $f_2$ and $f_3$ are invariant under the action of $\clg{T}$ 
    it follows that $f_2 = c_2 \otimes f'_2$ and
    $f_3 = c_3 \otimes f'_3$ for $c_2,c_3 \in O_K$ and 
    $f'_2 \in H^0((V(L_1),\clg{O}_{V(L_1)})$ and
    $f'_3 \in H^0((U_i,\clg{O}_{U_i})$. Thus, $\omega'$ extends locally to
    a log differential form on $\clg{T}_\Sigma \times U_i \times V(L_1)$, for all $i$, hence $\omega'$ extends to all of $E_\Sigma$. Finally,  
    by Lemma \ref{lem_log_forms_ssmodel}, $\omega'$ gives rise to an integral form 
    \[
    \omega' \in H^0(\ol{E},\Omega^{1,+}_{G^{\sharp,\rig}/K}).
    \] 
    
    \textbf{Step 4:} Assume that $u^\sharp(Y) \nsubseteq \ol{E}$. Since $V(\omega_{A'})^\rig$ can be covered by open balls of the form $V(p^{-r'}L_1)_\eta$ for $r' \geq 0$, there exists $r_0 \geq 0$ such that 
    $u^\sharp(Y)$ is contained in the pushout $\ol{E'} \subset G^{\sharp,\rig}$
    \[
    \begin{tikzcd}
        \ol{V}=V(L_1)_\eta \ar[r] \ar[d] \arrow[dr, phantom, "\ulcorner"] & \ol{E} \ar[d]  \ar[r] & G^\rig \ar[d, equals]\\
        V(p^{-r_0}L_1)_\eta \ar[r] & \ol{E}' \ar[r] & G^\rig.
    \end{tikzcd}
    \]
    We conclude by noticing that $p^r\omega' \in H^0(\ol{E},\Omega^{1,+}_{G^{\sharp,\rig}/K})$ also extends to $\ol{E}'$.
\end{proof}

\begin{construction}\label{constr_1mot_formula}
Let $M$ be a $1$-motive with semistable reduction and let $\omega \in H_\dr(M)$. Assume 
that $Y$ is constant over $K$,
that $u \colon Y \to M$ is injective, and that Proposition \ref{prop_subgrp_1mot} applies to $M$.
Pick an $O_K$-lattice $H \subset H_\dr(M)$ of integral forms, namely,
\[
H \subset H^0(\ol{E},\Omega^{1,+}_{E(A)^\rig/K})
\]
as in Construction \ref{rigid_construction_abelian}. 
We will construct an $O_K$-linear map
\[
H \to \Hom_{G_K}(T_p(M),B_\dr^+),\quad \omega \mapsto c_\omega.
\]
Recall that the Tate module of $M$ admits the following description:
\[
T_p(M) = \varprojlim_n M[p^n] \cong \varprojlim_n \left(\left\{ x \in G(\ol{K}) \mid p^n x \in u(Y(\ol{K})) \right\}/ u(Y)(\ol{K})\right).
\]
Fix $\omega \in H$. 
Let $x:=(x_n)_n \in T_p(M)$ and choose lifts 
$\hat{x}_n \in \ol{E}(A^\rig)(\C_p)$ of $x_n$ such that $\hat{x}_0 = u^\sharp(y)$.
Then, $c_\omega(x)$ is defined analogously to the abelian case. 
Explicitly, 
\[
p^n\hat{x}_n^* \omega = d(f_\omega(\lambda_1) + pf_\omega(\lambda_2) + \cdots + p^{n-1} f_\omega(\lambda_{n})),
\]
where $p\hat{x}_{n} = \lambda_{n} + \hat{x}_{n-1}$ for $n \geq 1$ and $df_\omega$ is the restriction of $\omega$ to $\ol{V}$ \footnote{As in the abelian case $df_\omega$ corresponds to a unique linear form $f_\omega \colon \ol{V} \to \G_a$.}.
By the same argument as in Construction \ref{rigid_construction_abelian}, the sequence of pairs
$(\hat{x}_n^* \omega, -s_n)_n$,
where $s_n = f_\omega(\lambda_1) + pf_\omega(\lambda_2) + \cdots + p^{n-1} f_\omega(\lambda_{n})$, defines an element
\[
c_{\omega} (x) := ([(\hat{x}_n^* \omega,-s_n)])_n \in B^+_\dr.
\]
\end{construction}

\begin{theorem}\label{prop_cdr}
    The elements $c_{\omega}(x)$ are independent of the choices of lifts $\hat{x}_n$ and the choice of $\ol{E}$. The resulting map $T_p(M) \to B_\dr^+$ induces a functorial isomorphism
    \begin{align*}
    c_\dr \colon H_\dr(M) \otimes_{O_K} B_\dr &\to H_\et(M) \otimes_{\Z_p} B_\dr
    \end{align*}
    of filtered $G_K$-modules for all $1$-motives $M$. In particular, $c_\dr$ preserves the weight filtration of $M$.
\end{theorem}

\begin{proof}
    The fact that $c_\omega(x)$ is well-defined and functorial follows as in the abelian case using that $Y$ is constant over $K$ and hence the pullback of $\omega$ along any $y \in u^\sharp(Y)$ vanishes.
    
    We begin by proving that $c_\omega$ induces a filtered isomorphism of $G_K$-modules
    for all $1$-motives $M = [u \colon Y\to G]$ satisfying the assumptions of Construction \ref{constr_1mot_formula}. In this case, $c_\dr$ is given by
    \begin{align}\label{comp_formula}
    c_\dr \colon H \otimes_{O_K} B_\dr &\to \Hom_{G_K}(T_p(M),B_\dr)\\
    \omega\otimes \lambda &\mapsto [x \mapsto \lambda c_{\omega}(x)] \nonumber
    \end{align}
    under the identifications
    \[
    H \otimes_{O_K} K \cong H_\dr(M), \quad H_\et(M) \otimes_{\Z_p} B_\dr \cong \Hom_{G_K}(T_p(M),B_\dr).
    \]
    We start by showing that $c_\dr$ is filtered. Let $\omega \in F^1 H_\dr(M) \cap H \cong H^0(G,\Omega^1_{G/K}) \cap H$.
    It suffices to show that  $c_{\omega}(x) \in F^1 \subset B_\dr^+$.
    Going through the proof of Proposition \ref{prop_subgrp_1mot}, we find that $\omega$ is represented by an integral form $\omega$ on $\ol{E}$ 
    which is pulled back along $q \colon \ol{E} \to {G}^{\sharp,\rig}$. 
    Thus, the restriction of $\omega$ to $\ol{V}$ vanishes, and hence $c_{\omega}(x) \in F^1$ for all $x \in T_p(M)$.
    Indeed, we have 
    \[
    p^n \hat{x}^*_n \omega = 0,
    \]
    for any choice of lifts $\hat{x}_n \in \ol{E}(\C_p)$ of $x_n$ and 
    therefore $c_\omega(x) \in F^1 = \ker( B^+_\dr \to O_{\C_p})$.

    Let $\phi \colon M \to M'$ be a morphism between $1$-motives satisfying the assumptions above. Proving that $c_\dr$ is functorial with respect to $\phi$ boils down to 
    \[
    c_{\phi^*\omega}(x) = c_{\omega}(\phi(x))
    \]
    for all $\omega \in H_\dr(M')$ and all $x \in T_p(M)$, where $\phi^*\omega$ is defined as follows. 
    The map $\phi$ induces a morphism between the universal extensions
    \[\phi^\sharp \colon G^\sharp \to G'^{\sharp}\] 
    and by definition $\phi^* \omega := \phi^{\sharp,*} \omega$.
    Let $\ol{E}'$ be an open subgroup of $G'^{\sharp,\rig}$ as in Proposition \ref{prop_subgrp_1mot}. Then, we may assume that
    $\omega$ is representable by an integral form $\omega$ on $\ol{E}'$.
    Since $c_\omega$ is independent of all the choices, we can choose a quasi-compact open subgroup $\ol{E} \subset \phi^{-1}(\ol{E}^{-1})$ satisfying the assumptions of Construction \ref{constr_1mot_formula}, and lifts of
    $\phi(x_n)$ to $\ol{E}(\C_p)$, for all $n\geq 0$. Then, we can use the integral form $\phi^{\sharp,*} \omega$ together with the chosen lifts to compute $c_{\phi^*\omega}(x)$ and the claim follows.

    The second step is to extend $c_\dr$ functorially to all $1$-motives.
    Observe that the realization functors $H_\dr(-)$ and $H_\et(-,\Q_p)$ extend functorially to the isogeny category $\clg{M}_1(K) \otimes \Q$. Moreover, recall that
    every $1$-motive $M$ decomposes in $\clg{M}_1(K)\otimes \Q$ as a direct sum
    \[M \cong [u' \colon Y' \to G']\oplus [\ker(u) \to 0], \]
    where $u'$ is injective, see Lemma \ref{lem_1mot_split}. 
    Therefore, it remains to extend the comparison isomorphism to the following cases:
    \begin{enumerate}[(i)]
        \item Morphisms $[Y_1 \to 0] \to [Y_2 \to 0]$.
        \item Morphisms $[Y_1 \to G_1] \to [Y_2 \to G_2]$, where $Y_i \to G_i$ is injective for $i \in \{0,1\}$.
        \item Morphisms $[Y_1 \to G_1] \to [Y_2 \to 0]$, where $Y_1 \to G_1$ is injective.
    \end{enumerate}
    
    \textbf{Case (i)} Let $M = [Y \to 0]$ be of the form. Then the de Rham and \'etale realizations are given by 
    $H_\dr(M) = \Hom(Y, \G_a)$ and $H_\et(M) \otimes B_\dr = \Hom_{\Z_p[G_K]}(Y(\ol{K}) \otimes \Z_p, B_\dr)$. We define 
    \[
       c_\dr \colon \Hom(Y ,\G_a) \otimes B_\dr \to  \Hom_{\Z_p[G_K]}(Y(\ol{K}) \otimes \Z_p, B_\dr)
    \]
    as the natural map induced by the inclusion $\ol{K} \to B_\dr$. Clearly, $c_\dr$ is an isomorphism of $G_K$-modules and functorial in $M$. It is filtered because the Hodge filtration on $M$ is trivial and the image of $\ol{K} \to B_\dr$ is contained in $F^0$. 
   
    \textbf{Case (ii)} Let $M = [u \colon Y \to G]$ be a $1$-motive and assume that $u$ is injective.
    There exists a finite extension $K'/K$ such that $M_{K'}$ satisfies the assumptions of Construction \ref{constr_1mot_formula}. In this case, we have constructed a $G_{K'}$-linear $B_\dr$-linear filtered morphism
    \[
    c_{\dr,K'} \colon H_\dr(M_{K'}) \otimes_{K'} B_\dr \to H_\et(M_{K'}) \otimes_{\Z_p} B_\dr.
    \]
    In fact, the underlying $B_\dr$-linear morphism is independent of the
    field extension $K'$, under the identifications $H_\dr(M) \otimes_{K} B_\dr = H_\dr(M_{K'})\otimes_{K'} B_\dr$ and
    $T_p(M_{K'}) = \mathrm{Res}_{G_{K'}} T_pM$.
    Thus, we may define
    \begin{align}\label{eq_cdr}
        c_\dr \colon H_\dr(M)\otimes_{K} B_\dr \to H_\et(M) \otimes_{\Z_p} B_\dr
    \end{align}
    for a general $1$-motive by (\ref{comp_formula}) after base changing to a finite extension $K'$ (or $\ol{K}$).
    The underlying map is independent of $K'$ and a filtered isomorphism if and only if it is after base change to any finite extension of $K$.
    It remains to check that (\ref{eq_cdr}) is $G_K$-linear.
    
    The claim boils down to showing that the map
    \[
    c_\omega \colon T_p(M_{\ol{K}}) \to B_\dr
    \]
    is $G_{K}$-equivariant if $\omega \in H_\dr(M_{\ol{K}})$ integral and defined over $K$, i.e. $\omega \in \pr_0^*H_\dr(M) \subset H_\dr(M_{\ol{K}})$, where $\pr_0 \colon M_{\ol{K}} \to M$ is the projection.

    To show this, recall how $G_K$ acts on $T_p(M_{\ol{K}})$. Let $\sigma \in G_K$.
    The twist of $M_{\ol{K}}$ is defined as $\sigma^*M_{\ol{K}} := M_{\ol{K}} \times_{\ol{K},\sigma} \ol{K}$\footnote{We write $\sigma \colon \Spec \ol{K} \to \Spec \ol{K}$ for the morphism induced by $\sigma^{-1}$}. Denote the projection $\sigma^*M_{\ol{K}}\to M_{\ol{K}}$ by $\phi_\sigma$.
    Let $x \in M_{\ol{K}}(\ol{K})$. Then, the product $\sigma^*x := x \times \sigma \colon \Spec \ol{K} \to \sigma^* M_{\ol{K}}$ defines a point of $\sigma^*M_{\ol{K}}$ which maps to $\sigma(\pr_0(x))$ under the $\sigma^* M_{\ol{K}}(\ol{K}) \to M(\ol{K})$.
    Finally $\sigma$ acts on $M_{\ol{K}}(\ol{K})$ by
    \[
    x \mapsto \phi_\sigma \circ \sigma^* x.
    \]
    Now, we can compute
    \begin{align*}
    c_{\pr_0^* \omega}(\phi_\sigma \circ \sigma^* x) & =
    (\sigma^* x) \phi_\sigma^* \pr^* \omega \\
    &= (\sigma(\pr_0(x))^* \omega \\
    &=\sigma^{*}(x^* \omega) = \sigma c_\omega(x),
    \end{align*}
    where used the functoriality of $c_{\dr,\ol{K}}$ and the following commutative diagram to compute the pullbacks
    \[\begin{tikzcd}
	\Spec \ol{K} && \Spec \ol{K} && M \\
	\\
	\sigma^*M_{\ol{K}} && M_{\ol{K}}.
    \arrow[from=1-1, to=1-3, "\sigma"]
	\arrow[from=1-1, to=3-1, "\sigma^*x"]
	\arrow[from=1-1, to=1-5, bend left, "\sigma(\pr_0(x))"]
	\arrow[from=1-3, to=3-3, "x"]
	\arrow[from=1-3, to=1-5,"\pr_0(x)"]
	\arrow[from=3-1, to=3-3, "\phi_\sigma"]
	\arrow[from=3-3, to=1-5, "\pr_0"]
\end{tikzcd}\]
    
    \textbf{Case (iii)} Note that a morphism of the form $f \colon M_1 := [Y_1 \to G_1] \to M_2:=[Y_2 \to 0]$ factors uniquely through the quotient $[Y_1 \to G_2] \to [Y_1 \to 0] \xrightarrow{g} [Y_2 \to 0]$. Thus, 
    \[H_\dr(f) \colon H_\dr(M_2) \to H_\dr (M_1)\]
    has to factor through inclusion $i_\dr \colon H_\dr(Y_1) \to H_\dr(M_1)$, and the same holds for $i_\et \colon H_\et(Y_1) \to H_\et(M_1)$.
    By the functoriality established in case (ii), 
    \[
    c_\dr \colon H_\dr(M_1) \otimes B_\dr \to H_\et(M_1) \otimes B_\dr
    \]
    maps the subspace $H_\dr(Y) \otimes B_\dr$ to $H_\et(Y) \otimes B_\dr$ and we define $c_\dr(f) = i_\et\circ c_\dr(g)$.
    
    That $c_\dr$ is an isomorphism in general follows the cases of $M = \G_m, M=A$ and $M = Y$ by looking at the weight filtration and base changing to a large enough field extension $K'/K$. 
\end{proof}

\begin{remark}
    The above proof also shows that the comparison isomorphism is functorial for morphisms $M^\rig \to M'^\rig$ in $\clg{M}_1(K)^\rig$.
\end{remark}

More generally, we have the following corollary, which proves the second part of Theorem \ref{thm_1_mot_intro}.

\begin{cor}\label{cor_functoriality_cdr}
    The comparison isomorphism $c_\dr$ constructed above is functorial for morphisms $M_1^\rig \to M_2^\rig$ in $D^b_{\rom{fppf}}(\Spa(K, O_K))$.
\end{cor}

\begin{proof}
Let $f$ be a morphism
\[
f \colon M_1^\rig \to M_2^\rig
\]
in $D^b_\fppf(\Spa(K, O_K))$. By Corollary \ref{cor_morphism_rig_mot}, $f$ can be written uniquely
as the composition
\[
M_1^\rig \xrightarrow{\can_1^{-1}} (M'_1)^\rig \xrightarrow{f'} (M'_2)^{\rig} \xrightarrow{\can_2} M_2^\rig,
\]
where $\can_i \colon (M'_i)^\rig \to M_i^\rig$ is the Raynaud replacement.
Since $c_\dr$ is clearly functorial for $\can_i$ and for $f'$ (by the remark above), and because the above decomposition is unique, we may define
\[
c_\dr(f) := c_\dr(\can_2) \circ c_\dr(f') \circ c_\dr(\can_1)^{-1},
\]
using that $c_\dr(\can_i)$ is an equivalence \footnote{This follows directly from the fact that $\can$ induces isomorphisms between the \'etale realizations.}. 
\end{proof}

\begin{example}\label{ex_tate_curve}
Let $E_q$ be the Tate elliptic curve over $K$ associated to an element $q \in O_K$ with $|q|<1$.
Let $\Phi \colon \G_m^\rig \to E_q^\rig$ be the $p$-adic uniformization of $E_q$ 
\cite[Chapter 5 Theorem 3.1]{Silver94}. The Raynaud replacement of $E$ \cite{raynaud_1-mot} is given by 
\begin{align*}
M := [u \colon \Z &\to \G_m] \\
1 &\mapsto q.
\end{align*}
The uniformization $\Phi$ induces a quasi-isomorphism in the category $\clg{M}_1(K)^\rig$
\[
M^\rig \xrightarrow{\sim} E^\rig.
\]
In particular, $\Phi$ induces an isomorphism between the de Rham and $p$-adic \'etale realizations of $M$ and $E$. 
Theorem \ref{thm_1-mot_comp_formula} allows us to compute the periods of $E_q$ by computing the periods of $M$.
For this, we choose a compatible system of $p$'th roots of unity 
\[
\epsilon = (1,\zeta_1,\zeta_2,\dots)
\]
and a compatible system of $p$'th roots of $q$
\[
\gamma = (q,q_1,q_2,\dots).
\]
In particular, $\epsilon$ and $\gamma$ generate $T_p(M)$. The weight filtration shows that $T_p(M)$ is an extension of $\Z_p$ by $\Z_p(1)$
\[
\Z_p(1) \to T_p(M) \to \Z_p.
\]
We see immediately that $\sigma \in G_K$ acts on $T_p(M)$ in the basis $(\epsilon,\gamma)$ by
\begin{equation}\label{eq_galois_action}
\begin{pmatrix}
  \chi(\sigma) & \eta(\sigma) \\
  0 & 1
\end{pmatrix},
\end{equation}
where $\chi$ is the cyclotomic character and $\eta\cdot \epsilon \colon G_K \to \Z_p(1)$ is a $1$-cocycle. The class of $\eta \cdot \epsilon$ in $H^1(G_K,\Z_p(1))$ is called the Kummer torsor.

Moreover, let $E(M) = [u^\sharp \colon \Z \to G^\sharp]$ be the universal extension of $M$. 
As in Example \ref{uni-ext-G_m}, $G^\sharp$ splits canonically as
$\G_a \times \G_m$ and 
\[
u^\sharp = 1 \times u \colon \Z \to \G_a \times \G_m.
\]
The splitting $G^\sharp \cong \G_a \times \G_m$ induces a splitting of the Hodge filtration 
\[
H_\dr(M) \cong \Hom(\Z,K) \oplus F^1 \cong \Hom(\Z,K) \oplus H_\dr(\G_m) \cong \Lie(\G_a)^\vee \oplus  H_\dr(\G_m),
\]
using that $\Lie(\G_a) \cong \Hom(\Z,K)^\vee$. Let $d \log x$ be a generator of $F^1 H_\dr(M) \cong H_\dr(\G_m)$ and let $ds \in  \Lie(\G_a)^\vee$ be the differential form corresponding to the identity $\id \colon \G_a \to \G_a$.

We claim that the period pairing
\[
\langle -, - \rangle_\dr \colon H_\dr(M) \times T_p(M) \to B_\dr
\]
is given by 
\begin{align*}
    \langle d\log x, \epsilon \rangle_\dr &=\log [\epsilon] =: t, & \quad
    \langle d\log x, \gamma \rangle_\dr &= \log [\gamma],\\
    \langle ds, \epsilon \rangle_\dr &= 0, &\quad
    \langle ds, \gamma \rangle_\dr &= 1.
\end{align*}
where $\log [\epsilon] = t \in B_\dr$ is the Fontaine element (defined by $\epsilon$) and $\log[\gamma]$ is the logarithm of $\gamma \in O_{\C_p}^\flat$ as defined in Appendix \ref{constr_classical_log}. 
Since $d\log x$ and $ds$ are already
integral, we only need to choose lifts of $\epsilon$ and $\gamma$ to $G^\sharp(\C_p)$:
\begin{align*}
\hat{\epsilon} &= ((0,1),(0,\zeta_1),(0,\zeta_2),\dots) \\
\hat{\gamma} & = ((1,q),(0,q_1),(0,q_2),\dots).
\end{align*}
By Theorem \ref{thm_1-mot_comp_formula} the periods of $M$ are given by
\begin{align*}
    \langle d\log x, \epsilon \rangle_\dr &=([d\log \zeta_n,0])_n, &\quad 
    \langle d\log x, \gamma \rangle_\dr &= ([d\log q_n,1])_n,\\
    \langle ds, \epsilon \rangle_\dr &= 0, &\quad
    \langle ds, \gamma \rangle_\dr &= ([0,1])_n = 1,
\end{align*}
where $[d\log \zeta_n, 0]$ denotes the element in $A_\dr \otimes^L \Z/p^n$ represented by 
\[
(d\log \zeta_n,0) \in \Omega^1_{(O_{\C_p},\can)/(O_K,\can)}\oplus O_{\C_p},
\]
as in Constructions \ref{rigid_construction_abelian}, \ref{constr_1mot_formula}.
In particular, we see that $\langle d\log x, -\rangle_\dr \in F^1 \subset B_\dr^+$. 
Since the de Rham pairing is $G_K$-equivariant, 
(\ref{eq_galois_action}) implies that
$\sigma \in G_K$ acts on $\langle d\log x, \epsilon \rangle_\dr$ and $\langle d\log x, \gamma \rangle_\dr$ as follows
\begin{align*}
\sigma \langle d\log x, \epsilon \rangle_\dr &= \chi(\sigma) \langle d\log x, \epsilon \rangle_\dr \\
\sigma \langle d\log x, \gamma \rangle_\dr &=\langle d\log x, \epsilon \rangle_\dr + \eta(\sigma) \langle d\log x, \epsilon \rangle_\dr,
\end{align*}
where $\chi$ is the cyclotomic character and $\eta\cdot \epsilon \in H^1(G_K,\Z_p(1))$ is the Kummer class.
Assume for now that $ \langle d\log x, \epsilon \rangle_\dr \equiv t$ and $\langle d\log x, \gamma \rangle_\dr \equiv \log[\gamma]$ modulo $F^2 B_\dr^+$.
Under this assumption, the claim follows from Tate's theorem (which implies $(F^n)^{G_K}=0$) and the following computations 
\[ \frac{\langle d\log x, \epsilon \rangle_\dr - t}{t} \in (F^1)^{G_K},
\]
and 
\[ 
\langle d\log x, \gamma \rangle_\dr - \log[\gamma] \in (F^2)^{G_K}.
\]
We refer to Appendix \ref{prop_font_log} and \ref{cor_graded_log} for the computation modulo $F^2$.

To summarize, the comparison isomorphism $\rho_\dr$ in the basis $(d\log x, \epsilon)$ and $(\epsilon^\vee,\gamma^\vee)$ is given by
\[
\begin{pmatrix}
t & 0 \\
\log[\gamma] & 1
\end{pmatrix},
\]
recovering the classical formula, see for example \cite[I.4 d)]{Andre_betti}.
\end{example}

\subsection{Comparison with Beilinson's construction}\label{sect_comp_comp}

Recall that Beilinson's comparison isomorphism $\rho_\dr$ is representable in $\DA_\et(K)$, Corollary \ref{cor_adr_reprs}, and consequently extends canonically to $1$-motives over $K$.
The goal of this subsection is to prove the first part of Theorem \ref{thm_1_mot_intro}. 

\begin{theorem}\label{thm_1-mot_comp_formula}
    The comparison isomorphism $\rho_\dr$ agrees with the comparison isomorphism $c_\dr$ of Theorem \ref{prop_cdr}. 
\end{theorem}

As we have seen in the proof of Theorem \ref{prop_cdr}, the comparison isomorphism is essentially determined by the case of $1$-motives
$M = [u \colon Y \to G]$ with injective $u$. Therefore, we focus on Beilinson's comparison morphism in this case.

\begin{lemma}\label{lem_1_mot_relcohomology}
Let $M = [u \colon \Z^k \to G]$ be a $1$-motive over $\ol{K}$. Assume that $u$ is injective. Furthermore, let $S = \{x_1,\dots,x_k\} \in G(\ol{K})$ be a set of generators of $u(\Z^k)$. Then there is a natural isomorphism
\begin{align*}
H^1_\dr(G, S) \to \Hom_{\DA_\et(\ol{K})}(M, \clg{A}_\dr \otimes \Q),
\end{align*}
where $H^1_\dr(G,S)$ denotes relative de Rham cohomology of the pair $(G,S)$.
\end{lemma}
\begin{proof}
Since $\clg{A}_\dr\otimes \Q$ represents de Rham cohomology, there are natural quasi-isomorphisms
\begin{align*}
\RS_\dr(G,S) &\simeq \rom{Cone}\left(R\Hom_{\DA_\et(K)}(M(G),\clg{A}_\dr\otimes \Q) \to
R\Hom_{\DA(K)_\et}(M(S),\clg{A}_\dr\otimes \Q) \right)[1]\\
&\simeq R\Hom_{\DA_\et(K)}(\rom{Cone}(M(S) \to M(G))[-1], \clg{A}_\dr\otimes \Q).
\end{align*}
The natural map $M(G) \to G$ induces a morphism of complexes
\[
\rom{Cone}(M[S] \to M(G)) \to \rom{Cone}(\Z^k \to G) \simeq  M[1],
\]
by identifying $\Z^k$ with the motive of $S$.
Applying $R\Hom(-,\clg{A}_\dr \otimes \Q)$ gives a morphism
\[
\phi \colon \RS_\dr(G, S) \to R\Hom_{\DA(K)_\et}(M, \clg{A}_\dr \otimes \Q).
\]
That $\phi$ is an isomorphism on $H^1$ follows directly from the cases $M = G$ and $M=\Z^k$.
\end{proof}

\begin{construction}\label{constr_alg_rhodr}
We reformulate Beilinson's construction in the case of $1$-motives. Let $M = [u \colon Y \to G]$ be a $1$-motive over $K$.
Since $1$-motives are compact objects in $\DA_\et(\ol{K})$, Lemma \ref{lem_1_mot_relcohomology} above implies that
\[ 
H_\dr(M) \cong \Hom_{\DA_\et(\ol{K})}(M_{\ol{K}}, \clg{A}_\dr) \otimes \Q.
\]
We can summarize the results from Section \ref{sect-comp-iso} by saying that $\rho_\dr$ is constructed from the following diagram
\[
\begin{tikzcd}[column sep= tiny]
    \Hom_{\DA_\et(\ol{K})}(M_{\ol{K}},\clg{A}_\dr)  \arrow[rd, "\beta_n"] &  &  \Hom_{\DA_\et(\ol{K})}(M[p^n], A_\dr/p^n) \arrow[ld, "\delta_n"] \\
                    &  \Hom_{\DA_\et(\ol{K})}(M_{\ol{K}}\otimes^L \Z/p^n,\clg{A}_\dr\otimes^L \Z/p^n),
\end{tikzcd}
\]
where the $\beta_n$ is induced by functoriality (along $M \to M \otimes^L \Z/p^n$) and the $\delta_n$ is induced by morphism $A_\dr= \clg{A}_\dr(\ol{K}) \to \clg{A}_\dr$\footnote{Recall that $M[p^n]$ denotes $H^0(\rom{Cone}(M \xrightarrow{p^n} M))$}.

More precisely, the Poincar\'e Lemma implies that $\delta_n$ is invertible and by definition $\rho_\dr$ is given by
\[
\rho_\dr = \varprojlim_n(\beta_n \circ \delta_n^{-1}) \colon \Hom_{\DA_\et(\ol{K})}(M_{\ol{K}},\clg{A}_\dr) \otimes \Q \to \Hom_{\DA_\et(\ol{K})}(T_p M, B_\dr^+),
\]
where we view $T_p M$ and $B_\dr^+$ as pro-systems of constant \'etale sheaves.
Furthermore, since $M[p^n]$ and $A_\dr/p^n$ are locally constant \'etale sheaves, the right-hand side is given by 
\[\Hom_{\Z_p[G_K]}(T_pM,B_\dr^+) := \varprojlim_i \varprojlim_n \Hom_{\Z[G_K]}(M[p^n](\ol{K}),(B_\dr^+/F^i)/p^n).\]
\end{construction}

\begin{proof}[Proof of Theorem \ref{thm_1-mot_comp_formula}]  
Let $M = [u \colon Y \to G]$ be a $1$-motive.
Using the same arguments as in Theorem \ref{prop_cdr}, it suffices to show that $c_\dr$ and $\rho_\dr$ agree after base change to a finite extension $K'/K$. Hence, we may assume that $M$ has semistable reduction, $u$ is injective, $Y$ is constant, and Proposition \ref{prop_subgrp_1mot} applies to $M$. In that case, the theorem reduces to showing that
\[
\rho_\dr(\omega) = c_\dr(\omega)
\]
for all $\omega$ in some integral lattice $H \subset H_\dr(M)$ as in Construction \ref{constr_1mot_formula}.

By Construction \ref{constr_alg_rhodr}, it suffices to compute the restriction of $\beta_n \circ \delta_n^{-1}$ to the small \'etale site of the point.

Let $P_{\Sigma}$ be a semistable compactification of $G$ as in the proof of Proposition \ref{prop_subgrp_1mot}. 
Let $\omega \in H_\dr(M)$ be integral.
According to Lemma \ref{lem_inv_to_cech} below, there exists an open cover $\clg{W} = (W_i)_i$ of $P_{\Sigma,\ol{K}}$ and a Cech $1$-cocycle $(s_i^*\omega, f(\alpha_{ij}))$ representing $\omega$ with values in the relative log de Rham complex of the pair
\(
(P_{\Sigma,\ol{K}}, S),
\)
where $S$ is a set of generators of $u(Y)$.
In particular, $(s_i^*\omega, f(\alpha_{ij}))$ maps to $\omega$ under the map
\begin{align*}
\Check{C}^1(\clg{W},L\Omega^\bullet_{(P_{\Sigma,\ol{K}},\can)/(O_K,\can)}) &\to 
H^1_\dr(P_{\Sigma,\ol{K}},S)\\
&\xrightarrow{\text{\ref{lem_1_mot_relcohomology}}} \Hom_{\DA_\et(\ol{K})}(M_{\ol{K}}, \clg{A}_\dr) \otimes \Q.
\end{align*}
It follows that the restriction of $\omega \Hom_{\DA_\et(\ol{K})}(M_{\ol{K}}, \clg{A}_\dr)$ to $\Spec \ol{K}$ is given by the following morphism of complexes
\[
\begin{tikzcd}
Y(\ol{K}) \ar[r, "u"] \ar[d, "f \circ u^\sharp"] & G(\ol{K}) \ar[d, "x \mapsto \hat{x}^*\omega"] \ar[r] & 0 \ar[d] \\
L\hat{\Omega}^0_{(O_{\ol{K}},\can)/(O_K,\can)} \ar[r] & L\hat{\Omega}^1_{(O_{\ol{K}},\can)/(O_K,\can)} \ar[r] & \cdots,
\end{tikzcd}
\]
where $u^\sharp \colon Y \to P_\Sigma$ is the inclusion map and $\hat{x} = s_i \circ x$ for some $i$ such that $x \in W_i$. Note that this morphism is only additive up to homotopy (and depends on the sections $s_i$ up to homotopy). 

By functoriality $\beta_n \circ \delta_n^{-1} (\omega)$ is given by the $H^0$ of the following map of complexes in degrees $-1,0,1,\dots)$
\[
\begin{tikzcd}
Y(\ol{K}) \ar[d, "f \circ u^\sharp"] \ar[r, "p^n \oplus -u"] & Y(\ol{K}) \oplus G(\ol{K}) 
\ar[d, "f \circ u^\sharp"', "x \mapsto \hat{x}^*\omega"] 
\ar[r, "u +p^n"] & G(\ol{K}) \ar[d, "x \mapsto \hat{x}^*\omega"] \\
L\hat{\Omega}^0_{(O_{\ol{K}},\can)/(O_K,\can)} \ar[r, "p^n \oplus -d"] & 
L\hat{\Omega}^0_{(O_{\ol{K}},\can)/(O_K,\can)} \oplus L\hat{\Omega}^1_{(O_{\ol{K}},\can)/(O_K,\can)}
\ar[r] & L\hat{\Omega}^1_{(O_{\ol{K}},\can)/(O_K,\can)} \oplus \cdots
\end{tikzcd}
\]
which is exactly $c_\dr(\omega)$ modulo $p^n$.

\end{proof}

 Let $M= [Y \to G]$ be a $1$-motive with semistable reduction and with universal vectorial extension $E(M) = [u^\sharp \colon Y \to G^\sharp]$. Assume that $Y \cong \Z^r$ is constant over $K$ and that $u$ is injective.
 Let $P_\Sigma$ be a semistable compactification of $G$ as Proposition \ref{prop_subgrp_1mot}, and let $V(L_1) \to E_\Sigma \to P_\Sigma$ be the vectorial torsor constructed in Proposition \ref{prop_subgrp_1mot}.
 In particular, $P_\Sigma$ and $E_\Sigma$ are log smooth schemes over $(O_K,\can)$ and there exists an $O_K$-lattice $H \subset H_\dr(M)$ such that every $\omega \in H$ extends to a log differential form on $E_\Sigma$.

\begin{lemma}\label{lem_inv_to_cech}
With the notations as above. Let  $S = \{x_1,\dots,x_r\}$ be a set of generators of $u(Y) \subset P_{\Sigma}$.
There exists a Zariski cover $\clg{W} = (W_i)_i$ of $P_\Sigma$ and a natural $O_K$-linear map
\[
    H \to \check{C}^1(\clg{W}\times O_{\ol{K}},L\hat{\Omega}^\bullet_{(P_\Sigma \times_{O_K} O_{\ol{K}},\can)/(O_K,\can)}(S)) 
\]
inducing\footnote{By composition with the natural map $\check{C}^\bullet(G,S) \to \RS_\dr(G,S)$.} an isomorphism $H_\dr(M_{\ol{K}}) \xrightarrow{\sim} H^1_\dr(G_{\ol{K}},S)$, where $L\hat{\Omega}^\bullet_{(P_\Sigma \times_{O_K}O_{\ol{K}},\can)/(O_K,\can)}(S)$ denotes the relative derived logarithmic de Rham complex of the pair $(P_{\Sigma,\ol{K}},S) := (P_{\Sigma} \times O_{\ol{K}},S)$.
\end{lemma}
\begin{proof}
The proof is a direct generalization of the argument for abelian varieties over a field, see \cite[Theorem 2.2]{coleman_duality_deRham}. Since $P_\Sigma$ is log smooth over $(O_K,\can)$, we may identify the derived de Rham complex of $P_\Sigma$ with the usual log de Rham complex. Furthermore, 
because $Y$ is constant over $K$, the restriction of any differential form to $x \in S$ vanishes. 
Hence, the de Rham complex of $P_\Sigma$ relative to $S$ is given by
\[
I(S) = \{ f \in \clg{O}_{P_{\Sigma}} \mid f_{|S} = 0 \} \to \Omega^1_{(P_{\Sigma},\can)/(O_K,\can)} \to \Omega^2_{(P_{\Sigma},\can)/(O_K,\can)} \to \cdots.
\]

Let $\omega \in H \subset H_\dr(M)$ be an invariant differential form on $G^\sharp$. As we have seen in Proposition \ref{prop_subgrp_1mot},
we may view $\omega$ as a closed \footnote{Since invariant forms are closed.} log differential form on $E_\Sigma$.
Pick an open cover $\clg{U} = (U_i)_{i\in I}$ of $A$ such that
\[
W_i := P_\Sigma \times_A U_i \cong U_i \times T_\Sigma \quad \text{and}\quad E_{\Sigma} \times_{P_{\Sigma}} W_i \cong V \times U_i \times T_\Sigma.
\]
Set $\clg{W} = (W_i)_i$.
Moreover, we choose local sections $s_i \colon W_i \to E_{\Sigma}$ such that
\[
s_i(x_l) = u^\sharp(e_l) \in E_{\Sigma}(O_K)
\]
for $x_l= u(e_l) \in S$. We can ensure that such $s_i$'s exist by shrinking the open sets $W_i$ such that each $W_i$ contains at most one $x_l \in S$.

Furthermore, denote by $\alpha_{i,j}$ the unique maps $\alpha_{i,j} \colon W_i \cap W_j \to V$ such that $s_j + \alpha_{i,j} = s_i$.
Finally, recall that the restriction of $\omega$ to $V$ is of the form $df$ for a unique linear form $f \colon V \to \G_a$.
One now verifies that the following tuple defines a Cech cocycle  of $\Omega^\bullet_{(P_{\Sigma},\can)/(O_K,\can)}(S)$ on $\clg{W}$: 
\[
(s_i^*\omega,f \circ \alpha_{i,j}) \in \prod_{i \in I} H^0(W_i,\Omega^1_{(P_{\Sigma},\can)/(O_K,\can)})
\oplus \prod_{i,j \in I} H^0(W_i \cap W_j,I(S)).
\]
Lemma \ref{lem_invariant_differentials_faltings} implies that the above cocycle lifts to $O_{\ol{K}}$, i.e., to  a cocycle of $L\hat{\Omega}^\bullet_{(P_\Sigma \times_{O_K} O_{\ol{K}},\can)/(O_K,\can)}(S)$.

Finally, we need to prove that sending $\omega$ to $(s_i^*\omega,f \circ \alpha_{i,j})$ induces an isomorphism $H_\dr(M_{\ol{K}}) \xrightarrow{\sim} H^1_\dr(G_{\ol{K}},S)$.
It suffices to show this before base changing to $\ol{K}$.
In the case where $Y = 0$, the result follows as in \cite[Theorem 2.2]{coleman_duality_deRham}.
For the general case, we have a diagram with exact rows of the form
\[
\begin{tikzcd}
H_\dr(S) \ar[r] \ar[d, dotted, "\exists?"] & H_\dr(M) \ar[r] \ar[d] & H_\dr(G) \ar[d]\\
H^0(S, \clg{O}_S) \ar[r, "\delta"] & 
\check{H}^1(\clg{W}, \Omega^\bullet_{(P_\Sigma,\can)/(O_K,\can)}(S)) \ar[r] & \check{H}^1(\clg{W}, \Omega^\bullet_{(P_\Sigma,\can)/(O_K,\can)}).
\end{tikzcd}
\]
In other words, we have to fill in the dotted arrow with an isomorphism that makes the left square commute. 
This works as follows. Let $x_i \in S$ and set $V' := V(L_1)/V(L_0)$
\footnote{Recall that $E_\Sigma$ was obtained via pushout from the $V(L_0)$-torsor $E_G$, see \ref{prop_subgrp_1mot}. Moreover, note that the universal extension of $\gr^0_W M = [Y \to 0]$ is given by $V' \times K$.}. 
Then, any invariant form in $H_\dr(S)$ is of the form $p^{-r}dg$ for a unique linear form $g \colon V' \to \G_a$. We define 
\[
H_\dr(S) \to H^0(S,\clg{O}_S), \quad dg \mapsto [x \in S \mapsto g \circ q \circ u^\sharp(x)],
\]
where $q \colon E_\Sigma \to V'$ is the natural map.

It remains to show that the resulting diagram commutes. In fact, we show this for a suitable choice of sections $s_i$. 
Let $s_i'$ be sections of $V(L_0) \to E_G \to P_\Sigma$, see Step 1 of Proposition \ref{prop_subgrp_1mot}, and define $s_i$ as 
\[
s_i \colon P_\Sigma \xrightarrow{s_i'} E_G \to E_\Sigma.\]
Finally, we need to compute the connecting morphism $\delta$. It is given by 
\[
f \in H^0(S, \clg{O}_S) \mapsto (f(x_i) - f(x_j), 0) \in \prod_i H^0(W_i,\Omega^1_{(P_{\Sigma},\can)/(O_K,\can)}) \oplus \prod_{i,j} H^0(W_i \cap W_j, I(S)),
\]
where $x_i$ is the unique $x \in S$ contained in $W_i$ (if there is none we set $f(x_i) = 0$). The commutativity follows from a diagram chase, using that $q \circ s_i$ is constant.
\end{proof}

\newpage

\section{A new ring of motivic \(p\)-adic periods} \label{sect_motivic_period_ring}
Let $K$ be $p$-adic field.
We define a new ring of \(p\)-adic periods following a construction of Ayoub, \cite{ayoub_new_weil},
based on the category of rigid analytic motives, \cite{ayoub_rig_motives}. As a starting point, we use the equivalence
\[
\widetilde{\Rig}_* \colon  \RDA_\et(K,\Q) \rightarrow  \DA_\et(K, \Rig_* \Q) 
\]
induced by the analyticification functor $(-)^\rig \colon \Sm_K \to \SmR_K$ to construct new Weil cohomology theories for rigid analytic varieties from Weil cohomology theories for schemes.

Let $\Gamma_W$ be a Weil cohomology theory with coefficients in a ring $A$ on $K$-schemes, which is represented by a motivic spectrum $\mathbf{\Gamma}_W \in  \DA_\et(K,\Q)$.
The assignment $M \mapsto \RS(K, \Gammabf_W \otimes M)$ extends to a monoidal functor 
\[
R_W \colon \DA_\et(K,\Q) \to D(A), 
\]
called the homological realization functor of $W$ \cite[Corollaire 2.9]{ayoub_new_weil}.
Composing $R_W$ with $\Rig^*$ defines a new realization functor on rigid analytic motives
\begin{equation*}
    \RC_W \colon  \RDA_\et(K, \Q) \xrightarrow{\widetilde{\Rig}_*}  \DA_\et(K, \Rig_* \Q) \xrightarrow{R_W} \Ht(\Mod(\AC_W)),
\end{equation*}
with coefficients in the complex $\AC_W := R_W(\Rig_* \Q) \in D(\Q)$. For $X \in \Sm_K$, we have
\begin{align}\label{rig_real_eq}
\RC_W(X^\rig) &= \RS(K, \Gammabf_W \otimes M(X) \otimes \Rig_*(\Q)) \\ &\cong 
\RS(K,\Gammabf_W \otimes M(X)) \otimes_{A} \AC_W \nonumber \\
&= R_W(X) \otimes_{A} \AC_W, \nonumber
\end{align}
by \cite[Lemme 2.29]{ayoub_new_weil}.
Applying the construction to \(p\)-adic \'etale cohomology and algebraic de Rham cohomology gives complexes $\AC_p$ and $\AC_\dr$, which are connective \cite[Corollaire 3.8/3.9]{ayoub_new_weil}. 
Let $A_p:= H_0(\AC_p)$ and $A_\dr:= H_0(\AC_\dr)$.
Then, composing $\RC_W$ (for $W = p,\dr$) with 
\[
A_W \otimes_{\AC_W} - \colon \Ht(\Mod(\AC_W)) \to D(A_W),
\]
gives rise to two new Weil cohomology theories on rigid varieties with coefficients in the classical rings $A_p$ and $A_\dr$, respectively. 
We denote the resulting realization functors by
\begin{equation*}
    \RR_p \colon  \RDA_\et(K,\Q) \rightarrow \De(A_p) \quad \text{and}\quad \RR_\dr \colon  \RDA_\et(K,\Q) \rightarrow \De(A_\dr).
\end{equation*}

Recall that, overconvergent de Rham cohomology \`a la Gro{\ss}e-Kl\"onne defines a Weil cohomology theory 
$\Gamma^\dagger_\dr$ for smooth rigid varieties \cite[Exemple 2.22]{ayoub_new_weil} with corresponding realization functor denoted by $R^\dagger_\dr$. Overconvergent de Rham cohomology extends algebraic de Rham cohomology in the sense that there is a natural isomorphism 
\[
\rho_\rig \colon \Gamma_\dr\to \Gamma_\dr^\dagger \circ (-)^\rig.
\]
In this situation, \cite[Proposition 2.35]{ayoub_new_weil} gives rise to a map 
\[
\int_K \colon A_\dr \to K
\]
together with a comparison isomorphism 
\begin{equation}\label{new_dr_to_dr}
    \RR_\dr \otimes_{A_\dr, \int_K} K \xrightarrow{\sim} R_\dr^\dagger.
\end{equation}
In other words, the "new" de Rham cohomology theory recovers overconvergent de Rham cohomology.

It is well-known that $p$-adic \'etale cohomology theory does not define a Weil cohomology theory for rigid analytic varieties, and hence we cannot expect a similar result in this case. 
However, the \(p\)-adic comparison isomorphism between de Rham and \'etale cohomology denoted by
\footnote{Evaluating $\rho_\dr$ on the (homological) motive of a smooth variety recovers Beilinson's comparison isomorphism $\rho_\dr$.}
\[
\rho_\dr \colon \Gamma_\dr \otimes_K B_\dr  \xrightarrow{\sim} \Gamma_p \otimes_{\Q_p} B_\dr,
\]
induces an isomorphism between the corresponding new realization functors
\begin{equation*}
    \rho^\rig_\dr \colon \RR_\dr \otimes_{K} B_\dr \xrightarrow{\sim} \RR_p \otimes_{\Q_p} B_\dr .
\end{equation*}
In particular, the \(p\)-adic comparison isomorphism yields a natural transformation 
\begin{equation*}\label{new_to_de_rham_trans}
 \int \colon \RR_p \rightarrow \RR_p \otimes_{\Q_p} B_\dr \xrightarrow{(\rho_\dr^\rig)^{-1}} \RR_\dr \otimes_{A_\dr} B_\dr,
\end{equation*}
and evaluating on $K$ gives rise to an integration morphism of $p$-adic periods
\begin{equation*}
   \int_p \colon  A_p \rightarrow A_p \otimes_{\Q_p} B_\dr \rightarrow A_\dr \otimes_{A_\dr} B_\dr \rightarrow B_\dr.
\end{equation*}
For $X \in \Sm_K$, the comparison isomorphism $\rho_\dr^\rig$ is related to the classical $p$-adic comparison isomorphism via the integration map
\begin{equation}\label{comp_iso_simple}
\rho_\dr^* \otimes \int_p \colon H^n_\et(X)^\vee \otimes_{\Q_p} B_\dr \otimes_{\Q_p} A_p  \xrightarrow{}
H^n_\dr(X)^\vee \otimes_K B_\dr,
\end{equation}
using (\ref{rig_real_eq}), and $\rho_\dr^*$ denotes the dual of the inverse of Beilinson's comparison isomorphism $\rho_\dr \colon H^n_\dr(X) \otimes B_\dr \to H^n_\et(X)\otimes B_\dr$. 

The goal for the rest of this section is to construct elements in the image of $\int_p$. 
Recall that $A_p$ is given by 
\begin{equation*}
A_p = \RR_{p}(K) = H^0(\RS(K, \Gammabf_p \otimes \Rig_* \Q))   = \Hom_{\RDA_\et(K,\Q)}(\Q, \Rig^* \Gammabf_p)
\end{equation*}
using the definition of the realization functor and the projection formula.
Therefore, it is natural to construct elements of $A_p$ by pairing \'etale cohomology classes of a $K$-scheme $X$ with analytic motivic homology $H_{n,m}(X^\rig,\Q)$. 
More precisely, the analytic motivic homology groups are defined as
\[
H_{n,m}(X^\rig,\Q) := \Hom_{ \RDA_\et(K,\Q)}(\Q(m)[n],M(X^\rig)).
\]
Moreover, the \'etale cohomology groups of $X$ are given by
\[H^n_\et(X,\Q_p(m))  \cong \Hom_{ \DA_\et(K,\Q)}(M(X),\Gammabf_p(m)[n]).\]
It follows that $\Rig^*$ induces a pairing
\begin{align}\label{period_pairing_new}
    \langle-,-\rangle \colon  H_{n,m}(X^\rig, \Q) \otimes H^n_{\et}(X,\Q_p(m)) &\rightarrow \Hom_{ \RDA_\et(K,\Q)}(\Q, \Rig^*\Gammabf_p) =: A_p \\ 
    \alpha \otimes \beta & \mapsto \Rig^*(\alpha) \circ \beta. \nonumber
\end{align}

The corresponding morphism
\[
r_p \colon H_{n,m}(X^\rig,\Q) \to H^n_\et(X,\Q_p (m))^\vee \otimes A_p 
\]
corresponds to the cycle class map of $\RR_p$ under the identification
\begin{align*}
     \RR_p(X^\rig) &\cong R_p(X) \otimes_{\Q_p} A_p.
\end{align*}
Before we move on, let us recall the definition of the cycle class map and introduce some useful notation. 

\begin{construction}
Let $\Gamma_W$ be a Weil cohomology theory on rigid varieties with coefficients $A$. We adopt the following notations
\begin{align*}
H_{n,W}(M) &= H_n(R_W(M)) \cong \Hom_{\RDA_\et(K,\Q)}(\UN, M \otimes \Gammabf_W[-n])
\\
\quad H^n_W(M) &= H_n(R_W(M^*)) \cong \Hom_{\RDA_\et(K,\Q)}(M, \Gammabf_W[n])\\
A_W(1) &= H_0 R_W(\UN(1)).
\end{align*}
Since $\Gamma_W$ is a Weil cohomology theory, $R_W(\UN(1))$ is concentrated in degree zero and
$A(1)$ is a free $A$-module of rank $1$. 
Moreover, note that if $M = M(X)$ is the motive of a smooth quasi-compact rigid variety $X$, then
$R_W(M)$ is a perfect complex of $A$-modules and $\Gamma_W(X) \cong R\Hom_{D(A)}(R_W(X),A)$ and
\[
R_W(X) \cong R\Hom_{D(A)}(R\Hom_{D(A)}(R_W(X),A),A).
\]
It follows that $H_{n,W}(X)^\vee \cong H^n_W(X)$.

The (homological) \textit{cycle class map} 
\[
r_W \colon H_{n,m}(M) \to H_{n,W}(M) \otimes A(-m)
\]
of $\Gamma_W$ is defined as follows.
For a motivic homology class
\[
\beta \colon \Q(m)[n] \rightarrow M
\]
we define $r_W(\beta)$ as the image of $R_W(\beta)$ under the following natural isomorphism
\begin{align*}
    \Hom_{D(A)}(A(m)[n],R_W(M)) &\cong
    \Hom_{D(A)}(A, R_W(M)(-m)[-n]) \\
    &\cong H_n(R_W(M(-m))) \\
    & \cong H_{n,W}(M) \otimes_{A} A(-m) =:  H_{n,W}(M)(-m).
\end{align*}
using that $A(m)$ is free of rank $1$ over $A$.

Let $\Gamma_{W'}$ be another Weil cohomology theory with coefficients $A'$, and assume we are given a morphism of Weil cohomology theories
\[
\phi \colon \Gamma_{W'} \to \Gamma_{W}.
\]
Then, the resulting diagram
\begin{equation}\label{commutivity_regulator}
\begin{tikzcd}
  & H_{n,m}(M) \arrow[dl, "r_{W'}"'] \arrow[dr, "r_W"] & \\
H_{n,W'}(M)(-m) \arrow[rr, "\phi"] & & H_{n,W}(M)(-m)
\end{tikzcd}
\end{equation}
commutes for every $M \in \RDA_\et(K,\Q)$.
\end{construction}

\begin{lemma}
    Let $m,n \geq 0$ and fix a generator $\epsilon \in \Z_p(1)$. Let $M \in \DA_\et(K,\Q)$ and let $\alpha\in H_{n,m}(\Rig^*(M))$.
    For every $\beta \in H^n_\et(M, \Q_p(m))$ pick $\beta' \in H^n_\et(M,\Q_p)$ such that
    $\beta = \beta' \otimes \epsilon^{\otimes m}$. Then we have
    \[
        \int_p \langle \alpha,\beta \rangle = t^m \langle \rho_\rig^{-1}(r^\dagger_\dr \alpha), \beta' \rangle_\dr \in B_\dr
    \]
    where $\langle -, - \rangle_\dr$ denotes period pairing associated to $\rho^*_\dr \colon H^n_\dr(X)^\vee \otimes B_\dr \to H^n_\et(X)^\vee \otimes B_\dr$, $t \in B_\dr$ is the Fontaine element corresponding to $\epsilon$ and $r^\dagger_\dr$ is the cycle class map associated to $R^\dagger_\dr$.
\end{lemma}

\begin{proof}
Since $\DA_\et(K,\Q)$ is generated by the motives of smooth varieties, it suffices to prove the claim for $M = M(X)$ for $X \in \Sm_K$. 
Choose $\eta \in H_\dr(\G_m)$ such that
$\rho_\dr(\eta) = t \epsilon^\vee \in H^1_\et(\G_m,\Q_p) \otimes B_\dr$.
The choice of $\eta$ determines a $K$-linear isomorphism 
\[
\sigma_m \colon H_\dr^n(X)^\vee \xrightarrow{\sim} H^n_\dr(X)^\vee (-m)
\]
by identifying $H^{n}_\dr(X)^\vee (-m)$ with $H_\dr^{n+m}(X \times \G_m^{\times m})^\vee$.
The claim follows from the commutativity of the following diagram:
\begin{equation*}
\begin{tikzcd}[column sep =huge]
H_{n,m}(X^\rig,\Q) \ar[d, "\rho_\rig^{-1}\circ r^\dagger_\dr "] \ar[r, "r_p"] & H^n_\et(X,\Q_p(m))^\vee \otimes_{\Q_p} A_p \ar[d, "\id \otimes \int"] \\
H^n_\dr(X)^\vee(-m) \otimes_{\Q_p} B_\dr \ar[d, "\sigma_m"] \ar[r, "\rho_{\dr,X}^{*} \otimes \rho^{*}_{\dr,\G_m^{\wedge m}}"] &  H^n_\et(X,\Q_p(m))^\vee \otimes_{\Q_p} B_\dr \ar[d, "\cdot t^{m}"] \\
H^n_\dr(X)^\vee \otimes_{\Q_p} B_\dr \ar[r, "\rho_\dr^{*}"] &  H^n_\et(X,\Q_p)^\vee \otimes_{\Q_p} B_\dr, 
\end{tikzcd}
\end{equation*}
where we used (\ref{comp_iso_simple}), and the compatibility of $\rho_\dr^\rig$ with the cycle class maps, (\ref{commutivity_regulator}), and the K\"unneth formula. 
\end{proof}

\begin{remark}
There is a slightly different interpretation of the ring $A_p$. We know that
the \'etale realisation functor $R_p \colon \DA_\et(K) \to D(\Q_p)$ does not functorially extend to rigid motives. By (\ref{rig_real_eq}), extending scalars to $A_p$ "adds" this functoriality. For example, Let $M,M' \in \DA_\et(K,\Q)$ and let $f \colon \Rig^*M \to \Rig^*M'$ be a morphism in $\RDA_\et(K,\Q)$.
The correspoding morphism $\RR_p(\Rig^*M) \to \RR_p(\Rig^*M')$ gives rise a "period" pairing
\[
H_{n,\et}(M) \otimes H^n_\et(M') \to A_p.
\]
Elements in $A_p \setminus \Q_p$ arising in this way witness the fact that $H_{n,\et}(-)$ is not functorial for $f$.
We will see a concrete example of this phenomenon in the next section.
\end{remark}

\subsection{Motivic periods of the Kummer motive}\label{sect_kummer_mot}

We will compute the motivic periods for the family of Kummer $1$-motives $\clg{K}_a = [\Z \to \G_m]$ defined by $1 \mapsto a$ for 
$a \in K^\times$ \footnote{$\clg{K}_a$ is the pullback of the Kummer motive $\clg{K} \in \DA(\G_m)$ along the inclusion $i_a \colon \Spec K \to \G_m$.}. We often view $\clg{K}_a$ as an extension of $\UN$ by $\UN(1)$ in $\DA(K,\Q)$.
We begin by analyzing the corresponding rigid analytic motive $\clg{K}_a^\rig \in \RDA_\grm(K,\Q)$ under the equivalence
\begin{equation}\label{eq_rig_da}
\RDA_\grm(K,\Q) \xrightarrow{\sim} \DA_N(k,\Q)
\end{equation}
from \cite{vezzani_rig_monodromy} obtained by choosing $p$ as a pseudo-uniformizer of $K$. Informally, $\DA_N(k,\Q)$ is an $\infty$-category with objects given by pairs $(M,N_M)$ with $M \in \DA(k,\Q)$ and a map $N_M \colon M \to M(-1)$ which is ind-nilpotent.

As $k \subset \ol{\FF}_p$, the underlying extension
\[
\UN(1) \to \clg{K}_a^\rig \to \UN 
\]
splits in $\DA(k,\Q)$, hence $K_a^\rig$ is determined by its the monodromy operator
\[
N_a \colon \UN(1) \oplus \UN \to \UN \oplus \UN(-1),
\]
which is given by
\[\begin{pmatrix}
    0 & v(a) \\
    0 & 0
\end{pmatrix},\]
see \textit{loc. cit.} Proposition 4.21.

Furthermore, the category $\DA(k,\Q)$ admits a natural Frobenius endomorphism induced by the relative Frobenius on $k$-schemes. This Frobenius structure is compatible with the monodromy operator in the sense that "$N\phi = \frac{1}{p} \phi N$", see Section 4.4 in \textit{loc. cit.} for a precise definition. It follows that the functor $(\ref{eq_rig_da})$ has a natural enrichment
\begin{equation}\label{eq_enrichement}
\RDA_\grm(K,\Q) \to \DA_{(\phi,N)}(k,\Q),
\end{equation}
where $\DA_{(\phi,N)}(k,\Q)$ is the $\infty$-category of motivic $(\phi,N)$-modules. Informally, its compact objects are given by commutative squares of compact objects in $\DA(k,\Q)$:
\[\begin{tikzcd}
M \ar[r, "\simeq","\alpha"'] \ar[d, "N_M"] & \phi^*M \ar[d, "N_{\phi^*N}"]\\
M(-1) \ar[r, "\simeq", "\alpha \frac{1}{p}"'] & \phi^*M(-1).
\end{tikzcd}\]
The upshot is the following: The rigid realization functor $R\Gamma_\rig \colon \RDA_\grm(K,\Q) \to D(K_0)$ is compatible with the $(\phi,N)$-structure coming from (\ref{eq_enrichement}).
In other words, there is an enrichment of $R\Gamma_\rig$ via (\ref{eq_enrichement}) to a functor
\begin{equation*}
\hat{R\Gamma}_\rig \colon \RDA_\grm(K,\Q) \to D_{\phi,N}(K_0)
\end{equation*}
with values in the derived category of $(\phi,N)$-modules. By \cite[Corollary 4.58]{vezzani_rig_monodromy}, there exists an equivalence of monoidal functors $R\Gamma_{\rom{HK}} \simeq \hat{R\Gamma}_\rig$, where $R\Gamma_{\rom{HK}}$ is the Hyodo-Kato realization functor.
For $X \in \Sm_K$ with semistable reduction and integral model $\XC$, it follows that the Frobnius and monodromy operator on 
\[
H^*_{\rom{HK}}(\XC_0) \cong H^* \hat{R\Gamma}_\rig(M(X^\rig)^\vee)
\]
is induced by the motivic Frobenius and monodromy operators on $\DA_{\phi,N}(k,\Q)$ by (\ref{eq_enrichement}).

To describe the Frobenius endomorphism of $\clg{K}_a^\rig$, we choose a branch of the $p$-adic logarithm
\[
\log_K \colon K^\times \to K
\]
with $\log_K(p) = 0$. There exists a Frobenius equivariant splitting $\clg{K}_a^\rig \cong \UN(1) \oplus \UN$ in $\DA(k,\Q)$, in other words, there exists a Frobenius equivariant section $\sigma \colon \UN \to K_a^\rig$ and under the induced splitting $\phi$ acts on $\clg{K}_a^\rig \cong \UN(1) \oplus \UN$ by
\[
\begin{pmatrix}
    \frac{1}{p} & 0\\
    0 & 1
\end{pmatrix}.
\]

The induced splitting of $H_\dr(\clg{K}_a)$ does not agree with the canonical splitting of the Hodge filtration. Indeed, by \cite[Example 2.9,2.10]{Deligne_fundamenta_group_proj} and \cite[Example 11.5]{ancona2024algebraicclasses}, Frobnius acts on $T_\dr(M)$ in the dual basis of  
$(d\log x,ds)$ \footnote{coming from the splitting of the Hodge filtration on $H_\dr(\clg{K}_a) \cong H_\dr(\G_m) \oplus H_\dr(\Z)$.} as:
\[\frac{1}{p}
\begin{pmatrix}
    1 & \log_K (a^{p-1}) \\
    0 & p
\end{pmatrix}.
\]

Moreover, the Tate module of $\clg{K}_a$ is generated by $\epsilon$ and $\alpha = (a,a^{1/p}, a^{1/p^2},\dots) \in O_{\C_p}^\flat$.
And by the same computation as in Example \ref{ex_tate_curve}, the comparison isomorphism $\rho_\dr \colon H_\dr(\clg{K}_a) \otimes B_\dr \to H_\et(\clg{K}_a)\otimes B_\dr$ in the above bases is given by
\[
\begin{pmatrix}
  t & 0 \\
  \log[\alpha] & 1
\end{pmatrix}.
\]

\textbf{Case 1:} $v(a) \neq 0$. As we have seen in Example \ref{ex_tate_curve}, the Kummer motive  $\clg{K}_a$ is the Raynaud replacement of the Tate elliptic curve $E_q$ corresponding to $q = a$.
In particular, we have 
\[
\Rig^* M(E_a) = \UN \oplus \clg{K}_a^\rig[1] \oplus \UN(1)[2]
\]
by \cite[Proposition 4.21]{vezzani_rig_monodromy}.
Thus, the analytic motivic homology groups of $E_a$ are given by:
\[
H_{1,m}(E_a^\rig) = \begin{cases}
    \Q \cdot f \; \text{if $m=1$} \\
    0 \; \text{otherwise},
\end{cases}
\]
where $f$ denotes the canonical map $\UN(1)[1] \to \clg{K}_a[1]$, and we used that the extension $\clg{K}_a$ does not split in $\DA_{N}(k,\Q)$, and that $\Hom(\UN,\UN(1)[1]) = k^\times \otimes \Q = 0$.

It follows that
\[
\rho_\rig^{-1}r^\dagger_\dr(f) = (d \log x)^\vee \in T_\dr(\clg{K}_a).
\]
Moreover, $\rho_\dr^* \colon T_\dr(\clg{K}_a) \to T_p(\clg{K}_a)$ with respect to the dual basis $(v_1,v_2)$ of $(d\log x, ds) \subset H_\dr(\clg{K}_a)$, and the basis $(\epsilon, \alpha)\subset T_p(\clg{K}_a) \otimes \Q_p$
is given by
\[
\begin{pmatrix}
    t^{-1} & - \log [\alpha]/t \\
    0 & 1 
\end{pmatrix}.
\]
Hence
\[
    \rho_\dr^{*}(( d \log x)^\vee ) = t^{-1} \cdot \epsilon, 
\]
and for $\beta = \beta' \otimes \epsilon \in H_\et(\clg{K}_a)(1)$ we have
\[
    \int_p \langle \rho_\rig^{-1}(r^\dagger(f)), \beta \rangle = 
    t \langle ( d \log x)^\vee, \beta' \rangle = 
     \lambda  \in \Q_p \subset B_\dr,
\]
where $\lambda \in \Q_p$ satisfies $\pr(\beta') = \lambda \epsilon^\vee \in H_\et(\G_m)$ and $\pr$ is the projection map
\begin{align*}
    \pr \colon H_\et(\clg{K}_a) &\to H_\et(\G_m).
\end{align*}

\textbf{Case 2:} $v(a) = 0$.
In this case, the extension $\UN(1) \to \clg{K}_a \to \UN$ is non-trivial in $\DA(K,\Q)$ (unless $a$ is a root of unity), but it splits in $\DA_N(k,\Q)$ and even in $\DA_{\phi, N}(k,\Q)$. 
Thus, if $a$ is not a root of unity,  we can expect to find interesting motivic periods.
In fact, 
\[
H_{1,0}(\clg{K}_a^\rig) = \Q \cdot \sigma,
\]
where $\sigma$ is a (Frobenius equivariant) section of $\clg{K}_a \to \UN$ in $D_{(\phi,N)}(k,\Q)$.
The corresponding splitting of the rigid motive $\clg{K}_a^\rig$ induces a splitting of the weight Filtration on the de Rham realization 
\[
W_{-1}T_\dr(\clg{K}_a) \oplus W^\perp \cong T_\dr(\G_m) \oplus T_\dr(\Z),
\]
via the Berthelot-Ogus comparison isomorphism between rigid cohomology and de Rham cohomology.
In general, this splitting is not the same as the canonical splitting of the Hodge Filtration, because the complementary subspace of $W^\perp$ has to be Frobenius invariant.
In order to find a generator of $W^\perp$ in terms of the basis $(v_1,v_2)$ as above, we need to diagonalize the Frobenius on $T_\dr(\clg{K}_a) \cong H^1 \hat{R\Gamma}_{\rig}(\clg{K}_a^\rig) \otimes_{K_0} K$; recall that $\phi$ acts as
\[\frac{1}{p}
\begin{pmatrix}
    1 & \log_K (a^{p-1}) \\
    0 & p
\end{pmatrix}.
\]
Thus, we find that a generator of $W^\perp$ is given by 
\[
u := \frac{\log_K a^{p-1}}{p-1} v_1 + v_2 = \log_K (a) v_1 + v_2.
\]
Since $\sigma$ is Frobnius (and $N$-equivariant), it follows that 
\[
\rho_\rig^{-1} r^\dagger_\dr (\sigma) =  \lambda u
\]
for some $\lambda \in K_0$ using that $R\Gamma^\dagger \simeq R\Gamma_\rig \otimes_{K_0} K$. Unfortunately, we don't know if this number is algebraic over $\Q$.

By the computation above, we have
\[
\rho_\dr^{-1}(u) = \frac{\log_K a - \log[\alpha]}{t}\epsilon + \alpha = - \frac{\log ([\alpha]/a)}{t} \epsilon + \alpha.
\]
In particular, we have found the motivic period
$\lambda \frac{\log([\alpha]/a)}{t} \in F^0 B_\cris \subset B_\dr$.

\textbf{Case 3:} $v(a) > 0$ revisited.
As we have seen above, the rigid analytic motive of the Tate curve $E_q$ only depends on the valuation of $q$. Thus, choosing an isomorphism $\psi \colon M(E_{aq})^\rig  \xrightarrow{\sim} M(E_{q})^\rig$ for $a \in O_K^\times$ not a root of unity, gives rise to a motivic homology class in
$H_{1,0}(M_1(E_q^\rig)^\vee \otimes M_1(E_{aq}^\rig))$. One way to compute the periods of this class is the following: $\psi$ induces an isomorphism
\[
\psi_p \colon H_\et(E_q) \otimes A_p \to H_\et(E_{aq}) \otimes A_p
\]
which is compatible with the comparison isomorphism
\[
H_\et(E_q) \otimes B_\dr \to H_\dr(E_{q}) \otimes B_\dr.
\]
It follows that the following diagram commutes
\[
\begin{tikzcd}
    H_\et(E_q) \otimes A_p \ar[d, "\rho_\dr^{-1} \otimes \int_p"] \ar[r, "\psi_p"] & H_\et(E_{aq}) \otimes A_p \ar[d, "\rho_\dr^{-1} \otimes \int_p"] \\
    H_\dr(E_q) \otimes B_\dr \ar[r, "\psi_\dr \otimes \id"] & H_\dr(E_{aq}) \otimes B_\dr.
\end{tikzcd}
\]
We can write $\psi_p$ as 
\[
\begin{pmatrix}
    \lambda_{11} & \lambda_{12} \\
    \lambda_{21} & \lambda_{22}
\end{pmatrix},
\] with respect to the generators $(\epsilon,\gamma)$ and $(\epsilon, \gamma')$, where $\gamma$ and $\gamma'$ are compatible systems of $p$'th roots of $q$, respectively, of $aq$, and
with $\lambda_{ij} \in A_p$ \footnote{These $\lambda_{ij} \in A_p$ are not necessarily unique!}. Note that $\psi_\dr$ preserves the corresponding splittings of the weight filtration on $H_\dr(M)$ and $H_\dr(M')$.
Thus, we compute using the description of $\rho_\dr^{-1}$ above
\begin{align*}
\int_p \lambda_{11} &= 1 &\quad \int_p \lambda_{12} &= 0\\
\int_p \lambda_{21} &= \log[\alpha]/t &\quad \int_p \lambda_{22} &= c_1,
\end{align*}
for some $c_1 \in \Q$ depending on $\psi$, where we used that $\gamma' = \alpha \gamma$ for $\alpha = (a,a^{1/p},a^{1/p^2},\dots)$.
In particular, the periods in this case differ from the periods in the previous case by
$\lambda \log_K(a)/t$! 

\subsection{Periods of stable hyperelliptic curves}\label{sect_hyperell}

The goal of this subsection is to compute the motivic periods of (semi-)stable hyperelliptic curves.
Let $C$ be a smooth projective, geometrically connected curve of genus $2$ over $K$ with affine equation
\[
y^2 + Q(x)y = P(x),
\]
where $Q,P \in K[x]$ are polynomials of $\deg Q \leq 3$ and $\deg P \leq 6$. Let $\scr{C}$ be a stable model of $C$ over $O_K$ and we write $C_{\ol{k}}$ for $\scr{C} \times_{O_K} \ol{k}$.

\begin{theorem}[{\cite{liu_stablecurves_g2}}]\label{thm_liu_curves}
The geometric special fiber of $\scr{C}$ can have the following forms:
\begin{enumerate}
    \item $C_{\ol{k}}$ is smooth.
    \item $C_{\ol{k}}$ is irreducible with one ordinary double point. In particular, the normalisation of $C_{\ol{k}}$ is an elliptic curve.
    \item $C_{\ol{k}}$ is irreducible with two double points.
    \item $C_{\ol{k}}$ consists of two projective lines meeting transversally in three points.
    \item $C_{\ol{k}}$ has two irreducible components:
    \begin{enumerate}[a)]
        \item Both irreducible components are smooth, and hence elliptic curves.
        \item One irreducible component is smooth.
        \item Both irreducible components are singular.
    \end{enumerate}
\end{enumerate}
\end{theorem}

We will show that the reduction type of $C$ is reflected by the Raynaud replacement of its associated $1$-motive.

\begin{theorem}
Let $C$ be a curve as above and let $M_1(C) = [0 \to \mathrm{Alb}(C)]$ be the associated $1$-motive. The Raynaud replacement $M_1'(C) = [Y \to G]$ of $M_1(C)$ can have the following form (after a finite extension of the base field):
\begin{enumerate}[(i)]
    \item In case 1. $M_1'(C) = M_1(C)$.
    \item In case 2. $M_1'(C)$ is of the form $[u \colon \Z \to G]$, where the semiabelian variety $G$ is an extension of an elliptic curve with good reduction $E$ by $\G_m$ and $\ol{u} \colon \Z \to E$ is non-zero.
    \item In case 5.a) $M_1'(C)$ is an abelian surface with good reduction whose special fiber is the product of two elliptic curves.
    \item In case 5.b) $M_1'(C)$ is an extension
    \[
    \clg{K}_a \to M_1'(C) \to E,
    \]
    where $|a| < 1$ and $E$ is an elliptic curve with good reduction.
    \item In the other cases $M_1'(C)$ is isogenous to a $1$-motive of the form $[\Z^2 \to \G_m^2]$ by $(1,0) \mapsto (q_1,a_1)$ and $(0,1) \mapsto (a_2,q_2)$ for with $|a_i| = 1$ and $|q_i| <1$ for $i = 1,2$.
\end{enumerate}
\end{theorem}

\begin{proof}
The theorem follows from the description of the uniformization of the Jacobian of $C$ in \cite[Chapter 5]{Lutkebohmert_rigid_curves} combined with the construction of the Raynaud replacement in \cite[\S 4.2]{raynaud_1-mot}. By \cite[Proposition 5.5.3, Theorem 5.5.11]{Lutkebohmert_rigid_curves}, the Jacobian of $C$ admits a uniformization by $\hat{Y} \xrightarrow{v} \hat{J}^\rig \to J^\rig$, where $\hat{Y}$ is a lattice in $\hat{J}$, and $\hat{J}$ is a torus extension
\[
\G_m^r \to \hat{J} \to B,
\]
lifting the generalized Jacobian $J_k$ of $C_k = \scr{C}\times_{O_K} k$ to characteristic zero.
The generalized Jacobian of $C_k$ is a semi-abelian variety
\[
\G_{m,k}^r \to J_k \to B_k := \prod_j \Pic^0 (\Tilde{C}_j),
\]
where the $C_j$'s are the irreducible components of $C_k = \scr{C}\times_{O_K} k$ and $\Tilde{C}_j$ is the normalization of $C_j$. 

In particular, $J_k$ corresponds to a morphism $\phi_k: H_1(C_k,\Z) \to B_k'$\footnote{$H_1(C_k,\Z)$ denotes the first homology class of the coincidence graph of the irreducible components of $C_k$.}, and $\hat{J}$ corresponds to a morphism $\phi_k: H_1(C_k,\Z) \to B'$ where $B$ is an abelian variety with good reduction lifting $B_k$, and $B'$ denotes the dual of $B$.
Finally, $M'(C) = [u \colon Y \to G]$ is given by the dual of the $1$-motive 
$\hat{M} := [v \colon \hat{Y} \to \hat{J}]$.
More precisely, $G$ is the torus extension of $B'$ corresponding to $\ol{v} \colon \hat{Y} \to B$ and 
$u \colon Y = H^1(C_k,\Z) \to G$ is the canonical lift of $\phi \colon Y \to B'$ to $G$,
see \cite[\S 2.4.1]{raynaud_1-mot}. Alternatively, we can describe $u$ as the connecting homomorphism in the long exact sequence
\[
\cdots \to \Hom(\hat{M},\G_m) \to Y \xrightarrow{u} G \to \Ext^1(\hat{M},\G_m) \to \cdots
\]
associated to the short exact sequence
\[
W_{-2}\hat{M} \to \hat{M} \to [Y \to B]
\]
after identifying $G$ with $\Ext^1([Y' \to B],\G_m)$ and $Y$ with $\Hom(\G_m^r,\G_m)$.

At this point, we can deduce the first and the third claim.
For the last claim, we easily check that $H_1(C_k,\Z) \cong \Z^2$ and then deduce the claim from the discreteness of the image of $\Z^2$ in $\G_m^2$, see \cite[Theorem 5.5.11]{Lutkebohmert_rigid_curves}.

For the remaining cases, we need to analyse the Jacobian of $C_k$ more closely; Zhang has done this \cite{zhang_jacobians}. The main theorem of \textit{loc. cit.} implies that the extension $J_k$ is trivial in case 4. of Theoerem \ref{thm_liu_curves} and non-trivial in case 2., and we conclude by duality.
\end{proof}

\begin{remark}
The geometry of the whole family $\scr{C}$  plays a central role in cases (ii) and (iv), in the sense that the geometry of the family determines whether the corresponding extensions split.
\end{remark}

\begin{example}[Tate surface]
We compute the period matrix and the resulting motivic periods of the Tate surface (Case (v) above). Moreover, this computation allows us to determine the monodromy operator of the corresponding rigid motive.
Let $M = [u \colon \Z^2 \to \G_m^2]$ by $(1,0) \mapsto (q_1,a_1)$ and $(0,1) \mapsto (a_2,q_2)$ for with $|a_i| = 1$ and $|q_i| <1$ for $i = 1,2$. We assume that $|q_1| \geq |q_2|$.

The Hodge Filtration on $H_\dr(M)$ splits canonically as
\[
H_\dr(M) \cong H_\dr(\G_m^2) \oplus H_\dr(\Z^2).
\]
We choose a basis $(d\log t_1, d\log t_2, ds_1,ds_2)$ of $H_\dr(M)$ as in Example \ref{ex_tate_curve}.
Similarly, $T_p(M)$ is generated by $((\epsilon,1),(1,\epsilon),(\gamma_1,\alpha_1)(\alpha_1,\gamma_1))$
where $\alpha_i$ and $\gamma_i$ are compatible systems of $p$'th roots of $a_i$, respectively, of $q_i$.
The period matrix with respect to these bases is given by
\[
\begin{pmatrix}
t & 0 & 0 & 0 \\
0 & t & 0 & 0 \\
\log [\gamma_1] & \log [\alpha_1] & 1 & 0 \\
\log [\alpha_2] & \log [\gamma_2] & 0 & 1
\end{pmatrix}
\]
and the matrix of the comparison isomorphism $\rho_\dr^*$ is given by
\[\frac{1}{t} 
\begin{pmatrix}
   1 & 0 & -\log[\gamma_1] & -\log[\alpha_2] \\
   0 & 1 & -\log[\alpha_1] & -\log[\gamma_2] \\
   0 & 0 & t & 0 \\
   0 & 0 & 0 & t
\end{pmatrix}.
\]
Under the equivalence $\RDA_{\rom{gr}}(K,\Q) \simeq \DA(k,N)$\footnote{corresponding to $\varpi = p$}, the underlying $1$-motive $M^\rig \in \DA(k,\Q)$ splits as
\[
\UN(1)^2 \oplus \UN^2,
\]
with monodromy operator $N$ 
\[
N \colon \UN(1)^2 \oplus \UN^2 \to \UN^2 \oplus \UN(-1)^2.
\]
Under this decomposition, $N$ is given by a matrix of the form
\[
\begin{pmatrix}
    0 & 0 & n_{11} & n_{12} \\
    0 & 0 & 0 & n_{22} \\
    0 & 0 & 0 & 0\\
    0 & 0 & 0 & 0
\end{pmatrix},
\]
using that $N^2 = 0$ by the $p$-adic weight monodromy conjecture. 
By the semistable comparison theorem and the fact that the monodromy operator on $\DA_N(k,\Q)$ induces the monodromy operator on log crystalline cohomology 
$N$ \cite[Remark 4.43]{vezzani_rig_monodromy}, we can use the period matrix to determine the parameters $n_{ij}$. We know that the monodromy operator acts on the period ring $B_\stb$ as the unique $B_\cris$-linear derivation satisfying $N \log [\varpi^\flat] = 1$, where $\varpi^\flat$ is a system of compatible roots of $p$. Thus, $N \log [\gamma_i] = v(q_i)$ and hence $n_{ii} = v(q_i)$ and $n_{12} = 0$.

Finally, we have $H_{1,1}(M^\rig) \cong \Q^2$ and $H_{1,0}(M^\rig) = 0$, and all motivic periods are contained in $\Q_p$.

\end{example}

\begin{example}[Mixed Case]
    In this example, we study the periods of the $1$-motives arising from hyperelliptic curves in the cases (ii) and (iv) above. Let $M = [u \colon \Z \to G]$, where $G$ is an extension of an elliptic curve with good reduction by $\G_m$. The Hodge filtration is given by
    \[
    F^1 H_\dr(G) = H_\dr(E) \subset F^1 H_\dr(M) = H_\dr(G) \subset H_\dr(M)
    \]
    with graded pieces $H_\dr(\G_m)$ and $H_\dr(\Z)$. Pick a basis $\omega_1,\omega_2,\omega_3,\omega_4$ of $H_\dr(M)$ as follows: $\omega_1,\omega_2$ is any basis of $H_\dr(E)$, $\omega_3 \in H_\dr(G)$ lifts $d\log x \in H_\dr(\G_m)$ and $\omega_4$ lifts $ds \in H_\dr(\Z)$.
    Furthermore, choose a basis $(x,y,\epsilon,\alpha)$ for $V_p(M)$ as follows: $(x,y)$ lift generators of $T_p(E)$, $\epsilon$ is a compatible system of $p$'th roots of unity and $\alpha$ is a system of $p$'th roots of $u(1) =: \alpha_0$. For $z \in G(\ol{K})$ denote its image in $E(\ol{K})$ by $\ol{z}$. 
    Then, the period matrix of $M$ is given by
    \[
    \begin{pmatrix}
    \langle \omega_1, \ol{x} \rangle  & \langle \omega_2, \ol{x} \rangle  & \langle \omega_3, x\rangle  &0  \\
    \langle \omega_1, \ol{y} \rangle  & \langle \omega_2, \ol{y} \rangle  & \langle \omega_3, y\rangle  &0  \\
    0 & 0 & t & 0 \\
    \langle \omega_1, \alpha \rangle  & \langle \omega_2, \alpha \rangle  & \langle \omega_3 , \alpha \rangle  & 1  \\
    \end{pmatrix}
    \]
    In particular, the first 2x2 block is exactly the period matrix of $E$. Moreover, the period matrix has the same shape as in the complex case, compare \cite[Proposition 2.3]{bertolin_elliptic_int}.
    
    \textbf{Case (ii):} In this case $\ol{\alpha}_0 \neq 0$. And we can compute the periods of $M$ more explicitly if $E$ has complex multiplication over $K$.
    Then, by a theorem of Serre \cite[A.2.4]{serre_ladic_reps}, $V_p(E)$ splits as a direct sum of $1$-dim $G_K$-representations $V_p(E) \cong V_1 \oplus V_2$.
    Since $E$ has good reduction, the representations $V_i$ correspond to crystalline characters $\chi_1,\chi_2$ of $G_K$.
    Furthermore, since $V_p(E)$ has Hodge-Tate weights $(0,1)$, we may assume that $\chi_1 = \chi_1' \cdot \chi_{\rom{cyc}}$ and that $\chi_1',\chi_2$ are unramified. In particular, we can choose $\omega_1,\omega_2 \in H_\dr(E)$ such the period matrix of $E$ is of the form
    \[
    \begin{pmatrix}
        c_1 t & 0 \\
        0 & c_2 \\
    \end{pmatrix}
    \]
    for $c_i \in K^\rom{unr}\subset \ol{K}$ in the maximal unramified extension of $K$. Since the comparison isomorphism doesn't change under finite field extensions, we may choose $\omega_i \in H_\dr(E_{K'})$ such that $c_i =1$ (after extending scalars to some finite extension $K'/K$).
    We compute the periods $\langle \omega_i, \alpha \rangle$. Let $\sigma \in G_{K'}$, then
    \[
    p^n (\sigma \ol{\alpha}_n - \ol{\alpha}_n) E[p^n](\ol{K}'),
    \]
    hence there exists $\eta_1(\sigma), \eta_2(\sigma) \in \Z_p$ such that $\sigma \ol{\alpha} - \ol{\alpha} = \eta_1(\sigma) \ol{x} + \eta_2(\sigma) \ol{y}$.
    One verifies easily that $\sigma \mapsto \eta_1 \ol{x}$
    and $\sigma \mapsto \eta_2(\sigma) \ol{y}$ define classes in $H^1(G_{K'},\Q_p(1))$, respectively in $H^1(G_{K'},\Q_p)$.
    We claim that there exists $b_i \in K'^\times$ such that $\langle \omega_i, \alpha\rangle = \log([\beta_i]/b_i) \in F^1 B_\dr$, where $\beta_i \in O_{\C_p}^\flat$ is a compatible system of $p$'th roots of $b_i$.
    Indeed, by Kummer theory, there exists $\beta_i = (b_i,b_{i,1},b_{i,2},\dots) \in O_{\C_p}^\flat$ with $b_i \in K'^\times$ such that
    $\sigma \beta_i = \epsilon^{\eta_i(\sigma)} \beta_i$ and hence $\sigma \log([\beta_i]/b_i) = \log([\beta_i]/b_i) + \eta_i(\sigma) t \in B_\dr$.
    It follows that $\langle \omega_i, \alpha\rangle - \log([\beta_i]/b_i) \in (F^1 B_\dr)^{G_K'} = \{0\}$ and the resulting period matrix is given by
        \[
    \begin{pmatrix}
    t  & 0  & \langle \omega_3, x\rangle  &0  \\
   0  & 1  & \langle \omega_3, y\rangle  &0  \\
    0 & 0 & t & 0 \\
    \log([\beta_1]/b_1)  &\log([\beta_2]/b_2)  & \langle \omega_3 , \alpha \rangle  & 1  \\
    \end{pmatrix}.
    \]
    As for the Kummer motive,  if $\ol{\alpha_0} \in E(K) = \clg{E}(O_K)$ vanishes modulo $p$ then these logarithms appear in the image of $\int_p \colon A_p \to B_\dr$, see Section \ref{sect_kummer_mot}. 
    
    \textbf{Case (iv):} In this case $\ol{\alpha_0} = 0$, and $\langle\omega_i,\alpha\rangle =0$. Moreover, since $u$ factors through $\G_m \to G$, the second 2x2 block is the period matrix of the Kummer motive $\clg{K}_a$ for $a = \alpha_0$ with $v(\alpha_0) > 0$.
    
    The periods $\langle \omega_3, x\rangle, \langle \omega_3, y\rangle$ are trivial if the extension $G$ splits because in this case we can choose 
    $x,y$ such that $\langle \omega_3, x\rangle = 0 = \langle \omega_3, y\rangle$. Otherwise, they remain mysterious to us.
    We can only say that they are analogues of the integrals of differentials of the third kind in the complex case, see \cite{bertolin_elliptic_int}. 
\end{example}
\newpage

\appendix
\renewcommand{\thetheorem}{\Alph{section}.\arabic{theorem}}
\setcounter{theorem}{0} 
\section{The category of $1$-motives}\label{sect_1mot_cat}
\begin{defn}
Let $S$ be a scheme. A \textit{$1$-motive over $S$} consists of the following data:
\begin{enumerate}
    \item An $S$-group scheme Y which is \'etale locally isomorphic to the constant group scheme $\Z^r$ for some $r\geq 0$.
    \item A semi-abelian $S$-scheme $G$.
    \item An $S$-morphism $u: Y \rightarrow G$.
    \end{enumerate}
A morphism of $1$-motives is a morphism of complexes of group schemes over $S$, 
or equivalently, a morphism of complexes of representable fppf sheaves. We denote
the category of $1$-motives over $S$ by $\clg{M}_1(S)$ . 
\end{defn}

It is convenient to view $1$-motives as complexes of sheaves of abelian groups with $Y$ placed in degree $0$, and we usually write $M = [u \colon Y\to G]$. 
In particular, $M$ sits in an exact triangle 
\begin{equation*}
    G[-1] \rightarrow M \rightarrow Y \xrightarrow{+1}.
\end{equation*}
The following fact justifies that we often don't distinguish between $M$ and its image in $D^b_{\fppf}(S,\Z)$.

\begin{prop}[{{\cite[Proposition 2.3.1]{raynaud_1-mot}}}]
Let $M_1$ and $M_2$ be two 1-motives over $S$. Then any morphism $M_1 \to M_2$ in $D_{\fppf}^b(\mathrm{S, \Z})$ is uniquely determined by a morphism $M_1 \to M_2$ of $1$-motives. 
In other words, the functor from $1$-motives to the bounded derived category of fppf sheaves is fully faithful.
\end{prop}

\begin{defn}\label{def_isogenycat_1mot}
    A morphism $f \colon M_1 \to M_2$ of $1$-motives is said to be an \textit{isogeny} if there exist $g \colon M_1 \to M_2$ and $n \in \Z$ such that $fg = n \id_{M_2}$ and $gf = n \id_{M_1}$.
    The \textit{isogeny category} $\clg{M}_1(S) \otimes \Q$ of $1$-motives over $S$ is defined as the localization of $\clg{M}_1(S)$ at the class of isogenies. Equivalently, $\clg{M}_1(S)\otimes \Q$
    can be described as the category with the same objects as $\clg{M}_1(S)$ and morphisms given by
    \[
    \Hom_{\clg{M}_1(S) \otimes \Q} (M_1,M_2) = \Hom_{\clg{M}_1(S)}(M_1,M_2) \otimes \Q,
    \]
    for all $M_1,M_2 \in \clg{M}_1(S)$, see \cite[Proposition B.1.2]{viale_kahn_1_mot}.
\end{defn}

\begin{remark}
    The category of $1$-motives is additive and has all finite limits and colimits \cite[Proposition C.1.3]{viale_kahn_1_mot} but is not abelian. Howoever, the isogeny category $\clg{M}_1(S)\otimes \Q$ is an abelian category \cite[Proposition 1.2.6]{viale_kahn_1_mot}.
\end{remark}

\begin{lemma}\label{lem_1mot_split}
    Let $M = [u \colon Y \to G]$ be a $1$-motive over a field $K$. There exists a $1$-motive $M' := [u' \colon Y' \to G]$ with $u'$ injective such that $M$ decomposes in the isogeny category $\clg{M}_1(K)\otimes \Q$ as a direct sum 
    \[
    M \cong [\ker u \to 0] \oplus [Y' \to G].
    \]
\end{lemma}

\begin{proof}
Let $M = [u \colon Y \to G]$ be a $1$-motive. By replacing $M$ by the isogeneous $1$-motive $[n]\circ u \colon Y \to G$, we may assume that $Y' := Y/\ker(u)$ is an \'etale lattice
and we have a natural short exact sequence in $\clg{M}_1(K) \otimes \Q$
\[
[\ker(u) \to 0] \to M \xrightarrow{(f,\id)} [u' \colon Y' \to G],
\]
where $u'$ is the map induced by $u$.
It remains to show that $(f,\id)$ admits a section in $\clg{M}_1(K) \otimes \Q$, i.e. that there exists a morphism of \'etale group schemes $g \colon Y' \to Y$ such that $\pi \circ s = n \id_{Y'}$.
Let $K'/K$ be a finite extension such that $Y$ and $Y'$ are constant over $K'$ and let $\hat{s} \colon Y'_{K'} \to Y$ be a section of $f_{K'}$. Recall that $G_{K'/K}$ acts on $\Hom(Y_{K'}, Y_{K'})$ and a morphism $g \colon Y'_{K'} \to Y_{K'}$ descends to $Y' \to Y$ if it is invariant under the action of $G_{K'/K}$. Finally, let $s = \sum_{\sigma \in G_{K'/K}} \sigma \hat{s}$ and check that $f \circ s = [K':K] \id_{Y'}$. This follows from the fact that $f$ is $G_K$-equivariant.
\end{proof}

\begin{defn}
    Let $M$ be a $1$-motive. The \textit{weight-filtration} $W_i M$ on $M$ is defined as the following ascending 3-step filtration (in the category of $1$-motives):
\begin{align*}
    &W_i(M)= 0 \quad \text{for}\; i \leq -3, \\
    &W_{-2}(M) = T,\\
    &W_{-1}(M) = G,\\
    &W_i(M) = M \quad \text{for} \; i \geq 0.
   \end{align*}
The graded pieces of $W_iM$ are:
\begin{align*}
    \gr^W_{-2}(M) = T,\;
    \gr^W_{-1}(M) = A,\;
    \gr^W_{0}(M) = Y.
\end{align*}
\end{defn}

We fix a $p$-adic field $K$ for the rest of the section.

\begin{defn}[{}{\cite[Section 4.2]{raynaud_1-mot}}]\label{def_rig_1mot}
    A rigid $1$-motive over $K$ is a two-term complex \\
    $[Y \to G]$ of abelian sheaves on $(\Spa(K,O_K))_\fppf$
    where 
    \begin{enumerate}
        \item $Y$ is \'etale locally isomorphic to the constant sheaf $\Z^r$ for some $r \geq 0$.
        \item $G$ is representable by an analytic group, that is, an extension of an abeloid
        \footnote{An abeloid variety is a proper, smooth, and connected analytic group.} variety by a torus.
    \end{enumerate}
    A morphism of rigid $1$-motives is a morphism of complexes of abelian sheaves. We denote the category of rigid $1$-motives by $\clg{M}_1(K)^\rig$.
\end{defn}

The analytification functor $X \to X^\rig$ from varieties over $K$ to rigid varieties over $K$ induces a functor
\[
(-)^\rig \colon \clg{M}_1(K) \to \clg{M}_1(K)^\rig
\]
which is not fully faithful.
Additionally, unlike $1$-motives, rigid $1$-motives do not embed fully faithfully into the derived category of fppf sheaves on $\Spa(K, O_K)$. However, there is a class of rigid $1$-motives for which this is true.

\begin{defn}
    A $1$-motive $[u \colon Y \to G]$ over $K$ is said to be \textbf{strict} if $G$ has potentially good reduction.
\end{defn}

\begin{prop}[{}{\cite[Proposition 4.2.4]{raynaud_1-mot}}]
Let $M_1,M_2$ be two $1$-motives over $K$.
\begin{enumerate}
    \item Assume that $M_1$ is strict. Then the natural map
    \[
    \Hom_{\clg{M}_1(K)^\rig}(M_1^\rig,M_2^\rig) \to \Hom_{D^b_\fppf(\Spa(K,O_K),\Z)}(M_1^\rig,M_2^\rig)
    \]
    is a bijection.
    \item Assume that $M_1$ and $M_2$ are strict. Then every morphism $M_1^\rig \to M_2^\rig$ is algebraizable.
\end{enumerate}
\end{prop}

The following theorem summarizes the construction of the Raynaud replacement of a $1$-motive.

\begin{theorem}[{}{\cite[Theorem 4.2.2]{raynaud_1-mot}}]
Let $M$ be a $1$-motive over $K$. There exists a strict $1$-motive $M'_K$ which depends functorially on $M$ together with a canonical morphism in $\clg{M}_1(K)^\rig$
\[
\can \colon M'^\rig \to M^\rig,
\]
which becomes an isomorphism in $D^b_\fppf(\Spa(K,O_K),\Z)$.
In particular, $\can$ induces an isomorphism between the \'etale realizations.
\end{theorem}

\begin{cor}\label{cor_morphism_rig_mot}
    Let $M_1,M_2$ be two $1$-motives over $K$ with Raynaud replacements $M'_1,M'_2$. Then there is a natural bijection
    \[
    \Hom_{\clg{M}_1(K)}(M'_1,M'_2) \xrightarrow{\sim} \Hom_{D^b_\fppf(\Spa(K,O_K),\Z)}(M_1^\rig,M_2^\rig).
    \]
\end{cor}

\subsection{De Rham and \'etale realization} \label{sect-1-mot-real}
Let $S$ be a scheme and $\ell$ be a prime invertible on $S$.
We define the $\ell$-adic realization following \cite[Ch.10]{deligne_hodge_3}. 
Let $M = [u: Y \rightarrow G]$ be a 1-motive, let $n$ be an integer invertible on $S$, and let $C(M,n)$ denote the cone of multiplication by $n$ on $M$. It is given by the following complex of sheaves in degrees $-1,\;0,\;1$ 
\begin{align*}
    Y \longrightarrow &Y \oplus G \longrightarrow G \\
    x  \mapsto &(-nx , -u(x)) \\
     &(x,y) \mapsto u(x) -ny.
\end{align*}
Since multiplication is injective on $Y$ and surjective on $G$, the cohomology of $C(M,n)$ is concentrated in degree $0$. We define 
$M[n] := H^{0}(C(M,n))$\footnote{It will be always be clear from the context when $M[n]$ denotes the translation of $M$ in $D^b(\rom{fppf})$.}. There is an associated short exact sequence of locally constant sheaves on $S_\et$
\begin{equation*}
    0 \rightarrow G[n] \rightarrow M[n] \rightarrow Y[n] \rightarrow 0.
\end{equation*}
Finally, the $\ell$-adic realization of $M$ is defined as the projective system of \'etale sheaves
\begin{equation*}
    T_{\ell}(M) := "\varprojlim_n" M[\ell^n].
\end{equation*}
The weight filtration $W$ induces a filtration on $T_{\ell}(M)$ with graded pieces given by 
$T_{\ell}(T),\; T_{\ell}(A)$ and $T_{\ell}(Y)$. 

Next, we recall the construction of the de Rham realization following Deligne \cite[Ch. 10]{deligne_hodge_3}.
Let $M=[u: Y \rightarrow G]$ be a $1$-motive over a field ${K}$ of characteristic zero. The construction of the de Rham realization relies on the existence of the universal vectorial extension of a $1$-motive.

\begin{defn}
    An \textit{extension} of an $S$-1-motive $M = [u \colon Y \to G]$ by a group scheme $H$ is an extension $E$ of $G$ by $H$ together with a morphism 
    $v \colon Y \to E$ that lifts $u$. We denote the group of isomorphism classes of extensions of $M$ by $H$ as $\Ext^1(M,H)$. If $H$ is a vector group, we say that $E$ is a vectorial extension.
\end{defn}

If $W$ is a vector group, an extension of $M$ by $W$ is the same as an extension of $M$ by $W[-1]$\footnote{The shift comes from the fact that $M$ sits in degrees $0,1$.} in the category of complexes of fppf sheaves because $\G_a$-torsors are always representable.

\begin{defn}
    Let $M$ be a $1$-motive and let $\clg{E}$ be a finitely generated projective $\clg{O}_S$-module. We say a vectorial extension $E$ of $M$ by $V(\clg{E})$ is \textit{universal} if the pushout morphism
    \[
    \Hom_{\clg{O}_S}(\clg{E},\clg{F}) \to \Ext^1(M,V(\clg{F}))
    \]
    is an isomorphism for all finite generated projective $\clg{O}_S$-modules $\clg{F}$.
\end{defn}

Let $K$ be a field of characteristic zero. 
The fact that universal vectorial extensions exist is a direct consequence of the following lemma, as explained in \cite[I. 1.7]{MM_Uniext_74}.
The universal extension is unique up to canonical isomorphism.

\begin{lemma}
Let $M$ be a $1$-motive and let $\G_a$ be the additive group over $K$. Then the following holds
\begin{equation*}
    \Hom(M,\G_a[-1]) = 0, \quad \text{and} \; \Ext^1(Y,\G_a) = 0.
\end{equation*}
Moreover,
\begin{enumerate}
    \item The extensions of $M$ by $\G_a[-1]$ (as complexes) have no non-trivial automorphisms.
    \item The sequence 
    \begin{equation*}
         0 \rightarrow \Hom(Y,\G_a) \rightarrow \Ext^1(M,\G_a) \rightarrow \Ext^1(G,\G_a) \rightarrow 0
    \end{equation*}
    is exact.
    \item Since $\Hom(T,\G_a) = \Ext^1(T,\G_a) = 0$, the natural map
    \begin{equation*}
        \Ext^1(A,\G_a) \rightarrow \Ext^1(G,\G_a)
    \end{equation*}
    is an isomorphism.
\end{enumerate}
\end{lemma}
\begin{proof}
The first assertion follows from the well-known facts
\[
\Hom(\G_m,\G_a) = \Hom(A,\G_a) = 0 = H^1(Y,\clg{O}_Y).
\] 
\begin{enumerate}
    \item Follows from the fact that $G$ has connected fibres and $\G_a$ is an \'etale sheaf. 
    \item Follows by looking at the long exact sequence associated with the distinguished triangle 
    \begin{equation*}
         G[-1] \rightarrow M \rightarrow Y \rightarrow G.
    \end{equation*}
    Since $\Hom(M,\G_a[-1]) = 0$, the long exact sequence reduces to the desired short exact sequence
    \begin{equation*}
         0 \to \Hom(Y, \G_a) \to \Ext^1(M, \G_a) \to \Ext^1(G,\G_a) \to \Ext^1(Y,\G_a)=0.
    \end{equation*}
    \item Follows from the long exact sequence associated with the short exact sequence
\begin{equation*}
   0 \rightarrow T \rightarrow G \rightarrow A \rightarrow 0.
\end{equation*}
\end{enumerate} 
\end{proof}

\begin{remark}
We note that the map
\begin{equation*}
    \Hom(Y,\G_a) \rightarrow \Ext^1(M,\G_a)
\end{equation*}
sends a morphism $v \colon Y \to \G_a$ to the extension of $M$ given by
\begin{equation*}
\begin{tikzcd}
    0 \ar[d] \ar[r] & Y \ar[r, "\id"] \ar[d, "v \oplus u"] & Y \ar[d, "u"] \\
    \G_a \ar[r] &\G_a \oplus G \ar[r] &G.
\end{tikzcd}
\end{equation*}
\end{remark}

We denote the universal vectorial extension of a $1$-motive $M$ by $E(M) = [u^\natural : Y \rightarrow G^\natural]$.

\begin{lemma}\label{lemma_uni_ext_M}
    Let $M$ be a 1-motive with universal vectorial extension 
    $E(M) = [Y \to G^\sharp]$. 
    The group scheme $G^\natural$ is isomorphic to the pushout of the universal vectorial extension of $G$ along the natural map $\Ext^1(G,\G_a)^\vee \to \Ext^1(M,\G_a)^\vee$:
\begin{equation*}
\begin{tikzcd}
   0 \ar[r]& V(\Ext^1(G,\G_a)^\vee) \ar[r] \ar[d] & E(G) \ar[r] \ar[d] & G \ar[d, "\id"] \ar[r] & 0 \\
   0 \ar[r]& V(\Ext^1(M,\G_a)^\vee) \ar[r] & G^\natural \ar[r] &G \ar[r] & 0.
\end{tikzcd}
\end{equation*}
\end{lemma}

\begin{lemma}[{}{\cite[Lemma 2.4]{bertapelle2008deligne}}]\label{lemma_uni_ext_crit}
    Let $M = [u\colon Y \to G]$ be a $1$-motive. Let $G^\sharp$ be the group scheme defined in the previous lemma. A morphism $v \colon Y \to G^\sharp$ lifting $u$ is a universal extension of $M$ if and only if the induced morphism $Y \to G^\sharp \to V(\underline{\Hom}(Y,\G_a)^\vee)$ is a universal extension of $Y$. 
\end{lemma}

\begin{remark}\label{rem_uni_ext}
    The lemma implies that if $v \colon Y \to G^\sharp$ defines the vectorial universal extension of $M$, then so does $v + f \colon Y \to G^\sharp$ for every morphism of the form $f \colon Y \to V(G) \to V(M)$. 
\end{remark}

\begin{example}\label{uni-ext-G_m} We compute the universal vectorial extension in the following examples.
\begin{enumerate}
    \item $M = [Y \to 0]$: The universal extension is given by 
    \[
    \mathrm{ev} \colon Y \to V(\Hom(Y,\G_a)^\vee).
    \]
    \item $M = [u \colon Y \to \G_{m}^r]$: In this case $\Ext^1(M,\G_a) = \Hom(Y,\G_a)$, and the universal vectorial extension is given by
\begin{equation*}
    0 \to V(\Hom(Y,\G_a)^\vee) \to V(\Hom(Y,\G_a)^\vee) \times \G_m^r \to \G_m^r \to 0,
\end{equation*}
where $u^\sharp \colon Y \to V(\Hom(Y,\G_a)^\vee) \times \G_m^r$ is induced by $\rom{ev} \times u$.
\end{enumerate}

\end{example}

\begin{defn}
Let $M$ be a 1-motive. The homological \textit{de Rham realization} $T_\dr(M)$ of $M$ is defined as 
\begin{equation*}
    \Lie(G^\natural) = T_e(G^\natural).
\end{equation*}
The Hodge filtration $F$ on $T_\dr(M)$ is defined as the descending two-step filtration
\begin{equation*}
    F^{0}T_\dr(M) = \ker (\Lie G^\natural \rightarrow \Lie G) \cong \Ext^1(M,G_a)^\vee \subset F^{-1}T_\dr(M) =  T_\dr(M).
\end{equation*}
The cohomological de Rham realization of $M$ is defined as the dual of the homological de Rham realization $H_\dr(M) = T_\dr(M)^\vee \cong \rom{Inv}(G^\sharp)$. 
The dual filtration of the Hodge filtration is given by the ascending two-step filtration
\begin{equation*}
    F^1 H_\dr(M) = \ker((F^{-1}T_\dr(M))^\vee \to (F^0T_\dr(M))^\vee) \cong \Lie(G)^\vee \hookrightarrow 
    F^0H_\dr(M) = \Lie(G^\natural)^\vee.
\end{equation*}
\end{defn}

The weight filtration on $M$ induces a weight filtration on $T_\dr(M)$ which interacts with the de Rham realization and Hodge filtration as follows:
\begin{alignat*}{2}
    &T_\dr(W_{-2}(M)) = \Lie T  \qquad &&F^0\cap T_\dr(W_{-2}(M)) = 0 \\
    &T_\dr(W_{-1}(M)) = \Lie E(G) \qquad &&F^0\cap T_\dr(W_{-1}(M)) = \Ext^1(G,\G_a)^\vee \\
    &T_\dr(W_{0}(M)) = \Lie G^\natural  \qquad &&F^0\cap T_\dr(W_{0}(M)) = \Ext^1(M,\G_a)^\vee. 
\end{alignat*}

\section{Derived logarithmic de Rham cohomology}\label{appendix_de_rham}
We quickly recall Gabber's construction of the logarithmic cotangent complex and the derived logarithmic de Rham complex. We follow the notations and conventions in \cite{bhatt2012padicderivedrhamcohomology}.

\begin{defn}
    The category of \textit{prelog rings} $\LogAlg^\pre$ is the full subcategory of the category of monoid morphisms with objects given by monoid maps $\alpha \colon M \to A$, where $M$ is a monoid and $A$ is a ring regarded as a monoid via multiplication. 
    A \textit{log ring} is a prelog ring such that $\alpha \colon \alpha^{-1}(A^\times) \to A^\times$ is an isomorphism. 
\end{defn}

The forgetful functor $\LogAlg^\pre_{(M \to A)} \to \Set \times \Set$ sending $(M \to A) \to (N \to B)$ to the pair $(N,B)$ admits a left adjoint 
\[
(X,Y) \mapsto T_{(M \to A)}(X,Y) := (\N^X\sqcup M) \to A[X,Y]).
\]

\begin{defn}
    Let $f \colon (M \to A) \to (N \to B)$ be a map of prelog rings. Let $P_\bullet(f)$ be the simplicial prelog ring given in degree $0$ by 
    \[
    (N_0 \to P_0) := T_{(M \to A)}(N,B)
    \]
    and in degree $n+1$ by $T_{(M \to A)}(N_n \to P_n)$ with the usual face and degeneracy maps, see \cite[Proposition 8.6.8]{weibel_homological_algebra}.
    The counit of the free-forgetful adjunction defines an augmentation $P_\bullet(f) \to (N \to B)$, which we call the \textit{canonical free resolution} of $f$.
\end{defn}

\begin{remark}
    There exists a model structure on simplicial prelog rings such that the canonical free resolution is a cofibration replacement of $f$. In particular, everything discussed below will work with any cofibrant replacement of $f$, see \cite[Section 5]{bhatt2012padicderivedrhamcohomology} for more details.
\end{remark}

\begin{defn}
    Let $f \colon (M \to A) \to (N \to B)$ be a map of prelog rings. The $B$-module of logarithmic K\"ahler differentials of $f$ is defined as
    \[
    \Omega^1_f := \left(\Omega^1_{B/A} \oplus (\coker(M \to N)^{\rom{gp}} \otimes_\Z B)\right)/
    \left((d \beta(n),0) - (n \otimes \beta(n))\right),
    \]
    where $\beta \colon N \to B$ is the structure map. We often write $\Omega^1_{(A,M)/(B,N)}$ when the map $f$ is clear from the context.
    The monoid map $d\log \colon N \to \Omega^1_f$ is defined by $n \mapsto (0,n \otimes 1)$. Moreover, the derivation $d \colon B \to \Omega^1_{B/A}$ defines by composition an $A$-linear derivation $B \to \Omega^1_f$. The corresponding complex $\Omega^\bullet_f$ is called the logarithmic de Rham complex of $f$.
\end{defn}

Gabber's logarithmic cotangent complex is defined for any map of prelog rings, mimicking the construction of the usual cotangent complex using the canonical resolution instead of the standard simplicial resolution by polynomial algebras.

\begin{defn}
    Let $f \colon (M \to A) \to (N \to B)$ be a map of prelog rings. Let $P_\bullet(f) = (N_\bullet \to P_\bullet)$ be the canonical resolution of $f$. The \textit{logarithmic cotangent complex} of $f$ is defined as the complex of 
    $B$-modules
    \[
    L_f := \Omega^1_{P_\bullet(f)/(M \to A)} \otimes_{P_\bullet} B.
    \]
    The maps $d\log \colon N_n \to \Omega^1_{P_n(f)/(M \to A)}$ for each $n$ fit together to give a map
    \[
    d \log \colon N \simeq |N_\bullet| \to L_f
    \]
    in the homotopy category of simplicial monoids.

    Furthermore, the \textit{derived logarithmic de Rham complex} of $f$, denoted by $L\Omega^\bullet_f$, is defined as the total complex (using direct sums) associated to the double complex
    \[
    \Omega^\bullet_{P_\bullet(f)/(M \to A)}.
    \]
    The maps $N_n \to \Omega^\bullet_{P_n(f)/(M \to A)}[1]$ fit together to define a map
    \[
    d\log \colon N \simeq |N_\bullet| \to L\Omega^\bullet_f[1]
    \]
    in the homotopy category of simplicial monoids.
\end{defn}

\begin{remark}
In the case where $N = M$, the derived logarithmic de Rham complex agrees with the usual cotangent complex. This is, for a map of prelog rings $f \colon (N \to A) \to (N \to B)$, the natural map
\[
L_{B/A} \to L_f
\]
is an isomorphism.
\end{remark}

Let $f \colon (M \to A) \to (N \to B)$ be a map of prelog rings. 
The derived logarithmic de Rham complex admits two natural filtrations.

\begin{defn}
    \textit{The Hodge filtration} $F^\bullet$ is obtained from the stupid filtration on the logarithmic de Rham complex, more precisely,
    \[
    F^n L\Omega^\bullet_f := |\Omega^{\geq n}_{P_\bullet(f)/(M \to A)}|.
    \]
    The Hodge filtration is decreasing, separated, and exhaustive, and the graded pieces are given by
    \[
    \gr^n_F  L\Omega^\bullet_f \simeq \bigwedge^n L_f [-n],
    \]
    where $\bigwedge^n$ is the derived wedge product.
\end{defn}

\begin{defn}[{}{\cite[Proposition 6.9]{bhatt2012padicderivedrhamcohomology}}]
    \textit{The conjugate filtration} $F^{\con}_\bullet$ is obtained from the canonical filtration on the logarithmic de Rham complex, more precisely,
    \[
    F^{\con}_n := |\tau_{\leq n} \Omega^{\bullet}_{P_\bullet(f)/(M \to A)}|.
    \]
    The conjugate filtration is increasing, separated, exhaustive, and independent of the resolution $P_\bullet$. In particular, there exists a convergent spectral sequence, called the \textit{conjugate spectral sequence}, of the form 
    \footnote{We follow homological convention, i.e., $d_r$ is a map $E_r^{p,q} \to E_r^{-r,q+r-1}$.}
    \[
    E^{p,q}_1 \colon H_{p+q}(\gr^\con_p(L\Omega^\bullet_f)) \Rightarrow H_{p+q} L\Omega^\bullet_f
   .
    \]
\end{defn}

The following proposition shows that the derived de Rham complex is almost completely degenerate in characteristic zero. This explains why we need to work with the Hodge-completed de Rham complex in characteristic zero.

\begin{prop}[{}{\cite[Proposition 6.13]{bhatt2012padicderivedrhamcohomology}}]
Let  $f \colon (M \to A) \to (N \to B)$ be a map of prelog $\Q$-algebras. Then
\[
\gr_n^\con L\Omega^\bullet_f \simeq \bigwedge^n (\rom{Cone}(M^\gp \to N^\gp) \otimes_Z A)[-n].
\]
In particular, the derived de Rham complex vanishes if $N = M$.
\end{prop}

\subsection{Derived logarithmic de Rham cohomology in positive characteristic}
In this section, we recall some fundamental results from \cite[Section 7]{bhatt2012padicderivedrhamcohomology} about derived logarithmic de Rham cohomology, especially in positive characteristic.

\begin{defn}
    Let $(M \to A)$ be a prelog $\FF_p$-algebra. The Frobenius map 
    \[ \Frob_{(M \to A)} \colon (M \to A) \to (M \to A)
    \]
    is defined as multiplication by $p$ on $M$ and the usual Frobenius on $A$. For a map $f \colon (M \to A) \to (N \to B)$ of prelog $\FF_p$-algebras, the Frobenius twist $(N \to B)^{(1)}$ is defined as the pushout of $f$ along $\Frob_{(M \to A)}$. The relative Frobenius $\Frob_f$ is defined by the following diagram \footnote{We drop the log structure from the notation for clarity.}
    \[
    \begin{tikzcd}
    A \arrow[dr, phantom, "\ulcorner"] \arrow[r, "f"] \arrow[d, "\Frob_A"'] & B \ar[ddr, bend left, "\Frob_B"] \arrow[d] \\
    A \ar[drr, bend right, swap, "f"] \arrow[r, "f^{(1)}"] & B^{(1)} \arrow[rd, "\Frob_f"'] \\
    & & B.
    \end{tikzcd}
    \]
\end{defn}

The conjugate filtration is our main tool for studying the derived de Rham complex in positive characteristic. In particular, we have a derived version of the classical Cartier isomorphism.

\begin{theorem}[{}{\cite[Theorem 7.4]{bhatt2012padicderivedrhamcohomology}}]
Let  $f \colon (M \to A) \to (N \to B)$ be a map of prelog $\FF_p$-algebras. Then the conjugate filtration on $L\Omega^\bullet_f$ is $B^{(1)}$-linear, and has graded pieces given by
\[
\mathrm{Cartier}_n \colon \gr_n^\con(L\Omega^\bullet_f) \simeq \bigwedge^n L_{f^{(1)}}[-n].
\]
In particular, the conjugate spectral sequence takes the form
\[
E^{p,q}_1 : H_{2p+q}(\bigwedge^p L_{f^{(1)}}) \simeq H_{2p+q}(\Frob_A^* \bigwedge^p L_f) \Rightarrow H_{p+q}(L\Omega^\bullet_f).
\]
\end{theorem}

Next, we discuss some transitivity properties of the derived logarithmic de Rham complex in positive characteristic.

\begin{prop}[{}{\cite[Proposition 7.8]{bhatt2012padicderivedrhamcohomology}}]
Let  $f \colon (M \to A) \xrightarrow{f} (N \to B) \xrightarrow{g} (P \to C)$ be a composite of maps of prelog $\FF_p$-algebras. Then $L\Omega^\bullet_{g \circ f}$ admits an increasing, bounded below, separated, exhaustive filtration with graded pieces of the form
\[
L\Omega^\bullet_f \otimes_{\Frob^*_A B} \Frob^*_A (\bigwedge^n L_g [-n]),
\]
where the second factor on the right is the base change of $\bigwedge^n L_g [-n]$, viewed as a $B$-module along the map $\Frob_A \colon B \to \Frob_A^* B$.
\end{prop}

\begin{defn}
    A map $f$ of prelog $\FF_p$-algebras is called relatively perfect if $\rom{Frob}_f$ is an isomorphism. A map of prelog $\Z_p$-algebras is called relatively perfect modulo \(p\) if 
    $f\otimes_{\Z_p} \FF_p$ is relatively perfect.
\end{defn}

The following result is a direct consequence of the transitivity for the derived de Rham complex in characteristic $p$. It will be useful for some computations later in the Appendix.

\begin{lemma}\label{lem_rel_perf}
    Let  $f \colon (M \to A) \xrightarrow{f} (N \to B) \xrightarrow{g} (P \to C)$ be a composite of maps of prelog $\Z_p$-algebras. If $A \to B$ is relatively perfect modulo \(p\), then $(L_f)^\wedge_{(p)} \simeq 0$, and    
    \[
    (L\Omega^\bullet_{g\circ f})^\wedge_{(p)} \xrightarrow{\sim} (L\Omega^\bullet_{g})^\wedge_{(p)}.
    \]
\end{lemma}

\begin{proof}
    Combine the transitivity property above with the derived Nakayama lemma and \cite[Corollary 7.11]{bhatt2012padicderivedrhamcohomology}
\end{proof}

\begin{example}
    Two main examples of maps which are relatively perfect modulo \(p\) are \[W \to A_\infi,\] where
    $W = W(k)$ is the ring of Witt vectors of the residue field $k$ of a $p$-adic field. And
    \[
    (\Z[\Q_{\geq 0}],0) \to (\Z[\Q_{\geq 0}], [\Q_{\geq 0}]).
    \]
\end{example}

\noindent\textbf{Notation.} For a ring $A$, we write $A\langle x \rangle$ for the free pd-polynomial ring in one variable $x$ over $A$.

\begin{theorem}[Bhatt, {{\cite[Theorem 3.23]{SzaBei18}}}]\label{thm_dr_surjective}
    Assume that $A \to B$ is a surjective morphism of flat $\Z/p^n\Z$-algebras with kernel $I = (f)$ generated by a non-zero divisor $f \in A$. Then
    \[
    L\Omega^\bullet_{B/A} \cong A\langle x \rangle /(x - f).
    \]
    Moreover, the Hodge filtration on the left-hand side corresponds to the divided power filtration on the right.
\end{theorem}

We end this section by stating the comparison theorem between derived log de Rham cohomology and log crystalline cohomology.

\begin{theorem}[{{\cite[Theorem 7.22]{bhatt2012padicderivedrhamcohomology}}}]\label{thm_crystalline_comp}
    Let $f \colon (M \to A) \to (N \to B)$ be a morphism of prelog $\Z/p^n$-algebras. There is a natural map of filtered $E_\infty$-algebras
    \[
    \rom{Comp}_f \colon L\Omega^\bullet_{(B,N)/(A,M)} \to \RS_\cris(f,\clg{O}_\cris),
    \]
    where the right-hand side denotes the log crystalline cohomology of the morphism $f$.
    The map is an isomorphism if $f$ is a $G$-lci map.
\end{theorem}

\begin{defn}
    A map $f \colon (M \to A) \to (N \to B)$ of prelog $\Z/p^n$-algebras is called \textit{$G$-lci} if both 
    $A$ and $B$ are $\Z/p^n$-flat, and $f$ can be factored as
    \[
    (M \to A) \xrightarrow{a} (P \to F) \xrightarrow{b} (N \to B),
    \]
    where $a$ is ind-log-smooth and of Cartier type modulo $p$
    \footnote{This means that $M \to N$ is an integral map of integral monoids, and $\Frob_f \colon M^{(1)} \to M$ is an exact morphism of monoids.}, and $b$ is a strict epimorphism.
\end{defn}

\subsection{\(p\)-adic period Rings}\label{appendix_period_rings}
The goal of this section is to give a concise summary of the construction and basic properties of Fontaine's \(p\)-adic period rings. 
We follow Bhatt's treatment \cite{bhatt2012padicderivedrhamcohomology}, based on Beilinson \cite{Beil_de_Rham},
and define the \(p\)-adic period rings in terms of the derived de Rham complex. Furthermore, we indicate how these constructions are related to the classical constructions of Fontaine \cite{fontaine_period_ring}. 

Fix a $p$-adic field $K$ with ring of integers $O_K$, uniformizer $\varpi$ and residue field $k$. Furthermore, we denote the ring of Witt vectors of $k$ by $W$. Moreover, we write
$A_\infi$ for $W(O_{\C_p}^\flat)$, where  $O_{\C_p}^\flat$ denotes the tilt of the perfectoid ring $O_{\C_p}$. Denote the maximal ideal of $O_{\C_p}^\flat$ by $\frk{m}^\flat$ and fix a pseudo-uniformizer $\varpi^\flat = (\varpi,\varpi^{1/p},\varpi^{1/p^2},\dots)$.
Recall that the map
\[
\theta \colon A_\infi \to O_{\C_p}
\]
lifting the first projection $O_{\C_p}^\flat \to O_{\C_p}/p$
is surjective with principal kernel.

\begin{defn}
    We define the following period rings \footnote{Recall that $C^\wedge_{(p)}$ denotes the derived \(p\)-completion.}
    \begin{align*}
    A_{\cris} := (L\Omega^\bullet_{(O_{\ol{K}},\can)/W})^\wedge_{(p)}, \quad 
    A_\stb := (L\Omega^\bullet_{(O_{\ol{K}},\can)/(W[x],x)})^\wedge_{(p)}, \quad
    A_\dr := (L\hat{\Omega}^\bullet_{(O_{\ol{K}},\can)/(O_K,\can)})^\wedge_{(p)},    
    \end{align*}
    where the map $W[x] \to O_{\ol{K}}$ is induced by $x \mapsto \varpi$.
\end{defn}

Note that all rings defined above are equipped with a $G_K$-action and a Hodge filtration. Moreover, the crystalline and semistable rings are equipped with a natural Frobenius action via the comparison with crystalline cohomology.
From the functoriality of the derived de Rham complex, we get natural maps
$A_\cris \to A_\stb \to A_\dr$. We will sometimes use the fact that we can equivalently define the above rings with $O_{\C_p}$ in place of $O_{\ol{K}}$.

We start by relating these period rings to the classical period rings were defined by Fontaine as follows.

\begin{defn}
    The ring $B_\dr^+$ is defined as
    \[
    B_\dr^+ := \varprojlim_{n} \left( A_\infi \left[ \frac{1}{p} \right]/(\ker(\theta))^n  \right).
    \]
    In particular, $B_\dr^+$ is a complete valuation ring with fraction field denoted by $B_\dr$.
    The ring $B_\cris^+$ is defined as
    \[
    B_\cris^+ := \left(\varprojlim_{n}  D(A_\infi, \ker \theta)/p^n \right) \left[ \frac{1}{p} \right], 
    \]
    where $D(A_\infi, \ker \theta)$ denotes the divided power envelope of $\ker \theta$ 
    in $A_\infi$.
\end{defn}

We remark that $B_\dr^+$ admits a topology which is finer than the valuation topology, and there exists a unique continuous map $B_\cris^+ \to B_\dr^+$ extending the map $A_\infi \to B_\dr$. Under the identifications below, this map is induced by the natural map $A_\cris \to A_\dr$. Although we exclusively use the de Rham construction of the period rings, we collect the fact that the de Rham construction agrees with the classical constructions by Fontaine.

\begin{theorem}[{{\cite[Proposition 9.8]{bhatt2012padicderivedrhamcohomology}}}]\label{thm_period_ring_comp}
We have the following natural identifications:
\begin{enumerate}
    \item The ring $A_\cris$ is isomorphic to the \(p\)-adic completion of
    $D(A_\infi, \ker(\theta))$. Under this identification, the Hodge filtration corresponds to the divided power filtration.
    \item The ring $A_\dr$ is isomorphic to the \(p\)-adic completion of the completion of $D(A_\infi,\ker(\theta))$ with respect to its divided power filtration. Under this identification, the Hodge filtration corresponds to the divided power filtration. In particular, $A_\dr\left[ \frac{1}{p} \right]$ identifies with $B_\dr^+$.
\end{enumerate}
\end{theorem}

The theorem follows from Theorem \ref{thm_dr_surjective} and Lemma \ref{lem_rel_perf} applied to $W \to A_\infi \to O_{\C_p}$ and the following computation using that $W \to A_\infi$ is relatively perfect modulo $p$.

\begin{lemma} Let $K$ be any finite extension of $\Q_p$.
    The natural map
    \[
     L_{(O_{\ol{K}},\can)/W} \to L_{O_{\ol{K}}/W}
    \]
    is a quasi-isomorphism after \(p\)-completion. Moreover, the natural maps
    \[
    L_{(O_{\ol{K}},\can)/W} \to L_{(O_{\ol{K}},\can)/O_K} \to L_{(O_{\ol{K}},\can)/(O_K,\can)}
    \]
    become isomorphisms after $- \hat{\otimes} \Q_p$.
\end{lemma}

\begin{proof}
    We start with the second assertion. The transitivity triangle for the short exact sequence of prelog rings
    \[
    O_K \to (O_K, \can) \to (O_{\ol{K}},\can),
    \]
    reads as follows
    \[
    L_{(O_K,\can)/O_K} \to L_{(O_{\ol{K}},\can)/O_K} \to L_{(O_{\ol{K}},\can)/(O_K,\can)}
    \xrightarrow{+1}.
    \]
    Note that all terms are concentrated in degree zero and $L_{(O_K,\can),O_K}$ is given by $O_K/\varpi \cdot \dlog \varpi$ and hence vanishes after $\hat{\otimes} \Q_p$. Similarly,  looking at the transitivity triangle for 
    \[
    W \to O_K \to (O_{\ol{K}},\can)
    \]
    gives the other quasi-isomorphism.
    
    For the first assertion, fix a compatible system of rational roots of $\varpi^\flat$ and consequently a compatible system of rational powers $[\varpi^\flat]^{\Q_{\geq 0}} \in A_\infi$. Now, consider the following diagram of prelog rings
    \[
    \begin{tikzcd}
        W \ar[r, "a"] & A_\infi \ar[r, "b"]  \ar[d, "\theta"] & (A_\infi, [\varpi]^{\Q_{\geq 0}}) \ar[d, "d"]\\
        &O_{\C_p} \ar[r, "e"] &(O_{\C_p},\can)
    \end{tikzcd}
    \]
    Since $a$ is relatively perfect modulo \(p\), we know that
    \[
    (L_{O_{\ol{K}}/W})^\wedge_{(p)} \simeq (L_{O_{\C_p}/A_\infi})^\wedge_{(p)},
    \]
    and hence it suffices to show that the map
    \[
    (L_{\theta})^\wedge_{(p)}  \to (L_{c \circ b})^\wedge_{(p)} 
    \]
    is a quasi-isomorphism. Since the square above is a pushout square after passing to log rings, by K\"unneth, it suffices to show that $(L_{b})^\wedge_{(p)} \simeq 0$. This follows by base change from the fact that
    \[
    (\Z[\Q_{\geq 0}],0) \to (\Z[\Q_{\geq 0}], [\Q_{\geq 0}])
    \]-
    is relatively perfect modulo \(p\), see \cite[Cor. 7.13]{bhatt2012padicderivedrhamcohomology}.
\end{proof}

\begin{remark}
    One can show by a direct computation that $L_{(O_{\ol{K}},\can)/W} \to L_{O_{\ol{K}}/W}$ is an equivalence before \(p\)-completion \cite[Lemma 6.7]{SzaBei18}.
\end{remark}

Two immediate consequences of the above lemma are:
\begin{enumerate}[i)]
    \item That we could equivalently define $A_\cris$ as the \(p\)-completed derived de Rham complex of $O_K \to (O_{\ol{K}},\can)$.
    \item That the natural map $A_\cris \to A_\dr$ becomes an isomorphism after Hodge completion and inverting \(p\).
\end{enumerate} 

\begin{lemma}\label{lem_cotangent_O_K}
    Let $K$ be a $p$-adic field.
    \begin{enumerate}
        \item The cotangent complex $L_{O_{\ol{K}}/O_K}$ is concentrated in degree zero, where it is given by $\Omega^1_{O_{\ol{K}}/O_K}$.
        \item The map $d \colon O_{\ol{K}} \to \Omega^1_{O_{\ol{K}}/O_K}$ is surjective. Moreover, $\Omega^1_{O_{\ol{K}}/O_K}$ is \(p\)-primary torsion and \(p\)-divisible.
        \item We have
        \[
            (L_{O_{\ol{K}}/O_K})^\wedge_{(p)} \simeq T_p(\Omega^1_{O_{\ol{K}}/O_K})[1].
        \]
    \end{enumerate}
\end{lemma}

\subsection{The logarithm map}\label{sect_log_map}
In this section, we recall the classical construction of the logarithm map. It is a $G_K$-equivariant monoid map
\[
\log \colon O_{\C_p}^\flat \setminus \{0\} \to B_\dr^+,
\]
which is trivial on $\ol{k}^\times$ and extends Fontaine's logarithm $\log \colon \Z_p(1) \to B_\dr^+$.
The construction depends on a (fixed) choice of a pseudo-uniformizer $\varpi^\flat = (\varpi,\varpi_1, \dots) \in O_{\C_p}^\flat$.
Moreover, we will give an alternative construction using the derived de Rham construction of $B_\dr$.

\begin{lemma}[{{\cite[Proposition 6.5]{FontaineOuyangGPAdic}}}]
    The kernel of the map
    \[
    \ol{\theta} \colon A_\cris \to O_{\C_p}/\varpi
    \]
    is a divided power ideal in $A_\cris$.
\end{lemma}

\begin{construction}[Fontaine]\label{constr_classical_log}
Let $x \in 1 + (\varpi^\flat) \subset O_{\C_p}^\flat$, then $[x]-1 \in \ker(\ol{\theta}) \in A_\cris$ and therefore
\[
\frac{([x]-1)^n}{n} = (n-1)! \gamma_n([x]-1) \in A_\cris.
\]
We define the logarithm of $x$ as
\[
\log [x] := \sum_{n \geq 1} (-1)^{n+1} \frac{([x]-1)^n}{n}.
\]
The sum converges in the \(p\)-adic topology on $A_\cris$.
For example, if $x = \epsilon$ is a compatible system of \(p\)'th roots of unity, then $\log [\epsilon] = t$
is by definition Fontaine's element. 

Next, we extend the logarithm to $x \in 1 + \frk{m}^\flat$
by noting that for every such $x$ there exists $m \geq 0$ such that $x^{p^m} \in 1 + (\varpi^\flat)$ and setting
\[
\log [x] = \frac{1}{p^m} \log [x^{p^m}].
\]
Next, we can directly extend the logarithm to $(O_{\C_p}^\flat)^\times$ using the decomposition
\[
(O_{\C_p}^\flat)^\times = \ol{k}^\times \times (1+\frk{m}^\flat)
\]
and imposing the condition that $\log$ is trivial on $\ol{k}^\times$. 

Finally, to extend the logarithm map to a monoid map $O_{\C_p}^\flat \setminus {0} \to B_\dr^+$, we only need to give its value at 
$\varpi^\flat$. 
Indeed, for any $x \in O_{\C_p}^\flat \setminus {0}$ there exists $s,r \geq 0$ and $y \in (O_{\C_p}^\flat)^\times$ such that
\[
x^s = y (\varpi^\flat)^{-r}
\]
and we set
\[
\log [x] = \frac{1}{s}(r \log [\varpi^\flat] + \log[y]).
\]
So far, we were working within $A_\cris$, but to define $\log [\varpi^\flat]$ we need to work in $B_\dr^+$. We note that 
\[
\frac{[\varpi^\flat]}{\varpi}-1 \in \ker( A_\infi [\frac{1}{p}] \to \C_p),
\]
and therefore we may define
\[
\log [\varpi^\flat] := \sum_{n \geq 1}(-1)^{n+1}\frac{(\frac{[\varpi^\flat]}{\varpi}-1)^n}{n}
= \sum_{n \geq 1}(-1)^n \frac{\xi^n}{n \varpi^{n}},
\]
with $\xi = [\varpi^\flat] - \varpi$ (which is also generator of $\ker(\theta)$).
The above series does not converge \(p\)-adically, but it does in the $\ker(\theta)$-adic topology! We verify that the resulting map is $G_K$-equivariant
\[
\sigma \log [\varpi^\flat] = \log([\varpi^\flat] [\epsilon^{\eta(\sigma)}]) 
 = \log[\varpi^\flat] + \eta(\sigma) \log[\epsilon].
\]

\end{construction}

The rest of this subsection is devoted to an alternative construction of the logarithm map in terms of the derived logarithmic de Rham complex.

\begin{construction}\label{constr_deRham_log}
    Recall that the derived de Rham complex comes with a monoid map
    \[
    d\log \colon O_{\ol{K}} \setminus \{0\} \to L\Omega^\bullet_{(O_{\ol{K}},\can)/W}[1]
    \]
    After derived $p$-completion, $d\log$ gives rise to a $G_K$-equivariant map
    \[
    \lambda \colon ((O_{\ol{K}} \setminus \{0\})^{\rom{gp}})^\wedge_{(p)} \cong \Z_p(1) \to A_\cris.
    \]
\end{construction}

\begin{prop}\label{prop_font_log}
    The map $\lambda \colon \Z_p(1) \to B_\cris^+$ coincides with Fontaine's logarithm map under the above identifications.
\end{prop}

\begin{proof}
    Let $\epsilon = (1,\zeta_1,\zeta_2,\dots)$ be a generator of $\Z_p(1)$.
    This is a continuation of the computation of $\langle d\log x, \epsilon \rangle_\dr$ from Example \ref{ex_tate_curve}. We have already reduced to proving the claim for $\gr_F^1(\lambda)$. From the transitivity triangle for $W \to A_\infi \to O_{\C_p}$ we get a short exact sequence
    \[
    0 \to I/I^2 \to \Omega^1_{A_\infi/W} \otimes_{A_\infi} O_{\C_p} \to \Omega^1_{O_{\C_p}/W} \to 0, 
    \]
    where $I = \ker (\theta)$.
    The Snake Lemma gives maps
    \[
    c_n \colon \Omega^1_{O_{\C_p}/W}[p^n] \to I/(I^2,p^n) 
    \]
    for all $n \geq 0$.
    Since $W \to A_\infi$ is relatively perfect modulo \(p\), the cotangent complex $(L_{A_\infi/W})^\wedge_{(p)}$ vanishes, and therefore the maps $c_n$ induce an isomorphism
    \[
    c \colon (L_{O_{\C_p}/W})^\wedge_{(p)}[-1] \simeq T_p(\Omega^1_{(O_{\C_p},\can)/W}) \xrightarrow{\sim}  \ker \theta/ (\ker(\theta)^2 \cong \gr_F^1 A_\cris.
    \]
    We have to show that
    \[
    c((d\log \zeta_n)_n) = [\epsilon] - 1.
    \]
    The Snake Lemma gives a recipe to compute $c_n(d \log \zeta_n)$. First, lift $d\log \zeta_n$ to $\omega \in \Omega^1_{A_\infi/W} \otimes_{A_\infi} O_{\C_p}$, then $p^n \omega 
    \in \ker (\Omega^1_{A_\infi/W}  \otimes_{A_\infi} O_{\C_p} \to \Omega^1_{O_{\C_p}/W})$, hence there exists $f \in I$ such that $p^n \omega = df$, and $c_n(d \log \zeta_n)$ is given by the class of $f$ in $I/(I^2,p^n)$.
    One easily checks that $f = [\epsilon]-1$ works using that $\zeta_n = \theta([\epsilon^{p^{-n}}])$.
\end{proof}

\begin{remark}
     Bhatt gives the following general formula for $c_n$, which can be verified using the same method as above. Write $f \in \ker(\theta)$ as 
    \[
    f = \sum_{k = 0}^{n-1} [f_k] p^k + p^n g
    \]
    for some $g \in A_\cris$ and $f_k \in O_{\C_p}^\flat$. Using that $O_{\C_p}^\flat$ is perfect and that $[-]$ is multiplicative we can rewrite $f$ as
    \[
    f = \sum_{k = 0}^{n-1} [f_k^{p^{-n+k}}]^{p^{n-k}} p^k + p^n g.
    \]
    Then, 
    \[
    c_n(f) = \sum_{k=0}^{n-1} (f_k^{(n-k)})^{p^{n-k}-1} df_k^{(n-k)} + d\theta(g) \stackrel{(*)}{=} 
    \sum_{k=0}^{n-1} f_k^{(0)} d\log f_k^{(n-k)} + d\theta(g) ,
    \]
    where $f_k = (f_k^{(0)}, f_k^{(1)}, \dots ) \in \lim_{x \mapsto x^p} O_{\C_p}$ and $(*)$ holds if $f_k \in O_{\C_p}^\times$ for all $k$. 
\end{remark}

We continue by discussing Kato's version of the semistable period ring. Its definition
also depends on a choice of (pseudo-) uniformizers $\varpi \in O_K$ and $\varpi^\flat \in O_{\C_p}^\flat$.
One can show that the resulting ring is independent of the last choice up to a transitive system of isomorphisms \cite[Section 5]{breuil_crystalline_coh}. 

\begin{construction}
    The ring $\hat{A_\stb}$ is defined as $A_\cris\langle X\rangle$, the free \(p\)-adically complete PD-polynomial ring in one variable $X$.
    We extend the $G_K$-action on $A_\cris$ to $\hat{A_\stb}$ by
    \[
    \sigma(X+1) = \frac{[\varpi^\flat]}{\sigma([\varpi^\flat])} (X+1).
    \]
    We equip $\hat{A_\stb}$ with the minimal pd-multiplicative Hodge filtration $F_H$ which extends the one on $A_\cris$ and satisfies $X \in F^1_H(\hat{A_\stb})$.
    We define $\phi \colon \hat{A_\stb} \to \hat{A_\stb}$ to be  the unique extension of $\phi$ from $A_\cris$ satisfying $\phi(X+1) = (X+1)^p$.
    Finally, we define $N \colon \hat{A_\stb} \to \hat{A_\stb}$ to be the continuous $A_\cris$-linear pd-derivation $N(X) = X+1$.
\end{construction}

\begin{prop}[{}{\cite[Proposition 9.21]{bhatt2012padicderivedrhamcohomology}}]
    Let $C$ be the prelog ring \[(A_\infi [y,y^{-1}], [\varpi^\flat]^{\Q_{\geq 0}} \cdot y^\Z)\] with $G_K$-action defined by
    \[
    \sigma(y) = \frac{[\varpi^\flat]}{\sigma([\varpi^\flat])} y.
    \]
    View $(O_{\C_p},\can)$ as a $C$-algebra via the usual map and $y \mapsto 1$. Then there is a natural isomorphism
    \[
    \hat{A_\stb} \cong (L\Omega^\bullet_{(O_{\C_p},\can)/C})^\wedge_{(p)}.
    \]
\end{prop}

The reason we introduced the ring $\hat{A_\stb}$ is that the $G_K$-extension of $A_\cris$ constructed below vanishes in $\hat{A}_\stb$ which then leads to an integral comparison theorem over $\hat{A}_\stb$ between log crystalline cohomology and \'etale cohomology, see \cite[Chapter 10]{bhatt2012padicderivedrhamcohomology}.

\begin{construction}
    Restricting the $d\log$ map to $\varpi^\N \subset O_{\ol{K}} \setminus \{0\}$ gives a monoid map
    \[
    \varpi^\N \to L\Omega^\bullet_{(O_{\ol{K}},\can)/W}[1]
    \]
    which induces after \(p\)-completion on the target and group completion on the source a  $G_K$-equivariant map
    \[
    \stb_\varpi \colon \Z \cong (\varpi^\N)^{\rom{gp}} \to A_\cris[1],
    \]
    representing a $1$-cocycle $\rom{cl}(\stb_\varpi) \in H^1(G_K,A_\cris)$. 
    The proof of Proposition \ref{prop_obstr_R_stb} shows that this class corresponds to the map
    \[
    \sigma \mapsto \log\big(\frac{\sigma [\varpi^\flat]}{[\varpi^\flat]}\big) = \log ([\sigma \varpi^\flat/\varpi^\flat]) = \log[\epsilon^{\eta(\sigma)}] = \eta(\sigma)t,
    \]
    for $\eta$ as in Example \ref{ex_tate_curve}. It follows immediately that $\rom{cl}(\stb_\varpi)$ vanishes in $B_\dr^+$ because
    \[
    \sigma \log [\varpi^\flat] - \log [\varpi^\flat] = \eta(\sigma) t.
    \]
    Similarly, $\rom{cl}(\stb_\varpi)$ vanishes after composition with $A_\cris \to \hat{A}_\stb$, because by construction
    \[
    \log X+1 := \sum_{n \geq 1} (-1)^{n+1}\frac{X^n}{n} \in \hat{A}_\stb
    \]
    converges in the \(p\)-adic topology and satisfies
    \[
    \sigma \log (X+1) - \log (X+1) =  \log\big(\frac{\sigma [\varpi^\flat]}{[\varpi^\flat]}\big) + \log (X+1) - \log (X+1) = \eta(\sigma) t.
    \]
\end{construction}

The remainder of this section (except for Corollary \ref{cor_graded_log}) is not directly relevant to the thesis. However, it gives a more conceptual reason for defining $A_\stb$.

\begin{prop}
    The class $\rom{cl}(\stb_\varpi)$ is the obstruction to $G_K$-equivariantly extending Fontaine's logarithm 
    \[
    \log \colon \Z_p(1) \to A_\cris\
    \]
    to a logarithm map
    \[
    \log \colon (\C_p^\flat)^\times \to A_\cris.
    \]
\end{prop}

By Proposition \ref{prop_obstr_R_stb}, there exists an extension of the logarithm map to a map with values in $A_\stb$.
Moreover, the computation of the following corollary shows that it coincides with Fontaine's construction after composition with the map $A_\stb \to B_\dr$ induced by $x \mapsto \varpi$.

We can finally finish the computation from Example \ref{ex_tate_curve}.

\begin{cor}\label{cor_graded_log}
    With the notations from above. The image of $\log[\varpi^\flat]$ in $gr^1_F B_\dr^+$ corresponds to 
    $(d\log \varpi_n)_n \in V_p(\Omega^1_{O_{\C_p}/W})$
    under the isomorphism
    \[
    c \colon  V_p(\Omega^1_{(O_{\C_p},\can)/(W,\can)}) \to  \ker \theta/(\ker \theta)^2[1/p] \cong gr^1_F B_\dr^+.
    \]
    Moreover, 
    the image of $A_\stb[\frac{1}{p}]$ inside $B_\dr^+$ under the natural map contains the classical semistable period ring $B_\stb^+$.
\end{cor}
\begin{proof}
    Since $\gr_F^1 B_\dr^+ \cong \C_p(1)$, it suffices to show the claim for $p d\log \varpi_n$. In this case, the claim follows from a very similar computation as in Proposition \ref{prop_font_log}.
\end{proof}

The class $\rom{cl}(\stb)$ also has an interpretation in terms of the so-called Faltings-Breul ring $R_{O_K}$. It is defined as
\[
R_{O_K} := (L\Omega_{(O_K,\can)/(W[x],x)}^\bullet)^\wedge_{(p)} \simeq W[x]\langle E(x)\rangle ^\wedge_{(p)},
\]
where $E(x)$ is the minimal polynomial of $\varpi$ over $W$.

\begin{prop}\label{prop_obstr_R_stb}
With notations as above.
\begin{enumerate}
    \item The class $\rom{cl}(\stb_\varpi)$ is the obstruction to factoring the natural map 
    \begin{equation}\label{eq_bf}
        (L\Omega_{(O_K,\can)/(W)}^\bullet)^\wedge_{(p)} \to A_\cris
    \end{equation}
    $G_K$-equivariantly through the projection map 
    \[
        L\Omega_{(O_K,\can)/W}^\bullet)^\wedge_{(p)} \to R_{O_K}
    \]
    \item Let $K_\infty = K(\mu_{p^\infty})$. The class $\rom{cl}(\stb_\varpi)$ vanishes in $H^1(G_{K_\infty},A_\cris)$.
    \item The class $\rom{cl}(\stb_\varpi)$ vanishes after composition with the natural map $A_\cris \to \hat{A_\stb}$. In particular, we can view $\hat{A_\stb}$ as a $R_{O_K}$-algebra.
\end{enumerate}
\end{prop}
\begin{proof}
Recall that the derived de Rham complex can be computed using any free polynomial resolution.
\begin{enumerate}
    \item Computing $(L\Omega_{(O_K,\can)/(W)}^\bullet)^\wedge_{(p)}$ as the derived de Rham complex of the composition
    \[
    W \to (W[x],x) \to (O_K,\can)
    \]
    yields
    \[
    (L\Omega_{(O_K,\can)/(W)}^\bullet)^\wedge_{(p)} \simeq R_{O_K} \hat{\otimes} 
    \left[ W[x] \xrightarrow{d} W[x] d\log x\right ]
    \simeq \left[R_{O_K} \xrightarrow{d} R_{O_K} d\log x\right]
    \]
    using that $(W[x],x)$ is a free polynomial prelog ring over $W$. Since $d \log$ is functorial, $\stb_\varpi(\varpi)$ is given by the image of $x$ under the composition 
    \[
    x^\N \xrightarrow{d\log} (L\Omega_{(W[x],x)/W})^\wedge_{(p)}[1] \to (L\Omega_{(O_K,\can)/W}^\bullet)^\wedge_{(p)}[1],
    \]
    in other words, $\stb_\varpi(\varpi)$ is given by the image of $d\log(x)$ under (\ref{eq_bf}). The claim now follows a short diagram chase.
    \item Choose a compatible system of \(p\)'th-power roots $\varpi^\flat = (\varpi, \varpi_1,\varpi_2,\dots)$ of $\varpi$. Then we have
    \[
    p^n d\log \varpi_n = d\log \varpi
    \]
    which implies that $\stb_\varpi$ is null-homotopic as a map of $\Z$-modules.
    Moreover, $\stb_\varpi \simeq 0$ is null-homotopic in the derived category of 
    $G_{K_\infty}$-modules because $G_{K_\infty}$ acts trivially on $\varpi^\flat$. We denote the corresponding null-homotopy by $H_{\omega^\flat}$.
    \item 
    The null-homotopies from the previous step give rise to a map 
     \begin{align*}
     H \colon G_K &\to A_\cris \\
     \sigma &\mapsto H_{\sigma \varpi^\flat} - H_{\varpi^\flat}:= \sigma H_{\varpi^\flat} -H_{\varpi^\flat}.
     \end{align*}
     In fact, $H$ defines a $1$-cochain which
     tautologically represents $\stb_\varpi \in H^1(G_K,A_\cris)$. To verify this, identify $A_\cris$ with the limit $\varprojlim_n L\Omega_{(O_{\ol{K}}/\can)/W}\otimes \Z/p^n$ and compute that
     \[
     H(\sigma) = \varprojlim_n \dlog \zeta_n^{\eta(\sigma)} = \eta(\sigma) \log [\epsilon] = \log (\frac{[\sigma \varpi^\flat]}{[\varpi^\flat]}) \in A_\cris,
     \]
     for $\eta$ as before. The claim now follows from Construction \ref{constr_classical_log}, because in $\hat{A}_\stb$ we have
     \[
     \eta(\sigma) t = \sigma \log(X+1) - \log(X+1).
     \]
\end{enumerate}    
\end{proof}

\begin{remark}
Step 1. in the proof also shows that $\rom{cl}(\stb_\varpi)$ vanishes in $H^1(G_K,A_\stb)$.  
\end{remark}

\section*{Declaration of generative AI and AI-assisted technologies in the manuscript preparation process}

During the preparation of this work the author used ChatGPT in order to improve the languague and clarity of the text. After using this tool/service, the author reviewed and edited the content as needed and takes full responsibility for the content of the published article.

\bibliographystyle{elsarticle-num}
\bibliography{refs}

\end{document}